\pdfoutput=1
\documentclass[a4paper,english]{jnsao}
\usepackage[utf8]{inputenc}
\usepackage{csquotes}
\usepackage{babel}
\usepackage[T1]{fontenc}
\usepackage{tikz}
\usepackage{pgfplots}
\pgfplotsset{compat=newest}
\pgfplotsset{plot coordinates/math parser=false}
\usetikzlibrary{plotmarks}
\newlength\figureheight
\newlength\figurewidth
\usepackage{grffile}
\usepackage{float}
\usepackage{enumitem}
\usepackage{mathtools}
\usepackage[nameinlink,capitalize]{cleveref}

\newif\ifprint
\printfalse

\usepackage{tikz,pgfplots,pgfplotstable}
\usepgfplotslibrary{colorbrewer,fillbetween}
\pgfplotsset{compat=newest}
\pgfplotsset{plot coordinates/math parser=false}
\usetikzlibrary{plotmarks}
\pgfplotsset{
    table/search path={img/nlpdhgm_block/},
}
\pgfplotsset{every tick label/.append style={font=\scriptsize}}

\floatstyle{ruled}
\newfloat{Algorithm}{t!}{loa}[section]
\crefname{Algorithm}{Algorithm}{Algorithms}

\theoremstyle{definition}
\newtheorem{assumption}[definition]{Assumption}
\crefname{assumption}{Assumption}{Assumptions}

\graphicspath{{./img/nlpdhgm_block/}}
\newcommand{\term}{\emph}

\newcommand{\field}[1]{\mathbb{#1}}
\newcommand{\N}{\mathbb{N}}

\newcommand{\R}{\field{R}}
\newcommand{\extR}{\overline \R}
\newcommand{\B}{B}

\newcommand{\norm}[1]{\|#1\|}

\newcommand{\abs}[1]{|#1|}

\newcommand{\inv}[1]{#1^{-1}}
\newcommand{\grad}{\nabla}
\newcommand{\freevar}{\,\boldsymbol\cdot\,}

\newcommand{\Union}\bigcup
\newcommand{\Isect}\bigcap
\newcommand{\union}\cup
\newcommand{\isect}\cap
\newcommand{\bigunion}\bigcup
\newcommand{\bigisect}\bigcap
\newcommand{\powerset}{\mathcal{P}}
\newcommand{\defeq}{:=}

\newcommand{\downto}{\searrow}
\newcommand{\upto}{\nearrow}
\newcommand{\subdiff}{\partial}

\newcommand{\MIN}[1]{{\underline {#1}}}

\DeclareMathOperator{\Dom}{dom}

\makeatletter
\def \uminus@sym{\setbox0=\hbox{$\cup$}\rlap{\hbox 
        to\wd0{\hss\raise0.5ex\hbox{$\scriptscriptstyle{-}$}\hss}}\box0}
    \def \uminus    {\mathrel{\uminus@sym}}
\makeatother

\newcommand{\mathvar}[1]{\textup{#1}}

\newcommand{\iprod}[2]{\langle #1,#2\rangle}
\newcommand{\adaptiprod}[2]{\left\langle #1,#2\right\rangle}

\makeatletter
\def \weaktostar@sym{\setbox0=\hbox{$\rightharpoonup$}\rlap{\hbox 
        to\wd0{\hss\raise1ex\hbox{$\scriptscriptstyle{*\,}$}\hss}}\box0}
    \def \weaktostar    {\mathrel{\weaktostar@sym}}
\makeatother

\renewcommand{\B}{\vmathbb{B}}

\def\tilde{\widetilde}

\def\NL{\mathvar{NL}}
\def\LIN{\mathvar{L}}
\def\Ynl{Y_{\NL}}
\def\Ylin{Y_{\LIN}}

\def\Pnl{P_{\NL}}

\newcommand{\setto}{\rightrightarrows}

\def\extR{\overline \R}
\def\linear{\vmathbb{L}}

\newcommand{\linearLArrow}[1][]{\linear_{\triangleleft\ifx&#1&\else,\,#1\fi}}
\newcommand{\linearLArrowSpecial}[1][]{\linear^{\star}_{\triangleleft\ifx&#1&\else,\,#1\fi}}

\def\realopt#1{\widehat #1}
\def\this#1{#1^i}
\def\nexxt#1{#1^{i+1}}
\def\overnext#1{\bar #1^{i+1}}

\def\realoptu{{\realopt{u}}}

\def\realoptx{{\realopt{x}}}

\def\realopty{{\realopt{y}}}

\def\nextu{\nexxt{u}}
\def\nextx{\nexxt{x}}

\def\nexty{\nexxt{y}}

\def\thisu{\this{u}}
\def\thisx{\this{x}}

\def\thisy{\this{y}}

\def\overnextx{\overnext{x}}

\def\E{\mathbb{E}}
\def\P{\mathbb{P}}

\def\Tau{T}
\def\TauTest{\Phi}
\def\tauTest{\phi}
\def\SigmaTest{\Psi}
\def\sigmaTest{\psi}

\def\Neigh{\mathcal{V}}

\def\SAlg{\mathcal{O}}

\def\iset#1{\mathring{#1}}
\def\dset#1{\breve{#1}}
\def\Space{U}

\newcommand{\Test}{Z}
\newcommand{\Precond}{M}
\newcommand{\Step}{W}
\newcommand{\Random}{\mathcal{R}}

\def\MIN{\underline}

\DeclareFontFamily{U}{mathx}{\hyphenchar\font45}
\DeclareFontShape{U}{mathx}{m}{n}{<-> mathx10}{}
\DeclareSymbolFont{mathx}{U}{mathx}{m}{n}
\DeclareMathAccent{\widebar}{0}{mathx}{"73}

\def\kgrad#1{\grad K(#1)}
\def\kgradconj#1{[\grad K(#1)]^*}

\def\dir#1{\Delta #1}

\def\realoptw{\realopt{w}}
\def\neighu{\mathcal{U}}

\def\neighx{\mathcal{X}}
\def\neighy{\mathcal{Y}}

\def\bar{\widebar}

\newcommand{\localRhoX}[1][]{\if\relax\detokenize{#1}\relax r_x\else r_{x,#1}\fi}
\newcommand{\localRhoY}[1][]{\if\relax\detokenize{#1}\relax r_y\else r_{y,#1}\fi}

\newcommand{\prox}{\mathrm{prox}}

\DeclareMathOperator{\Sym}{Sym}
\def\symD{\mathcal{E}}

\def\adaptNorm#1{\left\|#1\right\|}
\def\BiggNorm#1{\Biggl\|#1\Biggr\|}

\renewrobustcmd{\downto}{{{\mathchoice%
            {\rotatebox[origin=c]{-20}{$\to$}}
            {\rotatebox[origin=c]{-20}{$\to$}}
            {\rotatebox[origin=c]{-20}{\scalebox{0.75}{$\to$}}}
            {\rotatebox[origin=c]{-20}{\scalebox{0.6}{$\to$}}}
}}}

\renewrobustcmd{\upto}{{{\mathchoice%
            {\rotatebox[origin=c]{20}{$\to$}}
            {\rotatebox[origin=c]{20}{$\to$}}
            {\rotatebox[origin=c]{20}{\scalebox{0.75}{$\to$}}}
            {\rotatebox[origin=c]{20}{\scalebox{0.6}{$\to$}}}
}}}

\def\kgradconj#1{\grad K(#1)^*}

\title{Primal-dual block-proximal splitting for a class of non-convex problems}
\shorttitle{Non-convex primal-dual block-proximal splitting}

\date{2019-11-14 (revised 2020-04-22)}
\manuscriptlicense{}
\manuscriptcopyright{}

\author{%
    Stanislav Mazurenko\thanks{Loschmidt Laboratories, Masaryk University, Brno, Czechia (\email{stan.mazurenko@gmail.com}, \orcid{0000-0003-3659-4819})}
    \and
    Jyrki Jauhiainen\thanks{University of Eastern Finland, Kuopio, Finland (\email{jyrki.jauhiainen@uef.fi}, \orcid{0000-0001-6711-6997})}
	\and
    Tuomo Valkonen\thanks{ModeMat, Escuela Politécnica Nacional, Quito, Ecuador \emph{and} Department of Mathematics and Statistics, University of Helsinki, Finland (\email{tuomo.valkonen@iki.fi}, \orcid{0000-0001-6683-3572})}
    }
\shortauthor{S.~Mazurenko, J.~Jauhiainen, and T.~Valkonen}

\acknowledgements{
    This work was started when T.~Valkonen and S.~Mazurenko were at the University of Liverpool. There their work was supported by the EPSRC First Grant EP/P021298/1. More recently, T.~Valkonen has been supported by the Academy of Finland grants 314701 and 320022 as well as the Escuela Politécnica Nacional internal grant PIJ-18-03. S.~Mazurenko has been supported by the Operational Programme Research, Development and Education project MSCAfellow@MUNI (CZ.02.2.69/0.0/0.0/17\_050/0008496). J.~Jauhiainen has been supported by the European Union Horizon 2020 research and innovation programme under grant agreement No 764810.
}

\begin{document}

\maketitle

\begin{abstract}
    We develop block structure adapted primal-dual algorithms for non-convex non-smooth optimisation problems whose objectives can be written as compositions $G(x)+F(K(x))$ of non-smooth block-separable convex functions $G$ and $F$ with a non-linear Lipschitz-differentiable operator $K$. Our methods are refinements of the non-linear primal-dual proximal splitting method for such problems without the block structure, which itself is based on the primal-dual proximal splitting method of Chambolle and Pock for convex problems. We propose individual step length parameters and acceleration rules for each of the primal and dual blocks of the problem. This allows them to convergence faster by adapting to the structure of the problem. For the squared distance of the iterates to a critical point, we show local $O(1/N)$, $O(1/N^2)$ and linear rates under varying conditions and choices of the step lengths parameters.
    Finally, we demonstrate the performance of the methods on practical inverse problems: diffusion tensor imaging and electrical impedance tomography.
\end{abstract}

\section{Introduction}

We want to solve in Hilbert spaces $X$ and $Y$ the problem
\begin{equation}
    \label{eq:generic-problem}
    \tag{P$_0$}
	\min_{x\in X}~ G(x)+F(K(x)),
\end{equation}
where $G: X \to \extR$ and $F: Y \to \extR$ are convex, proper, and lower semicontinuous, but $K \in C^1(X; Y)$ is possibly non-linear. The linear case has been considered frequently in the literature, while in our earlier work \cite{tuomov-nlpdhgm,tuomov-pdex2nlpdhgm,tuomov-nlpdhgm-redo} we have developed first-order primal-dual methods for the generally non-convex problem with a non-linear $K$.
We refer to \cite{tuomov-firstorder} for a simplified overview of such methods.
In the present work, still with a non-linear $K$, we consider problems of the more specific form
\begin{equation}
    \label{eq:main-problem-primalonly}
    \tag{P}
    \min_{x \in X}~ \sum_{j=1}^m G_j(P_j x)+ \sum_{\ell=1}^n F_\ell(Q_\ell K(x)),
\end{equation}
where for all $j=1,\ldots,m$ and $\ell=1,\ldots,n$, the functions $G_j: X \to \extR$ and $F_\ell: Y \to \extR$ are convex, proper, and lower semicontinuous, and $P_1,\ldots,P_m \in \linear(X; X)$ as well as $Q_1,\ldots,Q_n \in \linear(Y; Y)$ are mutually orthogonal families of linear projection operators.
In other words, $G$ and $F$ are block-separable.
More specifically, we develop spatially adaptive and block-stochastic optimisation methods for the solution of \eqref{eq:main-problem-primalonly}.

As observed in \cite{tuomov-blockcp} for linear $K$, the adaptation of step lengths to individual blocks $j$ and $\ell$ can speed up the convergence of optimisation methods due to blockwise Lipschitz or strong convexity factors being better than the global factor.
Moreover, as now extensively studied, randomly sampling the blocks to be updated on each step can also improve convergence on very large-scale problems, in part due to the spatial adaptation, and in part due to being able to avoid communication in a cluster implementation of the algorithm. For more on stochastic block coordinate descent type methods, we refer to the review \cite{wright2015coordinate} and, among others, the original articles
\cite{nesterov2012efficiency,
      richtarik2012parallel,
      fercoq2015approx,
      richtarik2013distributed,
      zhao2014stochastic,
      shalev-shwartz2014accelerated,
      csiba2015stochastic,
      peng2016arock,
      bertsekas2015incremental}
on forward--backward type methods,
\mbox{%
\cite{zeng2014imagesegm,
      suzuki2014stochastic,
      zhang2015stochastic,
      combettes2016stochastic,
      bianchi2014stochastic,
      chambolle2018stochastic,
      fercoq2019coordinate,
      tuomov-blockcp}
}%
on primal-dual methods, and
\cite{qu2015sdna,
      pilanci2016iterative}
on second-order methods, all in the convex case.
For the non-convex case we point to \cite{xu2015block,xu2017globally}.
Compared to the latter, we work in the primal-dual setting and aim for spatial adaptation also in the deterministic setting. We also aim to prove convergence rates.

Several works consider, instead of a random selection of blocks, a random selection of terms of a sum of functions. In the non-convex case, recent mathematical works in this area include \cite{davis2019stochastic,milzarek2018stochastic}, aside from more applied works in the area of neural networks. In our block-stochastic approach, for non-convex $C^1$ functions $J_\ell$, ($\ell=1,\ldots,n$), we can with $K(x) \defeq (J_1(x), \ldots, J_n(x))$ and $F(z) \defeq \sum_{\ell=1}^n z_\ell$ write
\begin{equation}
    \label{eq:fb-problem}
    \min_x G(x)+\sum_{\ell=1}^n J_\ell(x)
    =\min_x G(x)+F(K(x)).
\end{equation}

To start describing our approach, using the conjugates $F_\ell^*$ of the convex, proper, lower semicontinuous functions $F_\ell$, we reformulate \eqref{eq:main-problem-primalonly} as the minmax problem
\begin{equation}
	\label{eq:main-problem}
    \tag{S}
	\min_{x\in X}\max_{y\in Y}~ \sum_{j=1}^m G_j(P_j x)+\iprod{K(x)}{y}-\sum_{\ell=1}^n F^*_\ell(Q_\ell y).
\end{equation}
If $K$ is linear, and the number of blocks $n=m=1$, a popular algorithm for solving this formulation is the primal-dual proximal splitting (PDPS) of Chambolle and Pock \cite{chambolle2010first}. It consists of alternating proximal steps with respect to the dual and primal variables, with the other variable fixed, and an over-relaxation step that ensures convergence. Its extension to non-linear $K$ (but still without blockwise structure) iterates \cite{tuomov-nlpdhgm,tuomov-nlpdhgm-redo}
\begin{equation*}
    \left\{\begin{aligned}
    \nextx & := \prox_{\tau_i G}(\thisx - \tau_i\kgradconj{\thisx}\thisy),\\
    \overnextx & :=  \nextx+\omega_i(\nextx-\thisx),\\
    \nexty & :=  \prox_{\sigma_{i+1} F^*}(\thisy + \sigma_{i+1} K(\overnextx))
    \end{aligned}\right.
\end{equation*}
for some step length and over-relaxation parameters $\tau_i$, $\sigma_{i+1}$, $\omega_i$ and $\prox_{\tau_i G}(x) \defeq (I + \tau_i \subdiff G)^{-1}(x)$.
Our purpose in this work is to randomise and adapt the method to the multi-block structure of \eqref{eq:main-problem}: \emph{firstly}, on each step we will only update random subsets of either or both primal and dual blocks, and, \emph{secondly}, even when we deterministically update every block on each step, we adapt the step lengths to the local structure of the problem in each block.

We organise our work as follows:
first, in \cref{sec:testing}, we introduce general notations, concepts, and the rough structure of the algorithm.
In \cref{sec:estimates} we start the convergence proof by deriving several technical estimates.
In \cref{sec:fulldual} we then use these estimates to derive convergence rates of more specific algorithms when only the primal updates are randomised. Likewise, in \cref{sec:fullprimal} we study the case when only the dual updates are randomised.
We finish our work in \cref{sec:numerical} with numerical experience in diffusion tensor imaging (DTI) and electrical impedance tomography (EIT).

\section{Notations, rough algorithm, and its testing}
\label{sec:testing}

Throughout this paper, we write $\linear(X; Y)$ for the space of bounded linear operators between Hilbert spaces $X$ and $Y$; $I$ is the identity operator; and $\iprod{x}{x'}$ is the inner product in the corresponding space. We write with $\powerset A$ for the power set of a set $A$ and $\chi_{A}(a)$ for the indicator function that equals $1$ if $a \in A$ and $0$ otherwise.
We set $\iprod{x}{x'}_T\defeq\iprod{Tx}{x'}$, and $\norm{x}_T\defeq\sqrt{\iprod{x}{x}_T}$, where in the latter we require $T \ge 0$.
For $T,S\in\linear(X; Y)$, the inequality $T\ge S$ means that $T-S$ is positive semidefinite.
If $H$ is a set-valued operator $X\setto X$, inequalities such as $\iprod{H(x)}{x'}\ge 0$ mean that $\iprod{w}{x'}\ge 0$ for every $w\in H(x)$.

We write $(\Omega,\SAlg, \P)$ for the probability space consisting of a sample set $\Omega$, a $\sigma$-algebra $\SAlg$ on $\Omega$, and a probability measure $\P$. We write $\Random(\SAlg; V)$ for the space of $V$-valued $\SAlg$-measurable random variables.
$\Random(\SAlg; \Space \setto \Space)$ is therefore the space of $\SAlg$-measurable random variables whose values are set-valued operators $\Space \setto \Space$.
Due to the iterative nature of optimisation algorithms, we introduce a sequence of $\sigma$-algebras $\{\SAlg_i\}_{i\in\N}$ such that $\SAlg_i\subseteq \SAlg_{i+1}$ and $\SAlg_i\subseteq \SAlg$ for any $i\in\N$.
We use $\SAlg_i$ to collect all the information available before the $(i+1)$:th iteration.
We write $\E_i[\freevar]\defeq\E[\freevar \mid \SAlg_i]$ for the corresponding conditional expectation.

Many conditions that we impose in the following sections only apply to the subspace on which the operator $K$ from the introduction acts non-linearly. Correspondingly, we introduce
\[
    \Ylin \defeq \{ y \in Y \mid \text{the map } x \mapsto \iprod{y}{K(x)} \text{ is linear} \}
    \quad\text{and}\quad
    \Ynl \defeq \Ylin^\perp,
\]
as well as the orthogonal projection $\Pnl$ to $\Ynl$.
See \cref{sec:numerical} for how such subspaces practically come about in applications.
We also use the short-hand notations
\[
    x_j \defeq P_j x
    \quad\text{and}\quad
    y_\ell \defeq Q_\ell y.
\]

\subsection{Abstract structure of the algorithm}

We generally use the symbol $x$ for primal variables (elements of $X$), and symbol $y$ for dual variables (elements of $Y$). We group these variables together into $u=(x, y) \in X \times Y$. This applies to indexed variables, $u^i :=(x^i,y^i)$, critical points $\realoptu=(\realoptx, \realopty)$, etc., without explicit introduction of the primal and dual components in each case.
We define the set-valued operator $H: X \times Y \setto X \times Y$ for $u=(x, y)$ as
\begin{equation}
	\label{eq:h}
	H(u) \defeq
	\begin{pmatrix}
		\subdiff G(x) + \kgradconj{x} y \\
		\subdiff F^*(y) - K(x)
	\end{pmatrix}
	\quad\text{with}\quad
	G(x) \defeq \sum_{j=1}^m G_j(P_j x)
	\quad\text{and}\quad
	F^*(y) \defeq \sum_{\ell=1}^n F^*_\ell(Q_\ell y).
\end{equation}
Then $0 \in H(\realoptu)$ encodes the critical point conditions for \eqref{eq:main-problem}.
These will also become the first-order necessary optimality conditions under a constraint qualification, e.g., when $G$ is $C^1$ and either the null space of $\kgradconj{x}$ is trivial or $\Dom F=X$ \cite[Example 10.8]{rockafellar-wets-va}.

Following the ``testing'' approach to convergence analysis from \cite{tuomov-proxtest}, we introduce the primal-dual \term{step length}, \term{testing}, and \term{preconditioning} operators
\begin{equation}
    \label{eq:steptestprecond}
	\Step_{i+1} \defeq
	\begin{pmatrix}
	\Tau_i & 0 \\
	0 & \Sigma_{i+1}
    \end{pmatrix},
    \quad
	\Test_{i+1} \defeq
	\begin{pmatrix}
	\TauTest_i & 0 \\
	0 & \SigmaTest_{i+1}
    \end{pmatrix},
    \quad\text{and}\quad
	\Precond_{i+1}\defeq
	\begin{pmatrix}
	I & -\inv\TauTest_i\Lambda_i^* \\
	-\inv\SigmaTest_{i+1}\Lambda_i & I
	\end{pmatrix}.
\end{equation}
Here $\Tau_i, \TauTest_i $ and $\Sigma_{i+1}, \SigmaTest_{i+1}$ are the respective primal and dual step length and testing operators, and $\Lambda_i$ is a term that we will develop to suitably decouple the updates of the primal and dual variables.
In the deterministic case, $\Tau_i, \TauTest_i \in \linear(X; X)$ and $\Sigma_{i+1}, \SigmaTest_{i+1} \in \linear(Y; Y)$ as well as $\Lambda_i \in \linear(X; Y)$.
Clearly, $\Test_{i+1}\Precond_{i+1}$ is self-adjoint.
For the stochastic setting we will impose our formal assumptions later in \eqref{eq:conditionality-assumptions}.
We will in particular require the tests $\TauTest_i$ and $\SigmaTest_{i+1}$ to already be known before the start of the $i$:th iteration (calculating $\thisu$), whereas the step lengths themselves will have to be known before the $(i+1)$:th iteration (calculating $\nextu$).

Finally, we write our proposed algorithm in the implicit form
\begin{equation}
	\label{eq:ppext}
	\tag{PP}
	0 \in \Step_{i+1}\tilde H_{i+1}(\nextu)+\Precond_{i+1}(\nextu-\thisu)
\end{equation}
for
\begin{equation}
    \label{eq:bregfun-def}
	\tilde H_{i+1}(\nextu)\defeq H(\nextu) +
	\begin{pmatrix}
		[\kgrad{\thisx}-\kgrad{\nextx}]^*{\nexty}  \\
		K(\nextx)-K(\nextx+\Omega^i(\nextx-\thisx))+\kgrad{\thisx}\Omega^i(\nextx-\thisx)
	\end{pmatrix}
\end{equation}
and some over-relaxation operator $\Omega_i$, which in the deterministic setting is in $\linear(X; X)$.
Here $\tilde H_{i+1}(u)$ is a partial linearization of $H(u)$ similar to \cite{tuomov-nlpdhgm}. It simplifies to $H(u)$ for a linear $K$.
In the following, by specifying the testing, step length, preconditioning, and over-relaxation operator, we develop more explicit methods from this implicit formulation, which itself is more amenable to convergence analysis.

\subsection{Testing for convergence}

The proximal point method iteratively solves $\nextu$ from
\begin{equation}
    \label{eq:ppexample}
    0 \in H(\nextu) + \inv\tau(\nextu-\thisu)
\end{equation}
given a step length parameter $\tau>0$.
If $H$ is a $\gamma$-strongly monotone operator and $\realoptu \in \inv H(0)$.
Then $\iprod{H(\nextu)}{\nextu-\realoptu} \ge \gamma\norm{\nextu-\realoptu}^2$.
This suggest ``testing'' \eqref{eq:ppexample} by the application of $\iprod{\freevar}{\nextu-\realoptu}$.
Subsequently to this testing, the strong monotonicity and Pythagoras' identity
\begin{equation*}
    \iprod{\nextu-\thisu}{\nextu-\realoptu}
    = \frac{1}{2}\norm{\nextu-\thisu}^2
    - \frac{1}{2}\norm{\thisu-\realoptu}^2
    + \frac{1}{2}\norm{\nextu-\realoptu}^2,
\end{equation*}
applied to $0 \in \iprod{H(\nextu) + \tau^{-1}(\nextu-\thisu)}{\nextu-\realoptu}$ yield
\[
    \frac{1+2\gamma\tau}{2}\norm{\nextu-\realoptu}^2
    +\frac{1}{2}\norm{\nextu-\thisu}^2
    \le
    \frac{1}{2}\norm{\thisu-\realoptu}^2.
\]
Telescoping this inequality, it is clear that $u^N \to \realoptu$ at the linear rate $O(1/(1+2\gamma\tau)^N)$.
The next theorem from \cite{tuomov-proxtest} generalises these simple arguments to the more general algorithm \eqref{eq:ppext} in the stochastic setting. 

\begin{theorem}[{\cite[Corollary 3.1]{tuomov-proxtest}}]
    \label{thm:convergence-result-main-h-stoch}
    On a Hilbert space $\Space$ and a probability space $(\Omega, \SAlg)$, let $\tilde H_{i+1}: \Random(\SAlg; \Space \setto \Space)$, and $\Precond_{i+1}, \Test_{i+1} \in \Random(\SAlg; \linear(\Space; \Space))$ for $i \in \N$.
    Suppose \eqref{eq:ppext} is solvable for $\{\nextu\}_{i \in \N} \subset \Random(\SAlg; \Space)$.
    If for all $i \in \N$ and almost all random events $\omega \in \Omega$, $(\Test_{i+1}\Precond_{i+1})(\omega)$ is self-adjoint, and the \term{expected fundamental condition}
    \begin{multline}
        \label{eq:convergence-fundamental-condition-iter-h-stoch}
        \E[\iprod{\Step_{i+1}\tilde H_{i+1}(\nextu)}{\nextu-\realoptu}_{\Test_{i+1}}]
        \ge
        \E\left[\frac{1}{2}\norm{\nextu-\realoptu}_{\Test_{i+2}\Precond_{i+2}-\Test_{i+1}\Precond_{i+1}}^2
        - \frac{1}{2}\norm{\nextu-\thisu}_{\Test_{i+1} \Precond_{i+1}}^2\right]
    \end{multline}
    holds, then so does the \term{expected descent inequality}
    \begin{equation}
        \label{eq:convergence-result-main-h-stoch}
        \E\left[\frac{1}{2}\norm{u^N-\realoptu}^2_{\Test_{N+1}\Precond_{N+1}}\right]
        \le
        \E\left[\frac{1}{2}\norm{u^0-\realoptu}^2_{\Test_{1}\Precond_{1}}\right]
        \quad
        (N \ge 1).
    \end{equation}
\end{theorem}

The condition \eqref{eq:convergence-fundamental-condition-iter-h-stoch} is simply a relaxation of the strong monotonicity we assumed above. It also includes the term $\frac{1}{2}\norm{\nextu-\thisu}_{\Test_{i+1} \Precond_{i+1}}$ intended to be used with forward steps.
	In application to \eqref{eq:ppexample}, we have $\Precond_{i+1} = I$, and we can take as the testing operator $\Test_{i+1} = \tauTest_{i} I$ with $\tauTest_{i+1}=(1+2\gamma\tau)\tauTest_i$ and $\tauTest_0=1$. Thus $\Test_{N+1}\Precond_{N+1}$ in \eqref{eq:convergence-result-main-h-stoch} forms a local metric that measures rates of convergence.
If we can ensure $\Test_{i+1}\Precond_{i+1}\ge\mu_i I$ for some deterministic $\mu_i \upto \infty$, then \eqref{eq:convergence-result-main-h-stoch} shows $\E[\norm{u^N-\realoptu}^2]$ to converge to zero at the rate $O(1/\mu_N)$.
We will in \cref{sec:estimates} develop lower bounds of this kind.

\subsection{Blockwise algorithm structure}

We now develop a more blockwise-refined structure of our proposed algorithm.
Inserting \eqref{eq:steptestprecond}, we can expand \eqref{eq:ppext} as the pair of implicit updates (compare \cite[§2.3]{tuomov-blockcp})
\begin{equation}
	\label{eq:lambda-algorithm}
	\left\{\begin{aligned}
		\nextx&=(I+\Tau_i \subdiff G)^{-1}
		(\thisx+\TauTest_i^{-1}[\Lambda_i^*-\TauTest_i\Tau_i\kgradconj{\thisx}](\nexty-\thisy)-\Tau_i\kgradconj{\thisx}\thisy),
		\\
		\nexty&=(I+\Sigma_{i+1} \subdiff F^*)^{-1}
		(\thisy+\SigmaTest_{i+1}^{-1}[\Lambda_i-\SigmaTest_{i+1}\Sigma_{i+1}\kgrad{\thisx}\Omega^i](\nextx-\thisx)
		\\
		&\qquad\qquad\qquad\qquad
		+\Sigma_{i+1}K(\nextx+\Omega^i(\nextx-\thisx))).
	\end{aligned}\right.
\end{equation}

Due to the block-separable structure of $G$ and $F^*$ in \eqref{eq:h}, we take for all $i \in \N$,
\begin{subequations}%
\label{eq:step-length-structure}%
\begin{align}%
	\Tau_i&\defeq\sum_{j\in S(i)} \tau_j^{i} P_j,
	&\Sigma_{i+1}&\defeq\sum_{\ell \in V(i+1)} \sigma_\ell^{i+1} Q_\ell,
	&\Omega_i&\defeq\sum_{j\in S(i)} \omega_j^i P_j,
	\\
	\TauTest_i&\defeq\sum_{j=1}^{m} \tauTest_j^{i} P_j,
	&\SigmaTest_{i+1}&\defeq\sum_{\ell=1}^{n} \sigmaTest_\ell^{i+1} Q_\ell,\quad\text{and}
	&\Lambda_i&\defeq\sum_{j=1}^{m} \sum_{\ell=1}^{n} \lambda_{\ell,j}^{i} Q_\ell \kgrad{\thisx} P_j,
\end{align}%
\end{subequations}%
for some (random) subsets of indices $S(i)\subseteq\{1,\dots,m\}$ and $V(i+1)\subseteq\{1,\dots,n\}$ and (random) parameters $\tau_j^i,\tauTest_j^i,\sigma_\ell^{i+1},\sigmaTest_\ell^{i+1}>0$, and $\omega_j^i, \lambda_{j,\ell}^i \in \R$. We wait until \eqref{eq:conditionality-assumptions} to specify the exact probabilistic setup, which we do not need before that.
Due to the block-separable structures of $G$ and $F^*$, the operators $(I+\Tau_i \subdiff G)^{-1}$ and $(I+\Sigma_{i+1} \subdiff F^*)^{-1}$ are also block-separable.


We also pick further subsets of indices $\iset{S}(i)\subseteq S(i)$ and $\iset V(i+1) \subset V(i+1)$; the rough idea is that $\nextx_j$ for $j \in \iset S(i)$ is updated within each step of the algorithm independently of $\nexty$. In the linear-$K$ case of \cite{tuomov-blockcp} also $\nexty_\ell$ for $\ell \in \iset V(i+1)$ would be updated independently of $\nextx$, but presently we are not able to ensure that.
However, we show at the end of this subsection that the primal blocks $\nextx_j$ for $j \in S(i)\setminus \iset S(i)$ still depend on $\nexty_\ell$ only for $\ell\in\iset V(i+1)$, as is the case for a linear $K$ in \cite{tuomov-blockcp}.
Moreover we require the ``nesting conditions''
\begin{subequations}%
\label{eq:chilo-definition}%
\begin{align}
	\chi_{\iset S(i)}(j)(1-\chi_{V(i+1)}(\ell))
    & =0,
    &
	(1-\chi_{S(i)}(j))\chi_{\iset V(i+1)}(\ell)
    & =0,
    \\
	\chi_{\iset S(i)}(j)\chi_{\iset V(i+1)}(\ell) & =0,
    \quad\text{and}
    &
	\chi_{S(i)\setminus\iset S(i)}(j)\chi_{V(i+1)\setminus\iset V(i+1)}(\ell) & =0
\end{align}
when
\begin{equation}
    \label{eq:neigh-def}
    \ell \in \this\Neigh_j \defeq \{ \ell \in \{1,\ldots,n\} \mid Q_\ell \grad K(\thisx) P_j \ne 0 \}.
\end{equation}%
\end{subequations}
These conditions force those dual blocks that are ``connected'' by $K$ to the ``independently updated'' primal blocks $\iset S(i)$ to also be (``dependently'') updated, and vice versa. They also disallow connections between independently updated blocks and dependently updated blocks.
Note that the last three equations in \eqref{eq:chilo-definition} are tantamount to the single equality $\chi_{V(i+1)}(\ell)\chi_{S(i)/\iset{S}(i)}(j)=\chi_{\iset{V}(i+1)}(\ell)$: they follow by multiplying the latter by $1-\chi_{S(i)}$, $\chi_{\iset{S}(i)}(j)$, and $\chi_{S(i)/\iset{S}(i)}(j)$, respectively; and vice versa
$\chi_{\iset{V}(i+1)}(\ell)=\chi_{\iset{V}(i+1)}(\ell)\chi_{S(i)}(j)=\chi_{\iset{V}(i+1)}(\ell)\chi_{S(i)/\iset{S}(i)}(j)=\chi_{V(i+1)}(\ell)\chi_{S(i)/\iset{S}(i)}(j)$.

\begin{example}
    We can trivially satisfy \eqref{eq:chilo-definition} by taking either $V(i+1)=\{1,\ldots,n\}$, $\iset V(i+1)=\emptyset$, and $\iset S(i)= S(i)$ or $S(i)=\{1,\ldots,m\}$, $\iset S(i)=\emptyset$, and $\iset V(i+1)=V(i+1)$. We will consider these two cases in the respective \cref{sec:fulldual} (full dual update methods) and \cref{sec:fullprimal} (full primal update methods). We may also alternate iterations between these two choices.
\end{example}

Following the notations for the subsets and their complements, we also write \[
	\iset{P}_i \defeq \sum_{j\in \iset{S}(i)}P_j,
	\quad
	\dset{P}_i \defeq \sum_{j\in S(i) \setminus \iset{S}(i)}P_j,
	\quad
	\iset{Q}_{i+1} \defeq \sum_{\ell\in \iset V(i+1)}Q_\ell,
	\quad\text{and}\quad
	\dset{Q}_{i+1} \defeq \sum_{\ell\in V(i+1) \setminus \iset V(i+1)}Q_\ell.
\]
In \eqref{eq:lambda-algorithm}, for the subsets $S(i)$ and $V(i+1)$ to have the intended meaning that only the corresponding blocks are updated, we need to ensure that $P_j\nextx=P_j\thisx$ for $j\not\in S(i)$ and $Q_\ell\nexty=Q_\ell\thisy$ for $\ell \not\in V(i+1)$.
This holds if $P_j\Lambda_i^*Q_\ell=0$ whenever $j\notin S(i)$, $\ell \in V(i+1)$ or $j \in S(i)$, $\ell \notin V(i+1)$ or $j\notin S(i)$, $\ell \notin V(i+1)$.
Similarly, for $\iset S(i)$ to have the intended meaning that $\nextx_j$ for $j \in \iset S(i)$ does not depend on $\nexty$, studying \eqref{eq:lambda-algorithm}, we are also led to require
\[
    \iset{P}_i[\Lambda_i^*-\TauTest_i\Tau_i\kgradconj{\thisx}]Q_\ell=0
    \quad \text{for any $\ell \in V(i+1)$}.
\]
Finally, since $\dset{P}_i\nextx$ may in \eqref{eq:lambda-algorithm} depend on $\nexty$, we require $\nexty$ to not depend on $\dset{P}_i\nextx$:
\[
    [\Lambda_i-\SigmaTest_{i+1}\Sigma_{i+1}\kgrad{\thisx}\Omega^i]\dset{P}_i=0
    \quad\text{and}\quad
    [I+\Omega^i]\dset{P}_i=0.
\]
Combining the above conditions on $\Lambda_i$ and $\Omega_i$, we arrive at
\begin{equation}
    \label{eq:lambda-conditions}
    \left\{\begin{aligned}
        &P_j\Lambda_i^*Q_\ell=0
        &
        &\text{whenever either $j\not\in S(i)$ or $\ell \not\in V(i+1)$ or both},
        \\
        &\iset{P}_i[\Lambda_i^*-\TauTest_i\Tau_i\kgradconj{\thisx}]Q_\ell=0
        &
        &\text{for } \ell \in V(i+1),
        \\
        &[\Lambda_i+\SigmaTest_{i+1}\Sigma_{i+1}\kgrad{\thisx}]\dset{P}_i=0,
        &
        &\text{and}\quad (\Omega^i+I)\dset{P}_i=0.
    \end{aligned}\right.
\end{equation}

Substituting \eqref{eq:lambda-conditions} into the identity
\[
    \Lambda_i=
    \sum_{\ell \in V(i+1)}Q_\ell\Lambda_i\iset{P}_i
    +\sum_{\ell \not\in V(i+1)}Q_\ell\Lambda_i\iset{P}_i
    +\Lambda_i\dset{P}_i
    +\sum_{j\not\in S(i)}\sum_{\ell=1}^n Q_\ell\Lambda_iP_j,
\]
we are led to take
\begin{equation}
    \label{eq:def-lambdai}
    \Lambda_i\defeq
    \sum_{\ell \in V(i+1)}Q_\ell\kgrad{\thisx}\Tau_i^*\TauTest_i^*\iset{P}_i
    -\SigmaTest_{i+1}\Sigma_{i+1}\kgrad{\thisx}\dset{P}_i,
\end{equation}
which in terms of the components $\this\lambda_{\ell,j}$ reads
\begin{equation}
    \label{eq:def-lambdacomponent}
    \lambda_{\ell,j}^i\defeq \left\{\begin{array}{ll}
    \tau^i_j\tauTest^i_j&\ell \in V(i+1), j\in \iset{S}(i),\\
    -\sigma^{i+1}_\ell\sigmaTest^{i+1}_\ell&\ell \in V(i+1), j\in S(i) \setminus \iset{S}(i),\\
    0&\text{otherwise}.
    \end{array}
    \right.
\end{equation}

Using the coupling conditions \eqref{eq:chilo-definition} between $\iset S(i)$ and $\iset V(i+1)$ in \eqref{eq:def-lambdai}, we deduce
\[
    \Lambda_i=\kgrad{\thisx}\Tau_i^*\TauTest_i^*\iset{P}_i-\iset{Q}_{i+1}\SigmaTest_{i+1}\Sigma_{i+1}\kgrad{\thisx}.
\]
Plugging $\Lambda_i$ into \eqref{eq:lambda-algorithm}, we get two cases for the primal variable. If $j\in \iset S(i)$, we have
\[
	\iset P_i \nextx = (I+\iset\Tau_i \subdiff G)^{-1}(\iset P_i\thisx-\iset{\Tau}_i\kgradconj{\thisx}\thisy),
	\quad\text{where}\quad
	\iset\Tau_i\defeq \iset P_i\Tau_i.
\]
If $j\in S(i)\setminus \iset S(i)$, given that $\Omega^i\dset{P}_i=-\dset{P}_i$ due to the last equality of \eqref{eq:lambda-conditions}, taking $\dset{\Tau}_i\defeq \dset P_i\Tau_i$, we have
\[
	\dset P_i\nextx = (I+\dset{\Tau}_i \subdiff G)^{-1}
	(\dset P_i\thisx-\dset\Tau_i\kgradconj{\thisx}\iset Q_{i+1}\nexty-\dset P_i\TauTest_i^{-1}\kgradconj{\thisx}\Sigma_{i+1}^*\SigmaTest_{i+1}^*\iset Q_{i+1}(\nexty-\thisy)).
\]
Also $\nextx = \iset P_i \nextx + \dset P_i \nextx + (I-\iset P_i -\dset P_i) \nextx$, therefore, for $\overnextx = \nextx +\Omega^i(\nextx-\thisx)$ we can expand
$\nextx  = \iset P_i\nextx -\Omega^i\dset{P}_i\nextx +(I-\iset P_i-\dset P_i)\thisx$
$= \iset P_i\nextx + (I-\iset P_i)\thisx -\Omega^i\dset{P}_i(\nextx-\thisx)$.
Consequently, the implicitly defined algorithm in \eqref{eq:lambda-algorithm} expands into the explicit successive updates for each of the involved projections :
\begin{equation}
    \label{eq:alg-ordered}
    \left\{
    \begin{aligned}
        \iset P_i \nextx & \defeq
        (I+\iset\Tau_i \subdiff G)^{-1}(\iset P_i\thisx-\iset{\Tau}_i\kgradconj{\thisx}\thisy),
        \\
        \overnextx& \defeq
        (I-\iset P_i)\thisx + \iset P_i\nextx+\iset P_i\Omega_i\iset P_i(\nextx-\thisx),
        \\
        \nexty& \defeq
        (I+\Sigma_{i+1} \subdiff F^*)^{-1} \Bigl(\thisy+\Sigma_{i+1}K(\overnextx)
        \\
        \MoveEqLeft[-2]
        	+\dset Q_{i+1}\SigmaTest_{i+1}^{-1}[\kgrad{\thisx}\Tau_i^*\TauTest_i^*-\SigmaTest_{i+1}\Sigma_{i+1}\kgrad{\thisx}\Omega_i]\iset P_i(\nextx-\thisx)\Bigr),
        \\
        \dset P_i\nextx& \defeq (I+\dset{\Tau}_i \subdiff G)^{-1} \Bigl(\dset P_i\thisx -\dset\Tau_i\kgradconj{\thisx}\iset Q_{i+1}\nexty
        \\
        \MoveEqLeft[-7]
        	-\dset P_i\TauTest_i^{-1}\kgradconj{\thisx}\Sigma_{i+1}^*\SigmaTest_{i+1}^*\iset Q_{i+1}(\nexty-\thisy)\Bigr),
		\\
		P_j\nextx & \defeq P_j\thisx \quad \text{for } j \notin S(i).
    \end{aligned}
    \right.
\end{equation}
In the following sections we will further develop and simplify this algorithm by imposing additional conditions on the step length and testing parameters through convergence analysis.

\section{General estimates}
\label{sec:estimates}

With the estimate \eqref{eq:convergence-result-main-h-stoch} in mind, our main task in this section is to prove \eqref{eq:convergence-fundamental-condition-iter-h-stoch}.
After introducing the assumptions we need for this work in \cref{sec:assumptions}, and bounding $\Test_{i+1}\Precond_{i+1}$ from below in \cref{sec:zimi-estim}, we do the first stage of this estimation in \cref{sec:penalty0} still deterministically.
Then in \cref{sec:penalty1} we refine these estimates by taking the expectation.
Finally in \cref{sec:penalty2} we combine the various estimates and state a self-contained result on the validity of \eqref{eq:convergence-result-main-h-stoch}.

\subsection{Assumptions}
\label{sec:assumptions}

We will need $K$ to be sufficiently smooth and to satisfy a somewhat technical ``three-point'' version of standard second-order growth conditions:

\begin{assumption}[Lipschitz $\kgrad{x}$]
    \label{ass:k-lipschitz}
    For some  $L\ge 0$ and a neighbourhood $\neighx_K \ni \realoptx$,
    \begin{equation}
        \label{eq:ass-k-lipschitz}
        \norm{\kgrad{x}-\kgrad{x'}}  \le L\norm{x-x'}
        \quad (x,x'\in\neighx_K).
    \end{equation}
\end{assumption}

Using the equality
\[
    K(x')=K(x)+\kgrad{x}(x'-x)+\int_{0}^{1}(\kgrad{x+s(x'-x)}-\kgrad{x})(x'-x)ds,
\]
we obtain for any $x,x' \in \neighx_K$ and $y\in \Dom F^*$ as a direct consequence of \cref{ass:k-lipschitz} that
\begin{equation}
    \label{eq:ass-k}
    \iprod{K(x')-K(x)-\kgrad{x}(x'-x)}{y} \le \frac{L}{2}\norm{x-x'}^2 \norm{y}_{\Pnl}.
\end{equation}
The norm of $y$ only needs to be evaluated within $\Ynl$ because $x \mapsto (I-\Pnl)K(x)$ is linear so the corresponding inner product with the integral term is zero.

\begin{assumption}[three-point condition on $K$]
    \label{ass:k-nonlinear}
    For a neighbourhood $\neighx_K$ of $\realoptx$, some $\Gamma_K=\textstyle\sum_{j=1}^{m}\gamma_{K,j} P_j \in \linear(X; X)$ with $\gamma_{K,j}\in \R$, $L_3 \ge 0$, and $p \in [1,2]$, for any $A = \sum_{j=1}^m a_j P_j \ge 0$ and some $\theta_A \ge 0$ the following holds
    \begin{multline}
        \label{eq:ass-k-nonlinear}
        \iprod{[\kgrad{x}-\kgrad{\realoptx}]^*\realopty}{x'-\realoptx}_A
        \\
        \ge
        \norm{x'-\realoptx}^2_{A\Gamma_K}
        +\theta_A\norm{K(\realoptx)-K(x)-\kgrad{x}(\realoptx-x)}^p-\frac{L_3}{2}\norm{x'-x}_A^2,
        \quad (x,x'\in\neighx_K).
    \end{multline}
\end{assumption}
	This assumption is trivially satisfied for $\gamma_{K,j}=L_3=0$ and any $\theta_A>0$ whenever $x \mapsto \iprod{K(x)}{\realopty}$ is linear.
	In Appendix \ref{sec:kcond} we also provide the constants ensuring this assumption, e.g., whenever the latter is block-separable and strongly-convex.
	For a less straight-forward example in the single-block case, we refer to \cite{tuomov-nlpdhgm-redo}. There we verified the assumption for the reconstruction of the phase and amplitude of a complex number from a noisy measurements.
	That example evidently applies to the present setting in the single-block case or as a separable block of $x \mapsto \iprod{K(x)}{\realopty}$.

We also need pointwise monotonicity of $\subdiff G$ and $\subdiff F^*$ at a root $\realoptu \in \inv H(0)$:

\begin{definition}
    \label{def:monot}
    Let $U$ be a Hilbert space, and $\Gamma \in \linear(U; U)$, $\Gamma\ge0$.
    We say that the set-valued map $H: U \setto U$ is \term{$\Gamma$-strongly monotone at $\realoptu$ for $\realoptw \in H(\realoptu)$} if there exists a neighbourhood $\neighu \ni \realoptu$ such that for any $u \in \neighu$ and $w\in H(u)$,
    \begin{equation}
    \label{eq:monot}
    \iprod{w-\realoptw}{u-\realoptu} \ge \norm{u-\realoptu}_{\Gamma}^2.
    \end{equation}
    If $\Gamma=0$, we say that $H$ is monotone at $\realoptu$ for $\realoptw$.
\end{definition}

\begin{assumption}
    \label{ass:gf}
    For any $\realoptw=(\realopt{\nu},\realopt{\xi}) \in H(\realoptu)$, the set-valued map $\subdiff G$ is $\sum_{j=1}^m \gamma_{G,j} P_j$-strongly monotone at $\realoptx$ for $\realopt{\nu}-\kgradconj{\realoptx}\realopty$ in the neighbourhood $\neighx_G$, and the set-valued map $\subdiff F^*$ is $\sum_{\ell=1}^n \gamma_{F^*,\ell}Q_\ell$-strongly monotone at $\realopty$ for $\realopt{\xi}+K(\realoptx)$ in the neighbourhood $\neighy_{F^*}$, where the constants $\gamma_{G,j}, \gamma_{F^*,\ell} \ge 0$ for all $j=1,\ldots,m$ and $\ell=1,\ldots,n$.
\end{assumption}

\subsection{A lower bound on the local metric}
\label{sec:zimi-estim}

To estimate $\Test_{i+1}\Precond_{i+1}$ from below, we formulate a block-adapted version of the basic step length condition $\tau\sigma\norm{K}^2 < 1$ from \cite{chambolle2010first}.
The assumptions of the following lemma replace the more abstract constructions of \cite[Definition 2.2 and Examples 2.3 and 2.4]{tuomov-blockcp}.
We recall from \eqref{eq:neigh-def} the ``set of connections'' $\this\Neigh_j$ and also introduce the set of ``simultaneous connections'', filtered by $\lambda_{k,j}^i$, as
\begin{equation}
	\label{eq:def-sim-connected-blocks}
    {\bar\Neigh_j}^i(\ell) \defeq \{k \in \{1,\ldots,n\} \mid Q_\ell \kgrad{\thisx} P_j \kgradconj{\thisx} Q_k \ne 0,\, \lambda_{k,j}^i \ne 0 \}.
\end{equation}

\begin{lemma}
	\label{lemma:sigmatest}
    Let $i \in \N$ and $0\le\delta\le\kappa<1$.
    For some factors $w_{j,\ell,k}^i=1/w_{j,k,\ell}^i>0$, ($\ell, k=1,\ldots,n$; $j=1,\ldots,m$), define
    \begin{gather}
        \label{eq:psi-bounding-bounds}
        w_{j,\ell}^i \defeq \chi_{\this\Neigh_j}(\ell) \sum_{k \in \this{\bar\Neigh_j}(\ell)} w_{j,\ell,k}^i
	\shortintertext{and suppose}
        \label{eq:sigmatest}
	    (1-\kappa) \nexxt\sigmaTest_\ell \ge  \BiggNorm{
				\sum_{j=1}^{m} \abs{\this\lambda_{\ell,j}} \sqrt{w_{j,\ell}^i/\this\tauTest_j}
				Q_\ell \kgrad{\thisx} P_j
        }^2
	    \quad (\ell=1,\ldots,n).
    \end{gather}
    Then
    \begin{equation}
        \label{eq:sigmatest-result}
        \Test_{i+1}\Precond_{i+1}\ge 
        \begin{pmatrix}
        \delta \TauTest_i & 0 \\
        0 & \frac{\kappa-\delta}{1-\delta}\SigmaTest_{i+1}
        \end{pmatrix}.
    \end{equation}
    
\end{lemma}

\begin{proof}
	Setting  $\zeta_{\ell,j} \defeq \inv{(\this\tauTest_j)}(\this\lambda_{\ell,j})^2/(1-\kappa)$, we use \cref{eq:sigmatest} and the orthogonality of the projections $\{P_j\}_{j=1}^m$
	to obtain for any $y \in Y$ that
	\[
		\begin{split}
		\sum_{\ell=1}^{n}\nexxt\sigmaTest_\ell \norm{Q_\ell y}^2
		&\ge
		\sum_{\ell=1}^{n}\adaptNorm{\sum_{j=1}^{m} \sqrt{\zeta_{\ell,j} w_{j,\ell}^i}  Q_\ell \kgrad{\thisx} P_j}^2 \norm{Q_\ell y}^2
		\ge
		\sum_{\ell=1}^{n}\adaptNorm{\sum_{j=1}^{m} \sqrt{\zeta_{\ell,j} w_{j,\ell}^i}  P_j \kgradconj{\thisx} Q_\ell y}^2
		\\
		&
		=
		\sum_{\ell=1}^{n}\sum_{j=1}^{m} \zeta_{\ell,j} w_{j,\ell}^i  \norm{P_j \kgradconj{\thisx} Q_\ell y}^2
		\ge
		\sum_{j=1}^{m}  \sum_{\ell \in \this\Neigh_j} \Biggl(\sum_{k \in \this{\bar\Neigh_j}(\ell)} w_{j,\ell,k}^i\Biggr) \zeta_{\ell,j} \norm{P_j \kgradconj{\thisx} Q_\ell y}^2.
		\end{split}
	\]
    Since $w^{i}_{j,k,\ell}=1/w^{i}_{j,\ell,k}$, we continue to estimate by Young's inequality
    \[
	    \sum_{\ell=1}^{n}\nexxt\sigmaTest_\ell \norm{Q_\ell y}^2
	    \ge
	    \sum_{j=1}^{m} \sum_{k,\ell=1}^{n} \zeta_{\ell,j}^{1/2} \zeta_{k,j}^{1/2} \iprod{P_j \kgradconj{\thisx} Q_\ell y}{\kgradconj{\thisx} Q_k y}.
	\]
    Here we also used \eqref{eq:def-sim-connected-blocks} to convert the second sum to run over all $k,\ell=1,\ldots,n$.
    As $y \in Y$ was arbitrary, inserting $\zeta_{k,j}$ and the structure \eqref{eq:step-length-structure} of $\SigmaTest_{i+1}$, $\TauTest_i$, and $\Lambda_i$, we deduce $(1-\kappa)\SigmaTest_{i+1}\ge \Lambda_i\inv\TauTest_i\Lambda_i^*$.
    
    On the other hand, applying Young's inequality with the factor $(1-\delta)$ we deduce that
    \begin{equation}
	    \label{eq:test-precond-expansion-estimate}
        \Test_{i+1}\Precond_{i+1}
        =
        \begin{pmatrix}
            \TauTest_i & -\Lambda_i^* \\
            -\Lambda_i & \SigmaTest_{i+1}
        \end{pmatrix}
	    \ge
	    \begin{pmatrix}
	    \delta \TauTest_i & 0 \\
	    0 & \SigmaTest_{i+1} - \frac{1}{1-\delta}\Lambda_i\inv\TauTest_i\Lambda_i^*
    \end{pmatrix}
    \end{equation}
    Thus \eqref{eq:sigmatest-result} holds.
    
\end{proof}

The next example demonstrates a simple choice of the weights $w_{j,k,\ell}$ that is likely to work if all the dual blocks $\ell$ have similar roles in the problem. In \cref{sec:numerical} we will also consider other options when some dual blocks have different roles.

\begin{example}[Equal weighting]
    \label{ex:weight-count}
    Suppose $\this\Neigh_j \subset \Neigh_j$ and $\this{\bar\Neigh_j}(\ell) \subset {\bar\Neigh_j}(\ell)$ where $\Neigh_j$ and ${\bar\Neigh_j}(\ell)$ do not depend on the iteration. If we take $w_{j,\ell,k}^i \equiv 1$, then $w_{j,\ell}=\chi_{\Neigh_j}(\ell)\#{\bar\Neigh_j}(\ell)$ counts the dual blocks ``simultaneously connected'' with $\ell$ via the primal block $j$ as defined by \eqref{eq:def-sim-connected-blocks}.
\end{example}

To provide further intuition into the result, let $w_{j,\ell}$ be as in \cref{ex:weight-count}.
With only one primal block ($j,m=1$), and assuming full connectedness ($w_{1,\ell}=n$ for all $\ell=1,\ldots,n$), \cref{lemma:sigmatest} requires $\sigmaTest_\ell \ge \zeta_{1,\ell} n \norm{Q_\ell \grad K(\thisx)}^2$.
Let $a \defeq \frac{1}{n} \sum_{\ell=1}^n \norm{Q_\ell \grad K(\thisx)}^2=\frac{1}{n}\norm{\grad K(\thisx)}^2$. After plugging $\this\lambda_{\ell,j}$ from \eqref{eq:def-lambdacomponent} into \eqref{eq:sigmatest}, the lemma then says that the step length parameters can be proportionally larger compared to the single dual block case ($n=1$) when $\norm{Q_\ell \grad K(\thisx)}^2 < a$, and have to be proportionally smaller when $\norm{Q_\ell \grad K(\thisx)}^2 > a$.
In \cref{sec:fulldual} and \cref{sec:fullprimal}, we further transform \eqref{eq:sigmatest} to obtain explicit step-length conditions. But now, for the remainder of \cref{sec:estimates}, we assume that \eqref{eq:sigmatest-result} holds and derive sufficient conditions to be able to apply \cref{thm:convergence-result-main-h-stoch}.

\subsection{Initial non-stochastic estimates}
\label{sec:penalty0}

The next lemma starts the verification of \eqref{eq:convergence-fundamental-condition-iter-h-stoch}.

\begin{lemma}
    \label{lemma:penalty-general}
    Suppose \cref{ass:k-lipschitz,ass:gf} hold together with \eqref{eq:sigmatest-result} for some $L\ge 0$, $\gamma_{G,j},\gamma_{F^*,\ell}\ge0$ ($j=1,\ldots,m$, $\ell=1,\ldots,n$), and $0\le\delta\le\kappa<1$. Then with $\tilde H_{i+1}$ given by \eqref{eq:bregfun-def} and $\Precond_{i+1}$ given by \eqref{eq:steptestprecond}, we have
    \begin{multline}
    	\label{eq:basic-deterministic-estimate}
	    \frac{1}{2}\norm{\nextu-\thisu}_{\Test_{i+1} \Precond_{i+1}}^2
	    + \frac{1}{2}\norm{\nextu-\realoptu}_{\Test_{i+1}\Precond_{i+1}-\Test_{i+2}\Precond_{i+2}}^2
	    + \iprod{\tilde H_{i+1}(\nextu)}{\nextu-\realoptu}_{\Step_{i+1}\Test_{i+1}}
	    \\
    	\ge
    	\frac{1}{2}\norm{\nextx-\thisx}_{R_x}^2
    	+\frac{1}{2}\frac{\kappa-\delta}{1-\delta}\norm{\nexty-\thisy}_{\SigmaTest_{i+1}}^2
    	+\frac{1}{2}\norm{\nextu-\realoptu}_{R'}^2
    	+D_i^K+D_i^\Lambda,
    \end{multline}
    where for an arbitrary $\Gamma_K\defeq\textstyle\sum_{j=1}^{m}\gamma_{K,j} P_j \in \linear(X; X)$ for $\gamma_{K,j}\in \R$ we set
    \begin{subequations}%
    \label{eq:penalty-general-vars}%
    \begin{align}
        \label{eq:def-rx}
        R_x&\defeq\delta\TauTest_{i}
        -L\norm{\Omega^i+I}^2\norm{\SigmaTest_{i+1}^*\Sigma_{i+1}^*(\nexty-\realopty)}_{\Pnl}I,
        \\
        \label{eq:penalty-general-vars-rprime}
        R' & \defeq
        \begin{psmallmatrix}
            \TauTest_i-\TauTest_{i+1}+2\sum_{j\in S(i)}\tauTest_j^{i} \tau_j^{i}(\gamma_{G,j}+\gamma_{K,j})P_j & 0
            \\
            0 & \SigmaTest_{i+1}-\SigmaTest_{i+2}+2\sum_{\ell \in V(i+1)}\sigmaTest_\ell^{i+1} \sigma_\ell^{i+1}\gamma_{F^*,\ell}Q_\ell
        \end{psmallmatrix},
        \allowdisplaybreaks
        \\
        \label{eq:def-dilambda}
        D_i^\Lambda &\defeq
        \iprod{[\Lambda_{i+1}-\Lambda_i](\nextx-\realoptx)}{\nexty-\realopty}
        \\
        \nonumber
        \MoveEqLeft[-1]
        +\iprod{\kgradconj{\thisx}(\nexty-\realopty)}{\nextx-\realoptx}_{\TauTest_{i}\Tau_{i}-\Sigma_{i+1}^*\SigmaTest_{i+1}^*},
        \quad\text{and}
        \allowdisplaybreaks
        \\
        \label{eq:def-dik}
        D_i^K&\defeq
        \iprod{[\kgrad{\thisx}-\kgrad{\realoptx}]^*{\realopty}}{\nextx-\realoptx}_{\TauTest_{i}\Tau_{i}}
        -\norm{\nextx-\realoptx}_{\TauTest_{i}\Tau_{i}\Gamma_K}^2
        \\
        \nonumber
        \MoveEqLeft[-1]
        +\iprod{K(\realoptx)-K(\thisx)-\kgrad{\thisx}(\realoptx-\thisx)}{\nexty-\realopty}_{\SigmaTest_{i+1}\Sigma_{i+1}}.
    \end{align}%
    \end{subequations}
\end{lemma}

\begin{proof}
    We bound from below all the terms on the left-hand side of \cref{eq:basic-deterministic-estimate}.
    For the first term, we have from \eqref{eq:sigmatest-result} that
    \begin{equation}
		\label{eq:sigmatest-result:copy}
        \Test_{i+1}\Precond_{i+1}\ge
        \begin{pmatrix}
        \delta \TauTest_i & 0 \\
        0 & \frac{\kappa-\delta}{1-\delta}\SigmaTest_{i+1}
		\end{pmatrix}.
    \end{equation}
	For the second term we use the expansion
    \begin{equation}
    	\label{eq:metric-update}
    	\Test_{i+1}\Precond_{i+1}-\Test_{i+2}\Precond_{i+2}=
    	\begin{pmatrix}
    		\TauTest_i-\TauTest_{i+1} & \Lambda_{i+1}^*-\Lambda_i^* \\
    		\Lambda_{i+1}-\Lambda_i & \SigmaTest_{i+1}-\SigmaTest_{i+2}
    	\end{pmatrix}.
    \end{equation}

    We need to work more to estimate the third term on the left-hand side  of \eqref{eq:basic-deterministic-estimate}. Since $0 \in H(\realoptu)$, we have $\subdiff G(\realoptx) \ni z_G \defeq - \kgradconj{\realoptx}\realopty$, and $\subdiff F^*(\realopty) \ni z_{F^*} \defeq K(\realoptx)$.
    We can therefore recall the definition of $H(u)$ from \eqref{eq:h} and rewrite
    \begin{multline*}
    	\iprod{H(u)}{u-\realoptu}_{\Step_{i+1}\Test_{i+1}}
    	=\iprod{\subdiff G(x)-z_G}{x-\realoptx}_{\TauTest_i\Tau_i}
    	+\iprod{\subdiff F^*(y)-z_{F^*}}{y-\realopty}_{\SigmaTest_{i+1}\Sigma_{i+1}}
    	\\
    	+\iprod{\kgradconj{x} y-\kgradconj{\realoptx}\realopty}{x-\realoptx}_{\TauTest_i\Tau_i}
    	+\iprod{K(\realoptx)-K(x)}{y-\realopty}_{\SigmaTest_{i+1}\Sigma_{i+1}}.
    \end{multline*}
    Recalling the definition of $\tilde H_{i+1}(\nextu)$ in \eqref{eq:bregfun-def}, we therefore expand the third term of \eqref{eq:basic-deterministic-estimate} as
    \begin{align*}
    	\iprod{\tilde H_{i+1}&(\nextu)}{\nextu-\realoptu}_{\Step_{i+1}\Test_{i+1}}
    	\\
    	&=\iprod{\subdiff G(\nextx)-z_G}{\nextx-\realoptx}_{\TauTest_i\Tau_i}
    	+\iprod{\subdiff F^*(\nexty)-z_{F^*}}{\nexty-\realopty}_{\SigmaTest_{i+1}\Sigma_{i+1}}
    	\\
    	&\quad
    	+\iprod{\kgradconj{\nextx} \nexty-\kgradconj{\realoptx}\realopty}{\nextx-\realoptx}_{\TauTest_i\Tau_i}
    	+\iprod{K(\realoptx)-K(\nextx)}{\nexty-\realopty}_{\SigmaTest_{i+1}\Sigma_{i+1}}
    	\\
    	&\quad
    	+\iprod{[\kgrad{\thisx}-\kgrad{\nextx}]^*{\nexty}}{\nextx-\realoptx}_{\TauTest_{i}\Tau_{i}}
    	\\
    	&\quad
    	+\iprod{K(\nextx)-K(\nextx+\Omega^i(\nextx-\thisx))+\kgrad{\thisx}\Omega^i(\nextx-\thisx)}{\nexty-\realopty}_{\SigmaTest_{i+1}\Sigma_{i+1}}.
	\end{align*}
	Due to \cref{ass:gf} and \eqref{eq:ass-k}, we have
    \begin{align}
		\label{eq:digamma}
        D_i^\Gamma&\defeq
        \iprod{\subdiff G(\nextx)-z_G}{\nextx-\realoptx}_{\TauTest_i\Tau_i}
        +\norm{\nextx-\realoptx}_{\TauTest_{i}\Tau_{i}\Gamma_K}^2
        +\iprod{\subdiff F^*(\nexty)-z_{F^*}}{\nexty-\realopty}_{\SigmaTest_{i+1}\Sigma_{i+1}} \\
		&
		\nonumber
		\ge
        \sum_{j\in S(i)}\tauTest_j^{i} \tau_j^{i}\norm{\nextx-\realoptx}^2_{P_j\Gamma_{G}P_j}
        +\norm{\nextx-\realoptx}_{\TauTest_{i}\Tau_{i}\Gamma_K}^2
        +\sum_{\ell \in V(i+1)}\sigmaTest_\ell^{i+1} \sigma_\ell^{i+1} \norm{\nexty-\realopty}^2_{Q_\ell\Gamma_{F^*}Q_\ell},
        \allowdisplaybreaks
        \\
        \shortintertext{and}
		\label{eq:diomega}
        D_i^\Omega&\defeq
        \iprod{K(\thisx)-K(\nextx+\Omega^i(\nextx-\thisx))+\kgrad{\thisx}(\Omega^i+I)(\nextx-\thisx)}{\nexty-\realopty}_{\SigmaTest_{i+1}\Sigma_{i+1}}
		\\
		\nonumber
		&
		\ge -\frac{L}{2}\norm{\Omega^i+I}^2\norm{\SigmaTest_{i+1}^*\Sigma_{i+1}^*(\nexty-\realopty)}_{\Pnl}\norm{\nextx-\thisx}^2.
    \end{align}
    Hence, recalling $D_i^K$ from \eqref{eq:def-dik}, we deduce
    \begin{equation}
        \label{eq:penalty-general-d0}
		\begin{aligned}[t]
        \iprod{&\tilde H_{i+1}(\nextu)}{\nextu-\realoptu}_{\Step_{i+1}\Test_{i+1}}
		\\
		& =
    	\iprod{[\kgrad{\thisx}-\kgrad{\realoptx}]^*{\realopty}}{\nextx-\realoptx}_{\TauTest_{i}\Tau_{i}}
        -\norm{\nextx-\realoptx}_{\TauTest_{i}\Tau_{i}\Gamma_K}^2
    	\\
    	&\quad
    	+\iprod{K(\realoptx)-K(\thisx)-\kgrad{\thisx}(\realoptx-\thisx)}{\nexty-\realopty}_{\SigmaTest_{i+1}\Sigma_{i+1}}
    	\\
    	&\quad
    	+\iprod{\subdiff G(\nextx)-z_G}{\nextx-\realoptx}_{\TauTest_i\Tau_i}
    	+\norm{\nextx-\realoptx}_{\TauTest_{i}\Tau_{i}\Gamma_K}^2
    	+\iprod{\subdiff F^*(\nexty)-z_{F^*}}{\nexty-\realopty}_{\SigmaTest_{i+1}\Sigma_{i+1}}
    	\\
    	&\quad
    	+\iprod{K(\thisx)-K(\nextx+\Omega^i(\nextx-\thisx))+\kgrad{\thisx}(\Omega^i+I)(\nextx-\thisx)}{\nexty-\realopty}_{\SigmaTest_{i+1}\Sigma_{i+1}}
    	\\
    	&\quad
    	+\iprod{\kgradconj{\thisx}(\nexty-\realopty)}{\nextx-\realoptx}_{\TauTest_{i}\Tau_{i}}
    	-\iprod{\kgrad{\thisx}(\nextx-\realoptx)}{\nexty-\realopty}_{\SigmaTest_{i+1}\Sigma_{i+1}}
		\\
		&
		=D_i^K+D_i^\Gamma+D_i^\Omega+\iprod{\kgradconj{\thisx}(\nexty-\realopty)}{\nextx-\realoptx}_{\TauTest_{i}\Tau_{i}-\Sigma_{i+1}^*\SigmaTest_{i+1}^*}.
		\end{aligned}
    \end{equation}
	
	Inserting the lower bounds from \eqref{eq:sigmatest-result:copy}, \eqref{eq:digamma}, and \eqref{eq:diomega} into \eqref{eq:penalty-general-d0}, and using \eqref{eq:def-dik} and \eqref{eq:metric-update}, we obtain
    \begin{multline*}
    	\frac{1}{2}\norm{\nextu-\thisu}_{\Test_{i+1} \Precond_{i+1}}^2
    	+ \frac{1}{2}\norm{\nextu-\realoptu}_{\Test_{i+1}\Precond_{i+1}-\Test_{i+2}\Precond_{i+2}}^2
    	+ \iprod{\tilde H_{i+1}(\nextu)}{\nextu-\realoptu}_{\Step_{i+1}\Test_{i+1}}
    	\\
    	\ge%
		\frac{1}{2}\norm{\nextx-\thisx}_{\delta \TauTest_i}^2
		+\frac{1}{2}\frac{\kappa-\delta}{1-\delta}\norm{\nexty-\thisy}_{\SigmaTest_{i+1}}^2
    	+\frac{1}{2}\norm{\nextu-\realoptu}_{R'}^2
    	+D_i^\Lambda
    	+D_i^K
        \\
        -\frac{L}{2}\norm{\Omega^i+I}^2\norm{\SigmaTest_{i+1}^*\Sigma_{i+1}^*(\nexty-\realopty)}_{\Pnl}\norm{\nextx-\thisx}^2
    \end{multline*}
    for $D_i^\Lambda$ as in \eqref{eq:def-dilambda}.
    Finally, using the definitions of $R_x$ in \eqref{eq:penalty-general-vars}, we observe
    \[		\frac{1}{2}\norm{\nextx-\thisx}_{\delta \TauTest_i}^2
	    -L\norm{\Omega^i+I}^2\norm{\SigmaTest_{i+1}^*\Sigma_{i+1}^*(\nexty-\realopty)}_{\Pnl}\norm{\nextx-\thisx}^2
	    = \norm{\nextx-\thisx}_{R_x}^2.
    \]
    This yields the claim.
\end{proof}

\subsection{Expectation estimates}
\label{sec:penalty1}

To further estimate $D_i^K$ and $D_i^\Lambda$, we have to take the expectation with respect to $\SAlg_{i-1}$. We will use a split definition of the step lengths, writing
\[
    \this{\tau_j}=\left\{
    \begin{array}{ll}
    \this{\iset\tau_j},&j\in\iset S(i),\\
    \this{\dset\tau_j},&j\in S(i)\setminus\iset S(i),
    \end{array}
    \right.
    \quad\text{and}\quad
    \nexxt{\sigma_\ell}=\left\{
    \begin{array}{ll}
    \nexxt{\iset\sigma_\ell},&\ell\in\iset V(i+1),\\
    \nexxt{\dset\sigma_\ell},&\ell \in V(i+1)\setminus\iset V(i+1),
    \end{array}
    \right.
\]
where we make for all $i \in \N$ the conditionality assumptions
\begin{subequations}%
\label{eq:conditionality-assumptions}%
\begin{align}%
    \this\tauTest_j,\nexxt\sigmaTest_\ell & \in \Random(\SAlg_{i-1}; (0, \infty)),
    &
    \this{\iset\tau_j}, \this{\dset\tau_j}, \nexxt{\iset\sigma_\ell}, \nexxt{\dset\sigma_\ell} & \in \Random(\SAlg_{i-1}; (0, \infty)),
    \\
    \quad
    S(i), \iset S(i) & \in \Random(\SAlg_i; \powerset{\{1,\ldots,m\}}),
    \quad\text{and}
    &
    V(i+1), \iset V(i+1) & \in \Random(\SAlg_i; \powerset{\{1,\ldots,n\}}).
\end{align}%
\end{subequations}%
Thus $\this{\iset\tau_j}$ always refers to what $\this\tau_j$ would be if $j \in \iset S(i)$, and similarly for the other variables. Moreover, these step lengths are already known on iteration $i-1$, prior to their use. The only part that is not known about $\Tau_i$ and $\Sigma_{i+1}$ before commencing iteration $i$ are the subsets of blocks to be updated.
Observe that \eqref{eq:conditionality-assumptions} and \eqref{eq:alg-ordered} imply
\begin{equation}
    \label{eq:iterate-conditionality}
    \nextx \in \Random(\SAlg_i; X)
    \quad\text{and}\quad
    \nexty \in \Random(\SAlg_i; Y)
    \quad (i \in \N).
\end{equation}
Also, for brevity, we write
\begin{align*}
    \this\pi_j&\defeq\P[j\in S(i) \mid \SAlg_{i-1}],
    &
    \this{\iset\pi_j} &\defeq\P[j\in \iset{S}(i) \mid \SAlg_{i-1}],
    \\
    \nexxt\nu_\ell&\defeq\P[\ell\in V(i+1) \mid \SAlg_{i-1}],
    \quad\text{and}
    &
    \nexxt{\iset\nu_\ell}&\defeq\P[\ell\in \iset V(i+1) \mid \SAlg_{i-1}].
\end{align*}

\begin{lemma}
	\label{lemma:nonneg-penalty-dk}
	Suppose \cref{ass:k-nonlinear} and \eqref{eq:conditionality-assumptions} hold for some $L_3\ge0$, $p\in[1,2]$, and $\theta_{A}\ge0$. For some $\rho_\ell > 0$ assume
    \begin{equation}
        \label{eq:y-locality-prob-1}
        1=\P[\norm{\nexty_\ell-\realopty_\ell}_{\Pnl}\le\rho_\ell \mid \SAlg_{i-1}]
        \quad (\ell=1,\ldots,m).
    \end{equation}
	Then
	$D^K_i$ defined in \eqref{eq:def-dilambda} satisfies 
	for any $\zeta_\ell>0$ with
	$\sum_{\ell=1}^n\nu^{i+1}_\ell\sigmaTest_\ell^{i+1}\sigma_\ell^{i+1}\zeta_\ell^{1-p}\rho_\ell^{2-p}\le p^p\E_{i-1}[\theta_{\TauTest_i\Tau_i}]$%
	the lower bound
	\begin{equation}
		\label{eq:nonneg-penalty-dk}
		\begin{split}
		\E_{i-1}[D_i^K]
		&
		\ge
		-\frac{L_3}{2}\E_{i-1}[\norm{\nextx-\thisx}_{\TauTest_i\Tau_i}^2]
		\\
		\MoveEqLeft[-1]
		-\sum_{\ell=1}^n \E_{i-1}\left[\sigmaTest_\ell^{i+1}\sigma_\ell^{i+1}(p-1)\zeta_\ell\norm{y_\ell^{i+1}-\realopty_\ell}_{\Pnl}^2\right].
		\end{split}
	\end{equation}
\end{lemma}

\begin{proof}
    Setting $A=\TauTest_i\Tau_i$ in \cref{ass:k-nonlinear}, we obtain
    \begin{multline*}
        \iprod{[\kgrad{\thisx}-\kgrad{\realoptx}]^*\realopty}{\nextx-\realoptx}_{\TauTest_i\Tau_i}
        \\
        \ge
        \norm{\nextx-\realoptx}^2_{\TauTest_i\Tau_i\Gamma_K}
        +
        \theta_{\TauTest_i\Tau_i}\norm{K(\realoptx)-K(\thisx)-\kgrad{\thisx}(\realoptx-\thisx)}^p-\frac{L_3}{2}\norm{\nextx-\thisx}_{\TauTest_i\Tau_i}^2.
    \end{multline*}
    Therefore, recalling the definition of $D_k^K$ in \eqref{eq:def-dik} and using \eqref{eq:iterate-conditionality},
    \begin{equation}
        \label{eq:dk-expectation-oim2}
        \begin{aligned}[t]
        \E_{i-1}[D_i^K]&\ge
        \E_{i-1}[\theta_{\TauTest_i\Tau_i}]\norm{K(\realoptx)-K(\thisx)-\kgrad{\thisx}(\realoptx-\thisx)}^p
        -\frac{L_3}{2}\E_{i-1}\bigl[\norm{\nextx-\thisx}_{\TauTest_i\Tau_i}^2\bigr]
        \\
        \MoveEqLeft[-1]
        +\iprod{K(\realoptx)-K(\thisx)-\kgrad{\thisx}(\realoptx-\thisx)}{\E_{i-1}[\Sigma_{i+1}^*\SigmaTest_{i+1}^*(\nexty-\realopty)]}.
        \end{aligned}
    \end{equation}
    By Young's inequality and \eqref{eq:y-locality-prob-1} as in \cite[(3.16) and (3.17)]{tuomov-nlpdhgm-redo}, for any $\zeta_\ell>0$,
    \begin{equation*}
        \begin{split}
        \iprod{K(\realoptx)-K(\thisx)&-\kgrad{\thisx}(\realoptx-\thisx)}{\Sigma_{i+1}^*\SigmaTest_{i+1}^*(\nexty-\realopty)}
        \\
        &
        \ge
        -\sum_{\ell \in V(i+1)} \sigmaTest_\ell^{i+1}\sigma_\ell^{i+1}\norm{\nexty_\ell-\realopty_\ell}_{\Pnl} \cdot \norm{K(\realoptx)-K(\thisx)-\kgrad{\thisx}(\realoptx-\thisx)}
        \\
        &
        \ge
        -\sum_{\ell \in V(i+1)} \sigmaTest_\ell^{i+1}\sigma_\ell^{i+1}
           (p-1)\zeta_\ell\norm{y_\ell^{i+1}-\realopty_\ell}_{\Pnl}^2
           \\ \MoveEqLeft[-1]
            -\sum_{i=1}^n \frac{\chi_{V(i+1)}(\ell)\sigmaTest_\ell^{i+1}\sigma_\ell^{i+1}\norm{y_\ell^{i+1}-\realopty_\ell}_{\Pnl}^{2-p}}{p^p\zeta_\ell^{p-1}}
              \cdot \norm{K(\realoptx)-K(\thisx)-\kgrad{\thisx}(\realoptx-\thisx)}^p.
        \end{split}
    \end{equation*}
    Taking the expectation $\E_{i-1}$, applying the assumption $\sum_{\ell=1}^n\nu^{i+1}_\ell\sigmaTest_\ell^{i+1}\sigma_\ell^{i+1}\zeta_\ell^{1-p}\rho_\ell^{2-p}\le p^p\E_{i-1}[\theta_{\TauTest_i\Tau_i}]$, and inserting the result in \eqref{eq:dk-expectation-oim2}, we obtain the claim \eqref{eq:nonneg-penalty-dk}.
\end{proof}

\begin{lemma}
	\label{lemma:nonneg-penalty-dl}
	Suppose \cref{ass:k-lipschitz} and \eqref{eq:conditionality-assumptions} are satisfied for some $L\ge 0$, and the nesting conditions \eqref{eq:chilo-definition} hold for any $j$ and $\ell$ on both iterations $i$ and $i+1$.
	For some $\eta^{i+1}>0$ assume
    \begin{subequations}%
    \label{eq:etamu-update-lemma}%
	\begin{align}%
		\iset{\pi}^{i+1}_j\tauTest^{i+1}_j\iset\tau^{i+1}_j&=
		\eta^{i+1}-\chi_{S(i)\setminus \iset S(i)}(j)\tauTest^i_j\dset\tau^i_j,
        \\
		\iset\nu^{i+2}_\ell\sigmaTest_\ell^{i+2}\iset\sigma_\ell^{i+2}&=
		\eta^{i+1}-\chi_{V(i+1)\setminus \iset V(i+1)}(\ell)\sigmaTest_\ell^{i+1}\dset\sigma_\ell^{i+1}.
	\end{align}%
    \end{subequations}%
	Then $D^\Lambda_i$ defined in \eqref{eq:def-dilambda} satisfies for any given $\alpha_x,\alpha_y>0$ the lower bound
	\begin{equation}
        \label{eq:nonneg-penalty-dl}
        \begin{aligned}[t]
		\E_i[D^\Lambda_i]+\frac{\nexxt{d}}{2}\norm{\nextx-\thisx}^2
		&
		\ge
		-\alpha_x\sum_{j=1}^{m}\chi_{S(i)\setminus \iset S(i)}(j)\tauTest^i_j\dset\tau^i_j\norm{\nextx_j-\realoptx_j}^2
        \\ \MoveEqLeft[-1]
		-\alpha_y\sum_{\ell=1}^{n}\chi_{V(i+1)\setminus \iset V(i+1)}(\ell)\sigmaTest_\ell^{i+1}\dset\sigma_\ell^{i+1}\norm{\nexty_\ell-\realopty_\ell}_{\Pnl}^2,
        \end{aligned}
	\end{equation}
	where
	\begin{equation*}
		\nexxt{d}\defeq \frac{L^2}{2\alpha_x}
		\left(\sum_{j\in S(i)\setminus \iset S(i)}\tauTest^i_j\dset\tau^i_j\right)\norm{\nexty-\realopty}_{\Pnl}^2
		+\frac{L^2}{2\alpha_y}\left(\sum_{\ell\in V(i+1)\setminus \iset V(i+1)}\sigmaTest_\ell^{i+1}\dset\sigma_\ell^{i+1}\right)\norm{\nextx-\realoptx}^2.
	\end{equation*}

	Moreover, if
    \begin{gather}
        \label{eq:nonneg-penalty-dl-2-prob-ass}
        \P[\norm{\nextx-\realoptx}\le\rho_x,\, \norm{Q_\ell(\nexty-\realopty)}_{\Pnl}\le\rho_\ell,(\ell=1,\ldots,n) \mid \SAlg_{i-1}]=1,
        \\
        \shortintertext{then}
        \label{eq:nonneg-penalty-dl-2}
        \E_{i-1}[\nexxt{d}\norm{\nextx-\thisx}^2]\le\E_{i-1}[{\this c_*}\norm{\nextx-\thisx}^2]
    \shortintertext{for}
        \label{eq:lambda-lambda}
        \begin{aligned}[t]
		{\this c_*} \defeq \frac{L^2}{2\alpha_x\alpha_y}
		\Bigl(&\alpha_y\sum_{\ell=1}^{n}\rho_\ell^2 \#(S(i) \setminus \iset S(i))\max_{j=1,\ldots,m}\tauTest^i_j\dset\tau^i_j
        \\
        &
		+\alpha_x\rho_x^2 \#(V(i+1)\setminus\iset V(i+1))\max_{\ell=1,\ldots,n}\nexxt\sigmaTest_\ell\nexxt{\dset\sigma}_\ell\Bigr).
        \end{aligned}
	\end{gather}
\end{lemma}

\begin{proof}
    We recall from \eqref{eq:def-dilambda} that
    \[\begin{split}
	    D^\Lambda_i&\defeq
	    \iprod{\kgradconj{\thisx}(\nexty-\realopty)}{\nextx-\realoptx}_{\TauTest_{i}\Tau_{i}-\Sigma_{i+1}^*\SigmaTest_{i+1}^*}
	    \\
	    &\quad
	    +\adaptiprod{\biggl[\sum_{\ell \in V(i+2)}Q_\ell\kgrad{\nextx}\Tau_{i+1}^*\TauTest_{i+1}^*\iset{P}_{i+1}
	    	-\SigmaTest_{i+2}\Sigma_{i+2}\kgrad{\nextx}\dset{P}_{i+1}\biggr](\nextx-\realoptx)}{\nexty-\realopty}
	    \\
	    &\quad
	    -\adaptiprod{\biggl[\sum_{\ell \in V(i+1)}Q_\ell\kgrad{\thisx}\Tau_i^*\TauTest_i^*\iset{P}_i
	    	-\SigmaTest_{i+1}\Sigma_{i+1}\kgrad{\thisx}\dset{P}_i\biggr](\nextx-\realoptx)}{\nexty-\realopty}.
    \end{split}\]
	Defining for brevity
    \[
        k_{\ell,j}\defeq \iprod{\kgradconj{\thisx}(\nexty_\ell-\realopty_\ell)}{\nextx_j-\realoptx_j}
        \quad\text{and}\quad
        k_{\ell,j}^+\defeq \iprod{\kgradconj{\nextx}(\nexty_\ell-\realopty_\ell)}{\nextx_j-\realoptx_j},
    \]
    and using \eqref{eq:conditionality-assumptions}, which implies $\this\tauTest_j\this\tau_j,\nexxt\sigmaTest_\ell\nexxt\sigma_\ell \in \Random(\SAlg_i; (0, \infty))$, we expand
	\[
        \begin{split}
		\E_i[D^\Lambda_i]
        &
        =
		\sum_{\ell=1}^{n}\sum_{j=1}^{m}\Bigl[
            (\chi_{S(i)}(j)\tauTest^i_j\tau^i_j-\chi_{V(i+1)}(\ell)\sigmaTest_\ell^{i+1}\sigma_\ell^{i+1})k_{\ell,j}
    		\\
    		\MoveEqLeft[-4]
    		+\E_i[\chi_{V(i+2)}(\ell)(\chi_{\iset S(i+1)}(j)\tauTest_j^{i+1}\iset\tau_j^{i+1}
    		-\chi_{S(i+1)\setminus \iset S(i+1)}(j)\sigmaTest_\ell^{i+2}\sigma_\ell^{i+2})k_{\ell,j}^+]
            \\
            \MoveEqLeft[-4]
    		-\chi_{V(i+1)}(\ell)(\chi_{\iset S(i)}(j)\tauTest_j^i\iset\tau_j^i-\chi_{S(i)\setminus \iset S(i)}(j)\sigmaTest_\ell^{i+1}\sigma_\ell^{i+1})k_{\ell,j}
            \Bigr].
		\end{split}
    \]
    Writing in the first term $\chi_{S(i)}(j)\tauTest^i_j\tau^i_j=\chi_{\iset S(i)}(j)\tauTest_j^i\iset\tau_j^i+\chi_{S(i)\setminus \iset S(i)}(j)\tauTest^i_j\dset\tau^i_j$, this rearranges as
    \begin{equation*}
        \begin{aligned}[t]
        \E_i[D^\Lambda_i]
		&
        =
		\sum_{\ell=1}^{n}\sum_{j=1}^{m}\biggl(
		\Bigl[\chi_{S(i)\setminus \iset S(i)}(j)\tauTest^i_j\dset\tau^i_j
		+(1-\chi_{V(i+1)}(\ell))\chi_{\iset S(i)}(j)\tauTest^i_j\iset\tau^i_j
        \\
        \MoveEqLeft[-5]
		+\chi_{V(i+1)}(\ell)(\chi_{S(i)\setminus \iset S(i)}(j)-1)\sigmaTest_\ell^{i+1}\sigma_\ell^{i+1}\Bigr]k_{\ell,j}
		\\
		\MoveEqLeft[-4]
		+\E_i\bigl[\chi_{V(i+2)}(\ell)\chi_{\iset S(i+1)}(j)\tauTest^{i+1}_j\iset\tau^{i+1}_j
        \\
        \MoveEqLeft[-7]
		-\chi_{V(i+2)}(\ell)\chi_{S(i+1)\setminus \iset S(i+1)}(j)\sigmaTest_\ell^{i+2}\sigma_\ell^{i+2}\bigr]k_{\ell,j}^+\biggr).
    	\end{aligned}
    \end{equation*}
	Using \eqref{eq:chilo-definition}, we continue
	\[\begin{split}
		\E_i[D^\Lambda_i]&=
		\sum_{\ell=1}^{n}\sum_{j=1}^{m}\biggl(
    		[\chi_{S(i)\setminus \iset S(i)}(j)\tauTest^i_j\dset\tau^i_j
    		-\chi_{V(i+1)\setminus \iset V(i+1)}(\ell)\sigmaTest_\ell^{i+1}\dset\sigma_\ell^{i+1}]k_{\ell,j}
    		\\
    		\MoveEqLeft[-4]
    		+\E_i[\chi_{\iset S(i+1)}(j)\tauTest^{i+1}_j\iset\tau^{i+1}_j
    		-\chi_{\iset V(i+2)}(\ell)\sigmaTest_\ell^{i+2}\iset\sigma_\ell^{i+2}]k_{\ell,j}^+
        \biggr),
    \end{split}
    \]
    after which a use of \eqref{eq:etamu-update-lemma} rearranges this as
    \[\begin{split}
        \E_i[D^\Lambda_i]&=
		\sum_{\ell=1}^{n}\sum_{j=1}^{m}
		(\iset{\pi}^{i+1}_j\tauTest^{i+1}_j\iset\tau^{i+1}_j
		-\iset\nu^{i+2}_\ell\sigmaTest_\ell^{i+2}\iset\sigma_\ell^{i+2})(k_{\ell,j}^+-k_{\ell,j})
		\\
		&=\sum_{\ell=1}^{n}\sum_{j=1}^{m}
		(\chi_{S(i)\setminus \iset S(i)}(j)\tauTest^i_j\dset\tau^i_j
		-\chi_{V(i+1)\setminus \iset V(i+1)}(\ell)\sigmaTest_\ell^{i+1}\dset\sigma_\ell^{i+1})(k_{\ell,j}-k_{\ell,j}^+).
	\end{split}\]
	Expanding $k_{\ell,j}-k_{\ell,j}^+$, using \cref{ass:k-lipschitz}, and continuing with Young's inequality, for any $\alpha_x,\alpha_y>0$,
	\[\begin{split}
		\E_i[D^\Lambda_i]
		&=
		\sum_{\ell=1}^{n}\sum_{j=1}^{m}\bigl[
    		(\chi_{S(i)\setminus \iset S(i)}(j)\tauTest^i_j\dset\tau^i_j
    		-\chi_{V(i+1)\setminus \iset V(i+1)}(\ell)\sigmaTest_\ell^{i+1}\dset\sigma_\ell^{i+1})
            \\
            \MoveEqLeft[-4]
            \cdot\iprod{\nexty_\ell-\realopty_\ell}{[\kgrad{\thisx}-\kgrad{\nextx}](\nextx_j-\realoptx_j)}
        \bigr]
		\\
		&\ge
		-\sum_{j=1}^{m}\chi_{S(i)\setminus \iset S(i)}(j)\tauTest^i_j\dset\tau^i_j
		\cdot \norm{\nexty-\realopty}_{\Pnl}L\norm{\nextx-\thisx}\norm{\nextx_j-\realoptx_j}
		\\
		\MoveEqLeft[-1]
		-\sum_{\ell=1}^{n}\chi_{V(i+1)\setminus \iset V(i+1)}(\ell)\sigmaTest_\ell^{i+1}\dset\sigma_\ell^{i+1}
        \cdot \norm{\nexty_\ell-\realopty_\ell}_{\Pnl}L\norm{\nextx-\thisx}\norm{\nextx-\realoptx}
		\\
		&\ge
		-\sum_{j=1}^{m}\chi_{S(i)\setminus \iset S(i)}(j)\tauTest^i_j\dset\tau^i_j\left(\alpha_x\norm{\nextx_j-\realoptx_j}^2
		+\frac{L^2}{4\alpha_x}\norm{\nexty-\realopty}_{\Pnl}^2\norm{\nextx-\thisx}^2\right)
		\\
		\MoveEqLeft[-1]
		-\sum_{\ell=1}^{n}\chi_{V(i+1)\setminus \iset V(i+1)}(\ell)\sigmaTest_\ell^{i+1}\dset\sigma_\ell^{i+1}\left(\alpha_y\norm{\nexty_\ell-\realopty_\ell}_{\Pnl}^2
		+\frac{L^2}{4\alpha_y}\norm{\nextx-\thisx}^2\norm{\nextx-\realoptx}^2\right).
	\end{split}\]
	This rearranges as \eqref{eq:nonneg-penalty-dl}.
    By \eqref{eq:nonneg-penalty-dl-2-prob-ass}, $\P[\nexxt{d}\le {\this c_*} \mid \SAlg_{i-1}]=1$. Hence \eqref{eq:nonneg-penalty-dl-2} follows.
\end{proof}

\begin{remark}
    For slightly stronger results, it would in \eqref{eq:nonneg-penalty-dl-2-prob-ass} and throughout the rest of the manuscript, be possible to take $\rho_x=\nexxt\rho_x$ and $\rho_\ell=\nexxt\rho_\ell$ dependent on the iteration.
\end{remark}

\subsection{Putting it all together}
\label{sec:penalty2}

We are now ready to state our main generic result providing the tool to estimate convergence rates based on growth rates of $\tauTest_j^i$ and $\sigmaTest_\ell^{i+1}$. 

\begin{theorem}
	\label{thm:test-det}
	Suppose \cref{ass:k-lipschitz,ass:k-nonlinear,ass:gf} hold for some $0<\delta\le\kappa<1$, $\gamma_{G,j},\gamma_{F^*,\ell}\ge0$, $\gamma_{K,j}\in\R$ ($j=1,\ldots,m$, $\ell=1,\ldots,n$), $L,L_3\ge0$, $p\in[1,2]$, $\theta_A\ge 0$ together with the nesting conditions \eqref{eq:chilo-definition}, the lower bound \eqref{eq:sigmatest-result} on the local metric, and the conditionality assumptions \eqref{eq:conditionality-assumptions} for all $i\le N-1$.
	For some sequence of $\eta^{i+1}>0$ assume the \term{coupling conditions}
	\begin{subequations}%
		\label{eq:etamu-update-deter}%
		\begin{align}
		\iset{\pi}^{i+1}_j\tauTest^{i+1}_j\iset\tau^{i+1}_j
		+\chi_{S(i)\setminus \iset S(i)}(j)\tauTest^i_j\dset\tau^i_j
		&=\nexxt\eta
		\quad (j=1,\ldots,m)
		\quad\text{and}
		\\
		\iset\nu^{i+2}_\ell\sigmaTest_\ell^{i+2}\iset\sigma_\ell^{i+2}
		+\chi_{V(i+1)\setminus \iset V(i+1)}(\ell)\sigmaTest_\ell^{i+1}\dset\sigma_\ell^{i+1}
		&=\nexxt\eta
		\quad (\ell=1,\ldots,n).
		\end{align}%
	\end{subequations}%
	Also assume for some $\rho_x,\rho_\ell\ge0$ and $\zeta_\ell\ge0$,
	\begin{subequations}%
		\label{eq:rules-det}%
		\begin{align}%
		\label{eq:bounded-dual-deter}
		&1=\P[\norm{\nextx-\realoptx}\le\rho_x, \norm{Q_\ell(\nexty-\realopty)}_{\Pnl}\le\rho_\ell, (\ell=1,\ldots,n) \mid \SAlg_{i-1}]
		\quad\text{and}
		\\
		\label{eq:zeta-rule-deter}
		&\E_{i-1}[\theta_{\TauTest_i\Tau_i}]\ge p^{-p}\textstyle\sum_{\ell=1}^n\nu^{i+1}_\ell\nexxt\sigmaTest_\ell\nexxt\sigma_\ell\zeta_\ell^{1-p}\rho_\ell^{2-p}
		\quad (\ell=1,\ldots,n).
		\end{align}%
	\end{subequations}%
	Finally, for $\this c_*$ defined in \eqref{eq:lambda-lambda} for some $\alpha_x,\alpha_y>0$ let
	\begin{align}
		\label{eq:lij}
		L^i_j
		& \defeq
		L_3+(L\norm{\Omega^i+I}^2\textstyle\sum_{\ell=1}^{m}\sigmaTest^{i+1}_\ell\sigma^{i+1}_\ell\rho_\ell+\this c_*)/\tauTest^i_j\tau^i_j,
		\\
        \label{eq:bar-gamma-g}
		\bar\gamma_{GK,j}^i&\defeq
		\gamma_{G,j}+\gamma_{K,j}-\chi_{S(i) \setminus \iset S(i)}(j)\alpha_x,
		\\
		\label{eq:bar-gamma-f}
		\bar\gamma_{F^*,\ell}^{i+1}&\defeq
		\begin{cases}
		\gamma_{F^*,\ell},& Q_\ell\Pnl=0,\\
		\gamma_{F^*,\ell}-(p-1)\zeta_\ell-\chi_{V(i+1)\setminus \iset V(i+1)}(\ell)\alpha_y, & Q_\ell\Pnl\ne 0.
		\end{cases}
	\end{align}
    Then
    \begin{multline}
	    \label{eq:convergence-estimate-blockwise}	    \delta\sum_{j=1}^{m}\E\left[\tauTest_j^i \norm{P_j(x^N-\realoptx)}^2\right]
	    +\frac{\kappa-\delta}{1-\delta}\sum_{\ell=1}^{n}\E\left[ \sigmaTest_\ell^{i+1} \norm{Q_\ell(y^N-\realopty)}^2\right]
	    \\	    \le \E[\norm{u^N-\realoptu}^2_{\Test_{N+1}\Precond_{N+1}}]
	    \le \E[\norm{u^0-\realoptu}^2_{\Test_{1}\Precond_{1}}]
    \end{multline}
	holds provided for every $i\le N-1$ both \cref{item:thm-test:primal} and \cref{item:thm-test:dual} are true:
	\begin{enumerate}[label=(\roman*)]
		\item\label{item:thm-test:primal}
		Either of the \term{primal test update conditions} holds for every $j=1,\ldots,m$:
		\begin{enumerate}[label=(\alph*)]
			\item\label{item:thm-test:primal:a}
			both $\tauTest^{i+1}_j
                \le(1+2\chi_{S(i)}(j)\tau_j^{i}\bar\gamma_{GK,j}^i)\tauTest^i_j$
                and $\delta \ge
                \chi_{S(i)}(j)L^i_j\tau^i_j$; or
			\item\label{item:thm-test:primal:b}
			for some $\tilde\gamma_{G,j}^i\in\Random(\SAlg_{i-1},\R)$, $\tilde\tau^i_j\defeq(\iset\pi^i_j\iset\tau^i_j+(\pi^i_j-\iset\pi^i_j)\dset\tau^i_j)/\pi^i_j$,
			\begin{subequations}
				\label{eq:tautest-update-tilde}
				\begin{align}
				\label{eq:tautest-update-deter}
                &
				\tauTest^{i+1}_j=
				(1+2\tilde\tau_j^{i}\tilde\gamma_{G,j}^i)\tauTest^i_j,
				\quad
				\tilde\tau^i_j\tilde\gamma_{G,j}^i<
				\E_{i-1}[\chi_{S(i)}(j)\tau^i_j\bar\gamma_{GK,j}^i],
				\quad\text{and}
				\\
				\label{eq:tau-bound-deter}
                &
				\delta\ge
				\chi_{S(i)}(j)\left(
				L^i_j\tau_j^{i}+
				\frac{2(\tau^i_j\bar\gamma_{GK,j}^i-\E_{i-1}[\chi_{S(i)}(j)\tau^i_j\bar\gamma_{GK,j}^i])(\tau^i_j\bar\gamma_{GK,j}^i-\tilde\tau^i_j\tilde\gamma_{G,j}^i)}
				{\E_{i-1}[\chi_{S(i)}(j)\tau^i_j\bar\gamma_{GK,j}^i]-\tilde\tau^i_j\tilde\gamma_{G,j}^i}
				\right).
				\end{align}
			\end{subequations}
		\end{enumerate}

		\item\label{item:thm-test:dual}
		Either of the \term{dual test update conditions} holds for every $\ell=1,\ldots,n$:
		\begin{enumerate}[label=(\alph*)]
			\item\label{item:thm-test:dual:a}
            $\nexxt\sigmaTest_\ell \le(1+2\chi_{V(i+1)}(\ell)\sigma_\ell^{i+1}\bar\gamma_{F^*,\ell}^{i+1})\sigmaTest^{i+1}_\ell$; or
			\item\label{item:thm-test:dual:b}
			for some $\tilde\gamma_{F^*,\ell}^{i+1}\in\Random(\SAlg_{i-1},\R)$, $\tilde\sigma^{i+1}_\ell\defeq(\iset\nu^{i+1}_\ell\iset\sigma^{i+1}_\ell+(\nu^{i+1}_\ell-\iset\nu^{i+1}_\ell)\dset\sigma^{i+1}_\ell)/\nu^{i+1}_\ell$:
            \begin{subequations}
			\begin{align}
                \label{eq:sigmatest-update-tilde}
                &
    			\sigmaTest^{i+2}_\ell=(1+2\tilde\sigma_\ell^{i+1}\tilde\gamma_{F^*,\ell}^{i+1})\sigmaTest^{i+1}_\ell,
    			\quad
    			\tilde\sigma^{i+1}_\ell\tilde\gamma_{F^*,\ell}^{i+1}< \E_{i-1}[\chi_{V(i+1)}(\ell)\sigma^{i+1}_\ell\bar\gamma_{F^*,\ell}^{i+1}],
    			\\
    			\label{eq:sigma-bound-deter}
                &\begin{aligned}[t]
    			\frac{\kappa-\delta}{1-\delta} & \ge
    			2(\sigma_\ell^{i+1}\bar\gamma_{F^*,\ell}^{i+1}-\E_{i-1}[\chi_{V(i+1)}(\ell)\sigma_\ell^{i+1}\bar\gamma_{F^*,\ell}^{i+1}])
                \\
                \MoveEqLeft[-8]
                \cdot
    			\frac{\chi_{V(i+1)}(\ell)(\sigma_\ell^{i+1}\bar\gamma_{F^*,\ell}^{i+1}-\tilde\sigma_\ell^{i+1}\tilde\gamma_{F^*,\ell}^{i+1})}
    			{\E_{i-1}[\chi_{V(i+1)}(\ell)\sigma_\ell^{i+1}\bar\gamma_{F^*,\ell}^{i+1}]-\tilde\sigma_\ell^{i+1}\tilde\gamma_{F^*,\ell}^{i+1}}.
                \end{aligned}
			\end{align}
            \end{subequations}
		\end{enumerate}
	\end{enumerate}
\end{theorem}

\begin{proof}
	We first apply \cref{lemma:penalty-general}.
	Recalling $R'$ from \eqref{eq:penalty-general-vars-rprime}, let us set
	\begin{equation}
		\label{eq:rprimeprime}
		\begin{aligned}[t]
		R'' &   \defeq R' -2
		\begin{psmallmatrix}
		\sum_{j=1}^{m}\tau_j^{i}\tauTest^i_j\chi_{S(i)\setminus \iset S(i)}(j)\alpha_x P_j & 0
		\\
		0 & \sum_{\ell=1}^{n}\sigma_\ell^{i+1}\sigmaTest^{i+1}_\ell
		(\chi_{V(i+1)}(\ell)(p-1)\zeta_\ell+\chi_{V(i+1)\setminus \iset V(i+1)}(\ell)\alpha_y)Q_\ell\Pnl
		\end{psmallmatrix}
		\\
		&
		=
		\begin{psmallmatrix}
		\TauTest_i-\TauTest_{i+1}+2\sum_{j\in S(i)}\tauTest_j^{i} \tau_j^{i}\bar\gamma_{GK,j}P_j & 0
		\\
		0 & \SigmaTest_{i+1}-\SigmaTest_{i+2}+2\sum_{\ell \in V(i+1)}\sigmaTest_\ell^{i+1} \sigma_\ell^{i+1}\bar \gamma_{F^*,\ell}Q_\ell
		\end{psmallmatrix}
		=
		\begin{psmallmatrix}
		\sum_{j=1}^m q_j^i P_j & 0 \\
		0 & \sum_{\ell=1}^n h_\ell^{i+1} Q_\ell
		\end{psmallmatrix}
		\end{aligned}
	\end{equation}
	for
	\begin{align*}
		q_j^i
		\defeq
		(1+2\chi_{S(i)}(j) \tau_j^{i}\bar\gamma_{GK,j}^i)\tauTest^i_j-\tauTest^{i+1}_j
		\quad\text{and}\quad
		h_\ell^{i+1}
		\defeq
		(1+2\chi_{V(i+1)}(\ell) \sigma_\ell^{i+1}\bar\gamma_{F^*,\ell}^{i+1})\sigmaTest^{i+1}_\ell-\sigmaTest^{i+2}_\ell.
	\end{align*}
	Thus
	\begin{equation}
		\label{eq:rprimeprime-expansion1}
		\E_{i-1}[\norm{\nextu-\realoptu}^2_{R''}]=
		\sum_{j=1}^{m}\E_{i-1}[q_j^i\norm{P_j(\nextx-\realoptx)}^2]
		+\sum_{\ell=1}^{n}\E_{i-1}[h_\ell^{i+1}\norm{Q_\ell(\nexty-\realopty)}^2].
	\end{equation}

	\paragraph{Estimation of $q_j^i$}
	Suppose $j \in \{1,\ldots,m\}$ satisfies \cref{item:thm-test:primal}\ref{item:thm-test:primal:a}.
	Then $q_j^i\ge0$ and $\delta\ge \chi_{S(i)}(j)L^i_j\tau^i_j$, so we immediately estimate
	\begin{equation}
		\label{eq:qi-estimate}
		\E_{i-1}[q_j^i\norm{P_j(\nextx-\realoptx)}^2]
		\ge
		-\E_{i-1}[\chi_{S(i)}(j)(\delta\tauTest^i_j-L^i_j\tauTest^i_j\tau^i_j)\norm{P_j(\nextx-\thisx)}^2].
	\end{equation}
	Otherwise, if $j \in \{1,\ldots,m\}$ satisfies \cref{item:thm-test:primal}\ref{item:thm-test:primal:b}, using \eqref{eq:iterate-conditionality} and that $q_j^i=\E_i[q_j^i]$ due to \eqref{eq:conditionality-assumptions} and \eqref{eq:bar-gamma-g}, we decompose
	\[
		\begin{split}
		\E_{i-1}[q_j^i\norm{P_j(\nextx-\realoptx)}^2]=\E_{i-1}\bigl[&
		q_j^i\norm{P_j(\nextx-\thisx)}^2
		+\E_{i-1}[q_j^i]\norm{P_j(\thisx-\realoptx)}^2
		\\
		&
		+2q_j^i\iprod{P_j(\nextx-\thisx)}{\thisx-\realoptx}\bigr].
		\end{split}
	\]
	Using $(1-\chi_{S(i)}(j))P_j(\nextx-\thisx)=0$ and Young's inequality with the factor $\alpha>0$, we obtain
	\begin{equation}
		\label{eq:qji-expectation-est2}
		\begin{aligned}[t]
		\E_{i-1}\bigl[q_j^i\norm{P_j(\nextx-\realoptx)}^2]\ge\E_{i-1}[
		&\chi_{S(i)}(j)(q_j^i-\alpha\abs{q_j^i})\norm{P_j(\nextx-\thisx)}^2
		\\
		&
		+(\E_{i-1}[q_j^i]-\chi_{S(i)}(j)\alpha^{-1}\abs{q_j^i})\norm{P_j(\thisx-\realoptx)}^2\bigr].
		\end{aligned}
	\end{equation}
	Since $\tauTest^{i+1}_j=(1+2\tilde\tau_j^{i}\tilde\gamma_{G,j}^i)\tauTest^i_j$ with $\tilde\gamma_{G,j}^i\in \Random(\SAlg_{i-1};\R)$, we have from \eqref{eq:tautest-update-deter}
	\begin{align*}
		\E_{i-1}[q_j^i]& =
		(1+2\E_{i-1}[\chi_{S(i)}(j)\tau_j^{i}\bar\gamma_{GK,j}^i])
		\tauTest^i_j-\E_{i-1}[\tauTest^{i+1}_j]
		=
		2\tauTest^i_j(\E[\chi_{S(i)}(j)\tau_j^{i}\bar\gamma_{GK,j}^i]-\tilde\tau_j^{i}\tilde\gamma_{G,j}^i)
		> 0,
		\\
		\shortintertext{and rearranging \eqref{eq:tau-bound-deter} for $j \in S(i)$:}
		q_j^i &=
		2\tauTest^i_j(\chi_{S(i)}(j)\tau_j^{i}\bar\gamma_{GK,j}^i-\tilde\tau_j^{i}\tilde\gamma_{G,j}^i) \ge (\E_{i-1}[q_j^i])^{-1}\abs{q_j^i}^2-\delta\tauTest^i_j+L^i_j\tauTest^i_j\tau^i_j.
	\end{align*}
	Therefore, taking $\alpha \defeq (\E_{i-1}[q_j^i])^{-1}\abs{q_j^i}$ for $j\in S(i)$ in \eqref{eq:qji-expectation-est2}, we verify \eqref{eq:qi-estimate} for the case \cref{item:thm-test:primal}\ref{item:thm-test:primal:b} as well.

	\paragraph{Estimation of $h_\ell^{i+1}$}
	Similarly, if $\ell \in \{1,\ldots,n\}$ satisfies \cref{item:thm-test:dual}\ref{item:thm-test:dual:a}, we have $h_\ell^{i+1}\ge0$, hence
	\begin{equation}
		\label{eq:hell-estimate}
		\E_{i-1}[h_\ell^{i+1}\norm{Q_\ell(\nexty-\realopty)}^2]\ge
		-\E_{i-1}\Bigl[\chi_{V(i+1)}(\ell)\frac{\kappa-\delta}{1-\delta}\sigmaTest^{i+1}_\ell\norm{Q_\ell(\nexty-\thisy)}^2\Bigr].
	\end{equation}
	Otherwise, when $\ell \in \{1,\ldots,n\}$ satisfies \cref{item:thm-test:dual}\ref{item:thm-test:dual:b}, using \eqref{eq:iterate-conditionality} and that $h_\ell^{i+1}=\E_i[h_\ell^{i+1}]$ due to \eqref{eq:conditionality-assumptions} and \eqref{eq:bar-gamma-f}, we estimate for arbitrary $\alpha>0$ that
	\begin{equation}
		\label{eq:helli-expectation-est2}
		\begin{aligned}[t]
		\E_{i-1}[h_\ell^{i+1} \norm{Q_\ell(\nexty-\realopty)}^2]
		\ge\E_{i-1}[&
		\chi_{V(i+1)}(\ell) (h_\ell^{i+1} -\alpha\abs{h_\ell^{i+1} })\norm{Q_\ell(\nexty-\thisy)}^2
		\\
		&+(\E_{i-1}[h_\ell^{i+1}]-\chi_{V(i+1)}(\ell)\alpha^{-1}\abs{h_\ell^{i+1}})\norm{Q_\ell(\thisy-\realopty)}^2].
		\end{aligned}
	\end{equation}
	Since $\sigmaTest^{i+2}_\ell=(1+2\tilde\sigma_\ell^{i+1}\tilde\gamma_{F^*,\ell}^{i+1})\sigmaTest^{i+1}_\ell$ with $\tilde\gamma_{F^*,\ell}^{i+1}\in \Random(\SAlg_{i-1};\R)$, from \eqref{eq:sigmatest-update-tilde} we have
	\begin{align*}
		\E_{i-1}[h_\ell^{i+1}]&=
		(1+2\E_{i-1}[\chi_{V(i+1)}(\ell)\sigma_\ell^{i+1}\bar\gamma_{F^*,\ell}^{i+1}])
		\sigmaTest^{i+1}_\ell-\E_{i-1}[\sigmaTest^{i+2}_\ell]>0
		\\
		\shortintertext{and rearranging \eqref{eq:sigma-bound-deter} for $\ell \in V(i+1)$:}
		h_\ell^{i+1} &\ge
		(\E_{i-1}[h_\ell^{i+1}])^{-1}\abs{h_\ell^{i+1}}^2-\frac{\kappa-\delta}{1-\delta}\sigmaTest^{i+1}_\ell.
	\end{align*}
	Consequently, taking $\alpha \defeq (\E_{i-1}[h_\ell^{i+1}])^{-1}\abs{h_\ell^{i+1}}$ for $\ell\in V(i+1)$  in \eqref{eq:helli-expectation-est2}, we obtain \eqref{eq:hell-estimate} for the case \cref{item:thm-test:dual}\ref{item:thm-test:dual:b} as well.

	\paragraph{Combining the estimates}
	Since \eqref{eq:qi-estimate} and \eqref{eq:hell-estimate} hold for all $j=1,\ldots,m$ and $\ell=1,\ldots,n$, respectively, continuing from \eqref{eq:rprimeprime-expansion1}, we get
	\[
		\begin{split}
		\E_{i-1}[\norm{\nextu-\realoptu}^2_{R''}]
		&\ge-\E_{i-1}\Bigl[
		\sum_{j=1}^{m}(\chi_{S(i)}(j)(\delta\tauTest^i_j-L^i_j\tauTest^i_j\tau^i_j)\norm{P_j(\nextx-\thisx)}^2
		\\
		\MoveEqLeft[-3]
		+\sum_{\ell=1}^{n}\Bigl(\chi_{V(i+1)}(\ell)\frac{\kappa-\delta}{1-\delta}\sigmaTest^{i+1}_\ell\norm{Q_\ell(\nexty-\thisy)}^2\Bigr)\Bigr].
		\end{split}
    \]
    Plugging $L^i_j$ from \eqref{eq:lij}, thus
    \begin{multline*}
        \E_{i-1}[\norm{\nextu-\realoptu}^2_{R''}]
		\ge
	    -\E_{i-1}\Biggl[
    	\sum_{j=1}^{m}\chi_{S(i)}(j)\left(\delta\tauTest^i_j-L\norm{\Omega^i+I}^2\sum_{\ell=1}^{m}\sigmaTest^{i+1}_\ell\sigma^{i+1}_\ell\rho_\ell\right)\norm{P_j(\nextx-\thisx)}^2
        \\
        +\sum_{\ell=1}^{n}\Bigl(\chi_{V(i+1)}(\ell)\frac{\kappa-\delta}{1-\delta}\sigmaTest^{i+1}_\ell\norm{Q_\ell(\nexty-\thisy)}^2\Bigr)
        -\sum_{j=1}^{m}(\chi_{S(i)}(j)(L_3\tauTest^i_j\tau^i_j+\this c_*)\norm{P_j(\nextx-\thisx)}^2\Biggr].
	\end{multline*}
	By the definitions of $R_x$ in \eqref{eq:penalty-general-vars} and $\rho_\ell$ in \eqref{eq:bounded-dual-deter}, we continue
	\begin{equation}
		\label{eq:test-det-est1}
		\begin{aligned}[t]
		\E_{i-1}[\norm{\nextu-\realoptu}^2_{R''}]
		& \ge
		-\E_{i-1}\Bigl[\norm{\nextx-\thisx}_{R_x}^2+\frac{\kappa-\delta}{1-\delta}\norm{\nexty-\thisy}_{\SigmaTest_{i+1}}^2
		\\
		\MoveEqLeft[-4]
		-\sum_{j=1}^{m}\chi_{S(i)}(j)(L_3\tauTest^i_j\tau^i_j+\this c_*)\norm{P_j(\nextx-\thisx)}^2\Bigr].
		\end{aligned}
	\end{equation}

	On the other hand, by the definition of $R''$ in \eqref{eq:rprimeprime},
	\begin{multline*}
		\E_{i-1}[\norm{\nextu-\realoptu}^2_{R''}]=
		\E_{i-1}\Bigl[\norm{\nextu-\realoptu}^2_{R'}
		-2\alpha_x\sum_{j=1}^{m}\tau_j^{i}\tauTest^i_j\chi_{S(i)\setminus \iset S(i)}(j)\norm{P_j(\nextx-\realoptx)}^2
		\\
		-2\sum_{\ell=1}^{n}(\chi_{V(i+1)}(\ell)(p-1)\zeta_\ell+\chi_{V(i+1)\setminus \iset V(i+1)}(\ell)\alpha_y)\sigma_\ell^{i+1}\sigmaTest^{i+1}_\ell\norm{Q_\ell(\nexty-\realopty)}^2_{\Pnl}\Bigr].
	\end{multline*}
	Combining with \eqref{eq:test-det-est1} and rearranging the terms, we therefore have
	\begin{equation}
		\label{eq:test-dest-est3}
		\E_{i-1}[\norm{\nextu-\realoptu}^2_{R'}
		+\norm{\nextx-\thisx}_{R_x}^2
		+\frac{\kappa-\delta}{1-\delta}\norm{\nexty-\thisy}_{\SigmaTest_{i+1}}^2]
		\ge
		\E_{i-1}[b_1 + b_2]
	\end{equation}
	for
	\begin{align*}
		b_1&\defeq
		\sum_{j=1}^n\chi_{S(i)}(j)
		L_3\tauTest^i_j\tau^i_j \norm{P_j(\nextx-\thisx)}^2
		+2\sum_{\ell=1}^{n}
		\sigma_\ell^{i+1}\sigmaTest^{i+1}_\ell \chi_{V(i+1)}(\ell)(p-1)\zeta_\ell   \norm{Q_\ell(\nexty-\realopty)}^2_{\Pnl},
		\\
		\shortintertext{and}
		b_2&\defeq
		2\alpha_x\sum_{j=1}^{m}
		\tau_j^{i}\tauTest^i_j\chi_{S(i)\setminus \iset S(i)}(j)
		\norm{P_j(\nextx-\realoptx)}^2
		\\\MoveEqLeft[-1]
		+2\alpha_y\sum_{\ell=1}^{n}
		\sigma_\ell^{i+1}\sigmaTest^{i+1}_\ell\chi_{V(i+1)\setminus \iset V(i+1)}(\ell)
		\norm{Q_\ell(\nexty-\realopty)}^2_{\Pnl}
		+\sum_{j=1}^n\chi_{S(i)}(j){\this c_*}\norm{P_j(\nextx-\thisx)}^2.
	\end{align*}
	Our conditions \eqref{eq:rules-det} and $\delta\ge \chi_{S(i)} (j)L^i_j\tau^i_j$ ensure the conditions of \cref{lemma:nonneg-penalty-dk,lemma:nonneg-penalty-dl}.
	By \cref{lemma:nonneg-penalty-dk} thus $\E_{i-1}[b_1+2D_i^K] \ge 0$ while using both \eqref{eq:nonneg-penalty-dl} and \eqref{eq:nonneg-penalty-dl-2} of \cref{lemma:nonneg-penalty-dl} establishes $\E_{i-1}[b_2+2D^\Lambda_i] = \E_{i-1}[b_2+2\E_i[D^\Lambda_i]] \ge 0$.
	Consequently \eqref{eq:test-dest-est3} yields
	\[
    	\E_{i-1}[\norm{\nextu-\realoptu}_{R'}^2+\norm{\nextx-\thisx}_{R_x}^2+
    	\frac{\kappa-\delta}{1-\delta}\norm{\nexty-\thisy}_{\SigmaTest_{i+1}}^2+2D^\Lambda_i + 2D^K_i]
    	\ge 0.
	\]
	We now use \cref{lemma:penalty-general} to verify \eqref{eq:convergence-fundamental-condition-iter-h-stoch}.
	Minding that each $\Test_{i+1}\Precond_{i+1}$ is self-adjoint by \cref{lemma:sigmatest}, a referral to \cref{thm:convergence-result-main-h-stoch} establishes \eqref{eq:convergence-result-main-h-stoch}.
	Using \eqref{eq:sigmatest-result} as well as $\tauTest_j^N,\sigmaTest_\ell^{N+1} \in \Random(\SAlg_{N-1}; (0, \infty))$ and $u^N \in \Random(\SAlg_{N-1}; X \times Y)$ that follow from \eqref{eq:conditionality-assumptions}, we estimate
	\begin{multline*}
		\E[\norm{u^N-\realoptu}^2_{\Test_{N+1}\Precond_{N+1}} \mid \SAlg_{N-1}]
		=
		\norm{u^N-\realoptu}^2_{\E[\Test_{N+1}\Precond_{N+1} \mid \SAlg_{N-1}]}
		\\
		\ge
		\delta\sum_{j=1}^{m} \tauTest_j^N \norm{P_j(x^N-\realoptx)}^2
		+\frac{\kappa-\delta}{1-\delta}\sum_{\ell=1}^{n} \sigmaTest_\ell^{N+1} \norm{Q_\ell(y^N-\realopty)}^2.
	\end{multline*}
	Taking the full expectation and using \eqref{eq:convergence-result-main-h-stoch}
	 establishes the claim.
\end{proof}

\begin{remark}
	The conditions \cref{item:thm-test:primal}\ref{item:thm-test:primal:a} and \cref{item:thm-test:dual}\ref{item:thm-test:dual:a} differ from \cref{item:thm-test:primal}\ref{item:thm-test:primal:b} and \cref{item:thm-test:dual}\ref{item:thm-test:dual:b} by larger $\bar\gamma_{GK,j}^i$ and $\bar\gamma_{F^*,\ell}^i$, and updating $\tauTest^{i+1}_j$ and $\sigmaTest^{i+2}_\ell\in\Random(\SAlg_i;\R)$ potentially non-deterministically.

	In \cref{sec:fulldual} we have $\iset\pi^i_j=\pi^i_j$, $\tau^i_j=\iset\tau^i_j$, $\iset\nu^{i+1}_\ell=0$, and $\sigma^{i+1}_\ell=\dset\sigma^{i+1}_\ell$. In \cref{sec:fullprimal} we take $\iset\pi^i_j=0$, $\tau^i_j=\dset\tau^i_j$, $\iset\nu^{i+1}_\ell=\nu^{i+1}_\ell$, and $\sigma^{i+1}_\ell=\iset\sigma^{i+1}_\ell$.
    Also \cref{item:thm-test:primal}\ref{item:thm-test:primal:b} and \cref{item:thm-test:dual}\ref{item:thm-test:dual:b} then simplify
	for $\tilde\gamma_{G,j}^i<\pi^i_j\bar\gamma_{GK,j}^i$ to
	\begin{subequations}
		\begin{align}
			\label{eq:tautest-update-tilde-eps}
			\tauTest^{i+1}_j&=
			(1+2\tau_j^i\tilde\gamma_{G,j}^i)\tauTest^i_j,
			&
			\delta&\ge
			\chi_{S(i)}(j)\tau_j^{i}\left(L^i_j+2(1-\pi_j^i)\bar\gamma_{GK,j}^i\frac{\bar\gamma_{GK,j}^i-\tilde\gamma_{G,j}^i}{\pi^i_j\bar\gamma_{GK,j}^i-\tilde\gamma_{G,j}^i}\right)
			\\
            \shortintertext{and, respectively, for $\tilde\gamma_{F^*,\ell}^{i+1}<\nu^{i+1}_\ell\bar\gamma_{F^*,\ell}^{i+1}$ to}
			\label{eq:sigmatest-update-tilde-eps}
			\sigmaTest^{i+2}_\ell&=(1+2\sigma_\ell^{i+1}\tilde\gamma_{F^*,\ell}^{i+1})\sigmaTest^{i+1}_\ell,
			&
			\frac{\kappa-\delta}{1-\delta}&\ge
			2\chi_{V(i+1)}(\ell)(1-\nu^{i+1}_\ell)\sigma_\ell^{i+1}\bar\gamma_{F^*,\ell}^{i+1}
			\frac{\bar\gamma_{F^*,\ell}^{i+1}-\tilde\gamma_{F^*,\ell}^{i+1}}
			{\nu^{i+1}_\ell\bar\gamma_{F^*,\ell}^{i+1}-\tilde\gamma_{F^*,\ell}^{i+1}}.
		\end{align}
	\end{subequations}
\end{remark}

\begin{remark}
	\label{rem:boundedness-assumptions}
	Another quite restrictive requirement that we will need in the next sections is the almost sure boundedness of the iterates in \eqref{eq:bounded-dual-deter}. We already had this requirement in the deterministic single-block algorithm in \cite[Section 4.3]{tuomov-nlpdhgm-redo} and \cite[Section 5]{tuomov-nlpdhgm-general}. We verified in \cite[Proposition 4.8.]{tuomov-nlpdhgm-redo} that this requirement can be restated in terms of the sufficiently close initialisation of iterations to the critical point, which is often required in non-convex optimisation.
	
	In this work, the rates for convergence are in expectation, hence, the required boundedness is in the almost sure terms. Moreover, in order to be able to update only some primal blocks on each iteration, \eqref{eq:bounded-dual-deter} now also requires the primal variable to be bounded. Through the simplified algorithms of \cref{sec:fulldual,sec:fullprimal}, treating respective non-randomised dual updates and non-randomised primal updates, we will somewhat relax these restrictions:
	\begin{itemize}[label={--},nosep]
		\item \Cref{alg:acc-full-dual-no-omega} of \cref{sec:fulldual} will not require the dual variable to be bounded if \cref{ass:k-nonlinear} holds with $p=2$; see \cref{col:acc-no-ry-full-dual,col:lin-no-ry-full-dual}.
		\item In \cref{sec:fullprimal}, we will not require any bound on the primal variable.
	\end{itemize}
	In some cases, boundedness can, moreover, be checked analytically based on the explicit formula for $F$. For example, for $F(x)=\abs{\iprod{a}{x}}$ or $F(x)=\norm{ax}$ the support of $F^*(y)$ is bounded by $\norm{a}$. Hence the range of the corresponding proximal operator is also bounded.
	In particular, if the $F$ is of such a form, the boundedness assumptions of  \cref{sec:fullprimal} are automatically satisfied.
\end{remark}

\section{Methods with full dual updates}
\label{sec:fulldual}

We now develop more specific methods based on \eqref{eq:alg-ordered} and study their convergence based on \cref{thm:test-det}.
In this section we take $\iset V(i+1) = \emptyset$, $V(i+1)=\{1,\ldots,n\}$, and $\iset S(i) = S(i)$ for all iterations $i$.
The nesting conditions \eqref{eq:chilo-definition} of \cref{thm:test-det} then hold, and the coupling conditions \eqref{eq:etamu-update-deter} become
\begin{equation}
	\label{eq:etamu-full-dual}
	\iset\pi^{i+1}_j\tauTest^{i+1}_j\iset\tau^{i+1}_j=\eta^{i+1}=\sigmaTest_\ell^{i+1}\dset\sigma_\ell^{i+1}.
\end{equation}
The dual update of \eqref{eq:alg-ordered} involves $\SigmaTest_{i+1}^{-1}[\kgrad{\thisx}\Tau_i^*\TauTest_i^*-\SigmaTest_{i+1}\Sigma_{i+1}\kgrad{\thisx}\Omega_i]$, in scalar form
\begin{equation}
    \label{eq:baromega}
	\frac{\tauTest^i_j\iset\tau^i_j-\omega_j^i\dset\sigma_\ell^{i+1}\sigmaTest_\ell^{i+1}}{\sigmaTest_\ell^{i+1}}
	=\dset\sigma_\ell^{i+1}\left(\frac{\eta^i}{\iset\pi^i_j\eta^{i+1}}-\omega_j^i\right)
	=\dset\sigma_\ell^{i+1}\left(\frac{\bar\omega^i}{\iset\pi^i_j}-\omega_j^i\right)
    \quad\text{for}\quad
    \bar\omega^i\defeq\frac{\eta^i}{\eta^{i+1}}.
\end{equation}
Therefore, with $\omega_j^i = \frac{\bar\omega^i}{\iset\pi^i_j}$, the updates \eqref{eq:alg-ordered} simplify to those of \cref{alg:acc-full-dual-omega}.
Moreover, \eqref{eq:def-lambdacomponent} reduces to $\lambda^i_{j,\ell}=\tauTest^i_j\tau^i_j\chi_{\iset S(i)}(j)$.
We thus verify \eqref{eq:sigmatest-result} via:

\begin{lemma}
    \label{lemma:sigmatest-fulldual}
    Suppose $\iset V(i+1) = \emptyset$, $V(i+1)=\{1,\ldots,n\}$, $\iset S(i) = S(i)$ for $i \in \N$;  the coupling condition \eqref{eq:etamu-full-dual} holds; $\bar\omega^i\le1$; as well as, for all $\ell=1,\ldots,n$ and $j=1,\ldots,m$,
	\begin{equation}
		\label{eq:sigmatest-fulldual}
		\bar\omega^i\dset\sigma^{i+1}_\ell\iset\tau_j^i\le\dset\sigma^0_\ell\iset\tau_j^0
		\quad\text{and}\quad
		1-\kappa
		\ge
		\BiggNorm{
			\sum_{j \in \iset S(i)} \sqrt{\frac{w_{j,\ell}^i\dset\sigma_\ell^0\iset\tau^0_j}{\iset\pi_j^i}}
			Q_\ell \grad K(\thisx) P_j
		}^2
	\end{equation}
	for some $0\le\kappa\le1$ and $w_{j,\ell,k}=1/w_{j,k,\ell}>0$ such that
	\begin{subequations}%
	\label{eq:fulldual:w-r}%
	\begin{gather}
		\label{eq:fulldual:w-r:w}
		w_{j,\ell}^i \defeq \chi_{\this\Neigh_j}(\ell) \sum_{k \in \this{\bar\Neigh_j}(\ell)} w_{j,\ell,k}
	\shortintertext{with}
		\label{eq:fulldual:w-r:v}
		{\bar\Neigh_j}^i(\ell) = \{k \in \{1,\ldots,n\} \mid Q_\ell \kgrad{\thisx} P_j \kgradconj{\thisx} Q_k \ne 0,\, j \in \iset S(i) \}.
	\end{gather}%
	\end{subequations}%
	Then the lower bound \eqref{eq:sigmatest-result} holds.
\end{lemma}

\begin{proof}
	By the first part of \eqref{eq:sigmatest-fulldual}, \eqref{eq:etamu-full-dual}, and $\lambda^i_{j,\ell}=\tauTest^i_j\tau^i_j\chi_{\iset S(i)}(j)$, we have
	\[
		\dset\sigma^0_\ell\iset\tau_j^0 \ge \frac{\this\eta\dset\sigma^{i+1}_\ell\iset\tau_j^i}{\nexxt\eta}
		=
		\frac{\iset\pi^{i}_j\tauTest^{i}_j(\iset\tau_j^i)^2}{\sigmaTest_\ell^{i+1}}
		=
		\frac{\iset\pi^{i}_j(\this\lambda_{j,\ell})^2}{\sigmaTest_\ell^{i+1}\tauTest^{i}_j}
		\quad (j \in \iset S(i)).
	\]
	By the orthogonality of the projections $P_j$, we may insert this estimation into the second part of \eqref{eq:sigmatest-fulldual}, obtaining \eqref{eq:sigmatest}; compare the proof of \cref{lemma:sigmatest}.
    The definition of ${\bar\Neigh_j}^i(\ell)$ in \eqref{eq:def-sim-connected-blocks} also reduces to that in \eqref{eq:fulldual:w-r:v}, while the definition of $w_{j,\ell}^i$ in \eqref{eq:fulldual:w-r:w} is exactly that in \eqref{eq:psi-bounding-bounds}.
    We finish by applying \cref{lemma:sigmatest} to verify \eqref{eq:sigmatest-result}.
\end{proof}

\begin{remark}
    \label{rem:sigmatest-fulldual-cond}
    The first part of \eqref{eq:sigmatest-fulldual} relaxes the property $\tau^i\sigma^i = \tau^0\sigma^0$ of the basic PDPS \cite{chambolle2010first}.
\end{remark}

\begin{remark}
	With deterministic updates ($\this{\iset\pi_j} \equiv 1$), \eqref{eq:etamu-full-dual} couples $\this{\iset \tau_j}\this\tauTest_j=\this{\dset\sigma_\ell}\this\sigmaTest_\ell$. With $\this\sigmaTest_\ell \equiv \sigmaTest_\ell^0$, \eqref{eq:sigmatest-fulldual} therefore becomes a block-coupled variant of the condition $\tau_i \sigma_i \norm{K}^2 < 1$ from \cite{chambolle2010first}.
\end{remark}

Finally, we also remind that \eqref{eq:bar-gamma-g} and \eqref{eq:bar-gamma-f} for this section simplify to
\begin{equation}
	\label{eq:bargamma-fulldual}
			\bar\gamma_{GK,j}^i\equiv\bar\gamma_{GK,j}\defeq
		\gamma_{G,j}+\gamma_{K,j},
		\quad\text{and}\quad
		\bar\gamma_{F^*,\ell}^{i+1}\equiv\bar\gamma_{F^*,\ell}\defeq
		\begin{cases}
			\gamma_{F^*,\ell},& Q_\ell\Pnl=0,\\
			\gamma_{F^*,\ell}-(p-1)\zeta_\ell-\alpha_y, & Q_\ell\Pnl\ne 0.
		\end{cases}
\end{equation}

\begin{Algorithm}
    \caption{Full dual updates \#1}
    \label{alg:acc-full-dual-omega}
	Assume the problem structure \eqref{eq:main-problem-primalonly}, equivalently \eqref{eq:main-problem}.
     For each iteration $i \in \N$, choose a sampling pattern for generating the random set of updated primal blocks $S(i) \in \Random(\SAlg_i; \powerset\{1,\ldots,m\})$ with corresponding blockwise probabilities $\iset\pi_j^i\defeq \P[j\in S(i)\mid\SAlg_{i-1}]>0$. Also choose a rule for the iteration and block-dependent step length parameters $\iset\tau_j^i,\dset\sigma_\ell^i,\bar\omega^i>0$ from one of \cref{thm:acc-full-dual}, \ref{thm:acc2-full-dual}, or \ref{thm:lin-full-dual}.
	Pick an initial iterate $(x^0,y^0)$ and on each iteration $i \in \N$ update all blocks $\nextx_j=P_j\nextx$, ($j=1,\ldots,m$), and  $\nexty_\ell=Q_\ell\nexty$, ($\ell=1,\ldots,n$), of $\nextx$ and $\nexty$ as:
    \begin{align*}
        \nextx_j& \defeq \begin{cases}
            (I+\iset\tau^i_j P_j\subdiff G_j P_j)^{-1}(\thisx_j-\iset\tau^i_jP_j\kgradconj{\thisx}\thisy),
            &
            j\in S(i),
            \\
            \thisx_j,
            &
            j\notin S(i),
        \end{cases}
        \\
        \overnextx_j& \defeq \begin{cases}
            \nextx_j+\bar\omega^i(\nextx_j-\thisx_j)/\iset\pi^i_j,
            &
            j\in S(i),
            \\
            \thisx_j,
            &
            j\notin S(i),
        \end{cases}
        \\
        \nexty_\ell& \defeq
        (I+\dset\sigma^{i+1}_\ell Q_\ell\subdiff F^*_\ell Q_\ell)^{-1}(\thisy_\ell+\dset\sigma^{i+1}_\ell Q_\ell K(\overnextx)).
    \end{align*}
\end{Algorithm}

\subsection{Accelerated rates}

We start with simple step length rules for $O(1/N)$ rates on the blocks admitting second-order growth ($\gamma_{G,j}+\gamma_{K,j}>0$ for primal blocks $j$ or $\gamma_{F^*,\ell}>0$ for dual blocks $\ell$).
Throughout, for simplicity, we assume iteration-independent probabilities, $\iset\pi_j^i = \pi_j^i \equiv \iset\pi_j$ for all $i \in \N$.

\begin{theorem}
	\label{thm:acc2-full-dual}
	Suppose \cref{ass:k-lipschitz,ass:k-nonlinear,ass:gf} hold with $L, L_3\ge0$; $p\in[1,2]$;
	$\gamma_{G,j}+\gamma_{K,j}\ge0$, ($j=1,\ldots,m$), and $\bar\gamma_{F^*,\ell}\ge0$, ($\ell=1,\ldots,n$), for some $\alpha_y,\zeta_\ell\ge0$ as defined in \eqref{eq:bargamma-fulldual}.
    Let the iterates $\{\thisu=(\thisx, \thisy)\}_{i \in \N}$ be generated by \cref{alg:acc-full-dual-omega} with iteration-independent probabilities $\iset\pi^i_j\equiv\iset\pi_j$ and step length parameters
	\begin{subequations}
		\label{eq:step-rules-acc2-full-dual}
		\begin{align}
		\dset\sigma^{i+1}_\ell& \defeq \frac{\dset\sigma^i_\ell}{1+2\dset\sigma^i_\ell\bar\gamma_{F^*,\ell}},
		\quad
		\bar\omega^i \equiv 1,
		\quad\text{and}\quad
		\iset\tau^{i+1}_j \defeq \frac{\iset\tau^i_j}{1+2\iset\tau^i_j\tilde\gamma_{G,j}},
		\end{align}
	\end{subequations}
	with
	either $0\le\tilde\gamma_{G,j}<\iset\pi_j(\gamma_{G,j}+\gamma_{K,j})$ or $\tilde\gamma_{G,j}=\gamma_{G,j}+\gamma_{K,j}=0$ for each $j=1,\ldots,m$;
	and initial $\iset\tau_j^0, \dset\sigma_\ell^0>0$ satisfying for some $0<\delta<\kappa<1$, $\rho_x, \rho_\ell\ge0$, ($\ell=1,\ldots,n$), and $w_{j,\ell}^i$ as in \eqref{eq:fulldual:w-r} the bounds ($i \in \N;\,j=1,\ldots,m$)
	\begin{subequations}
		\label{eq:initialisation-acc2-full-dual}
		\begin{align}
		\label{eq:initialisation-acc2-full-dual-sigmatest}
		1-\kappa
		&\ge
		\adaptNorm{
			\sum_{j \in \iset S(i)} \sqrt{\frac{w_{j,\ell}^i\dset\sigma_\ell^0\iset\tau^0_j}{\iset\pi_j}}
			Q_\ell \grad K(\thisx) P_j
          }^2
		\shortintertext{and}
        \label{eq:initialisation-acc2-full-dual-tau}
		\delta&\ge
		\iset\tau^0_j\bar L+\iset\tau^0_j\cdot
		\begin{cases}
				2(1-\iset\pi_j)(\gamma_{G,j}+\gamma_{K,j})
					\frac{\gamma_{G,j}+\gamma_{K,j}-\tilde\gamma_{G,j}}{\iset\pi_j(\gamma_{G,j}+\gamma_{K,j})-\tilde\gamma_{G,j}}
				& \gamma_{G,j}+\gamma_{K,j}>0\\
				0& \gamma_{G,j}+\gamma_{K,j}=0\\
		\end{cases}
		\shortintertext{with}
		\label{eq:acc2-full-dual-barl-omega}
		\bar L&\defeq
		L_3+L\Bigl(\max_{j=1\ldots m}\Bigl(\frac{1}{\iset\pi_j}+1\Bigr)^2\textstyle\sum_{\ell=1}^{n}\rho_\ell
		+\frac{nL}{2\alpha_y}\rho_x^2\Bigr).
		\end{align}
	\end{subequations}
	Assume for $A\defeq\textstyle\sum_{j\in S(i)}(\iset\pi_j)^{-1}P_j$ that
	\begin{subequations}
		\label{eq:locality-acc2-full-dual}
		\begin{align}
		\label{eq:acc2-full-dual-rhol-theta}
		\E_{i-1}[\theta_{A}]&\ge
		p^{-p}\textstyle\sum_{\ell=1}^n\zeta_\ell^{1-p}\rho_\ell^{2-p}
        \quad\text{and}
        \\
        \label{eq:acc2-full-dual-rhol-omega}
        1&=
        \P[\norm{\nextx-\realoptx}\le\rho_x, \norm{Q_\ell(\nexty-\realopty)}_{\Pnl}\le\rho_\ell, (\ell=1,\ldots,n) \mid \SAlg_{i-1}].
		\end{align}
	\end{subequations}
	Then $\E[\norm{P_j(x^N-\realoptx)}^2] \to 0$ at the rate $O(1/N)$ for all $j$ such that $\tilde\gamma_{G,j}>0$ and $\E[\norm{Q_\ell(y^N-\realopty)}^2] \to 0$ at the rate $O(1/N)$ for all $\ell$ such that $\bar\gamma_{F^*,\ell}>0$.
\end{theorem}

\begin{proof}
	We use \cref{thm:test-det} whose conditions we need to verify.
	We have already verified the nesting condition \eqref{eq:chilo-definition} for $\iset V(i+1) = \emptyset$, $V(i+1)=\{1,\ldots,n\}$, and $\iset S(i) = S(i)$ in \cref{alg:acc-full-dual-omega}.
	The coupling condition \eqref{eq:etamu-update-deter} we have reduced to \eqref{eq:etamu-full-dual}, which we now verify.
	For some $\eta^0>0$ we set $\eta^i\equiv \eta^0$, $\tauTest^0_j \defeq \eta^0(\iset\pi_j\iset\tau_j^0)^{-1}$, and $\sigmaTest_\ell^0 \defeq \eta^0/\dset\sigma_\ell^0$. Then we update
	\begin{equation}
		\label{eq:acc2-fulldual-tautest-tildegammag}
		\tauTest^{i+1}_j=(1+2\iset\tau_j^i\tilde\gamma_{G,j})\tauTest^i_j,
		\quad
		\sigmaTest_\ell^{i+2}=(1+2\dset\sigma^{i+1}_\ell\bar\gamma_{F^*,\ell})\sigmaTest_\ell^{i+1}.
	\end{equation}
	By \eqref{eq:step-rules-acc2-full-dual}, consequently, $\dset\sigma_\ell^{i+1}\sigmaTest_\ell^{i+1}=\nexxt\eta=\iset\pi_j\tauTest^{i+1}_j\iset\tau^{i+1}_j$ for all $\ell$ and $j$. Consequently \eqref{eq:etamu-full-dual} holds. Clearly so does \eqref{eq:conditionality-assumptions} due the deterministic step length and testing parameter updates.
	The conditions \eqref{eq:rules-det} follow from \eqref{eq:locality-acc2-full-dual} given that $\theta_{\TauTest_i\Tau_i}=\eta^i\theta_{A}=\eta^{i+1}\theta_{A}=\dset\sigma_\ell^{i+1}\sigmaTest_\ell^{i+1}\theta_{A}$.

    The step length parameters $\iset\tau^i_j$ and $\dset\sigma_\ell^{i+1}$ are non-increasing in $i$ by the defining \eqref{eq:step-rules-acc2-full-dual}.
	Also using \eqref{eq:initialisation-acc2-full-dual-sigmatest}, we thus verify \eqref{eq:sigmatest-fulldual}. Now \cref{lemma:sigmatest-fulldual} verifies \eqref{eq:sigmatest-result}.

	We still need to verify \cref{thm:test-det}\,\cref{item:thm-test:primal} and \cref{item:thm-test:dual}.
	Regarding the latter, $\sigmaTest^{i+2}_\ell\le(1+2\dset\sigma_\ell^{i+1}\bar\gamma_{F^*,\ell}^{i+1})\sigmaTest_\ell^{i+1}$ trivially as long as $\bar\gamma_{F^*,\ell}^{i+1}\ge0$, which follows from the assumptions on $\gamma_{F^*,\ell}$. Therefore \cref{thm:test-det}\,\ref{item:thm-test:dual} option \ref{item:thm-test:dual:a} holds.
	Regarding \cref{thm:test-det}\,\cref{item:thm-test:primal}, we first of all observe that \eqref{eq:lambda-lambda} reduces to $\this c_*=nL^2\eta^{i+1} \rho_x^2/(2\alpha_y)$.
    Moreover, in \cref{alg:acc-full-dual-omega} we took $\omega^i_j\defeq\bar\omega^i/\iset\pi_j=1/\iset\pi_j$ by \eqref{eq:step-rules-acc2-full-dual}.
	Consequently \eqref{eq:lij} becomes
	\begin{equation}
		\begin{split}
		\label{eq:acc2-omega-full-dual-barl-bound}
		L^i_j &\defeq
		L_3+\Bigl(L\max_{j\in S(i)}(\omega^i_j+1)^2\textstyle\sum_{\ell=1}^{m}\sigmaTest^{i+1}_\ell\dset\sigma^{i+1}_\ell\rho_\ell+\frac{nL^2\eta^{i+1} \rho_x^2}{2\alpha_y}\Bigr)\frac{1}{\tauTest^i_j\iset\tau^i_j}
		\\
		&=
		L_3+L\iset\pi_j\Bigl(\max_{j\in S(i)}(1/\iset\pi_j+1)^2\textstyle\sum_{\ell=1}^{n}\rho_\ell+\frac{nL}{2\alpha_y}\rho_x^2\Bigr)\frac{\eta^{i+1}}{\eta^i}
		\le \bar L.
	\end{split}
	\end{equation}
	We now consider two cases for the satisfaction of \cref{thm:test-det}\,\ref{item:thm-test:primal} option \ref{item:thm-test:primal:a} or \ref{item:thm-test:primal:b}:
	\begin{enumerate}[label=(\Alph*)]
		\item If $\gamma_{G,j}+\gamma_{K,j}=0$, then $\tilde\gamma_{G,j}=0$ and $\tauTest^{i+1}_j=\tauTest^i_j$ by \eqref{eq:acc2-fulldual-tautest-tildegammag}, so option \ref{item:thm-test:primal:a} holds.

		\item If $\gamma_{G,j}+\gamma_{K,j}>0$, then  \eqref{eq:initialisation-acc2-full-dual-tau}, \eqref{eq:acc2-omega-full-dual-barl-bound}, and $\iset\tau_j^i\le\iset\tau_j^0$ show \eqref{eq:tautest-update-tilde-eps}, hence \ref{item:thm-test:primal:b}.
	\end{enumerate}
    
	We can now apply \cref{thm:test-det} to obtain \eqref{eq:convergence-estimate-blockwise}.
	From \eqref{eq:acc2-fulldual-tautest-tildegammag} we have
	\begin{align*}
		\tauTest^{i+1}_j&=\tauTest^i_j+2\tilde\gamma_{G,j}\eta^i/\iset\pi_j=\tauTest^i_j+2\tilde\gamma_{G,j}\eta^1/\iset\pi_j=\ldots=\tauTest^1_j+2i\tilde\gamma_{G,j}\eta^1/\iset\pi_j
        \quad\text{and}
		\\
		\sigmaTest_\ell^{i+2}&=\sigmaTest_\ell^{i+1}+2\bar\gamma_{F^*,\ell}\eta^{i+1}=\sigmaTest_\ell^{i+1}+2\bar\gamma_{F^*,\ell}\eta^1=\ldots=\sigmaTest_\ell^1+2(i+1)\bar\gamma_{F^*,\ell}\eta^1.
	\end{align*}
	Therefore, for any $j$ such that $\tilde\gamma_{G,j}>0$ and $\ell$ such that $\bar\gamma_{F^*,\ell}>0$, $\tauTest^N_j$ and $\sigmaTest_\ell^{N+1}$ grow as $\Omega(N)$.
	This together with \eqref{eq:convergence-estimate-blockwise} gives the claim.
\end{proof}

We can improve the convergence to $O(1/N^2)$ in the primal variable if all the primal blocks exhibit second-order growth. This is achieved by making the dual step lengths grow as in the basic single-block convex case of \cite{chambolle2010first}.

\begin{theorem}
	\label{thm:acc-full-dual}
	Suppose \cref{ass:k-lipschitz,ass:k-nonlinear,ass:gf} hold with $L, L_3\ge0$; $p\in[1,2]$;
	$\gamma_{G,j}+\gamma_{K,j}>0$, ($j=1,\ldots,m$), and $\bar\gamma_{F^*,\ell}\ge0$, ($\ell=1,\ldots,n$), for some $\alpha_y,\zeta_\ell\ge0$ as defined in \eqref{eq:bargamma-fulldual}.
	Let the iterates $\{\thisu=(\thisx, \thisy)\}_{i \in \N}$ be generated by \cref{alg:acc-full-dual-omega} with iteration-independent probabilities $\iset\pi^i_j\equiv\iset\pi_j$ and step length parameters
    \begin{align}
 	    \label{eq:step-rules-acc-full-dual}
		\dset\sigma^{i+1}_\ell& \defeq \frac{\dset\sigma^i_\ell}{\bar\omega^i},
		\quad
		\iset\tau^{i+1}_j \defeq \frac{1}{1+2\iset\tau^i_j\tilde\gamma_{G,j}}\frac{\iset\tau^i_j}{\bar\omega^i},
		\quad\text{and}\quad
		\bar\omega^i
		\defeq\max_{j=1,\ldots,m}\frac{1}{\sqrt{1+2\iset\tau_j^i\tilde\gamma_{G,j}}}
	\end{align}
    with
    $0<\tilde\gamma_{G,j}<\iset\pi_j(\gamma_{G,j}+\gamma_{K,j})$;
     and initial $\iset\tau_j^0, \dset\sigma_\ell^0>0$ satisfying for some $0<\delta\le\kappa<1$, $\rho_x, \rho_\ell\ge0$, ($\ell=1,\ldots,n$), and $w_{j,\ell}^i$ as in \eqref{eq:fulldual:w-r} the bounds
	\begin{subequations}
		\label{eq:initialisation-acc-full-dual}
		\begin{align}
        \label{eq:initialisation-acc-full-dual-sigmatest}
		1-\kappa
		&\ge
		\adaptNorm{
			\sum_{j \in \iset S(i)} \sqrt{\frac{w_{j,\ell}^i\dset\sigma_\ell^0\iset\tau^0_j}{\iset\pi_j}}
			Q_\ell \grad K(\thisx) P_j
        }^2
        \quad (i \in \N)
        \quad \text{and}
		\\
        \label{eq:initialisation-acc-full-dual-tau}
		\delta&\ge
		\iset\tau^0_j\left(\bar L+2(1-\iset\pi_j)(\gamma_{G,j}+\gamma_{K,j})
		\frac{\gamma_{G,j}+\gamma_{K,j}-\tilde\gamma_{G,j}}{\iset\pi_j(\gamma_{G,j}+\gamma_{K,j})-\tilde\gamma_{G,j}}
		\right)
        \quad\text{with}
		\\
		\label{eq:full-dual-barl-omega}
		\bar L&\defeq
		L_3+\frac{L}{\bar\omega^0}\Bigl(\max_{j=1\ldots m}\Bigl(\frac{1}{\iset\pi_j}+1\Bigr)^2\textstyle\sum_{\ell=1}^{n}\rho_\ell
		+\frac{nL}{2\alpha_y}\rho_x^2\Bigr).
		\end{align}
	\end{subequations}
	Assume for $A\defeq\textstyle\sum_{j\in S(i)}(\iset\pi_j)^{-1}P_j$ that
	\begin{subequations}
		\label{eq:locality-acc-full-dual}
		\begin{align}
        \label{eq:full-dual-rhol-theta}
        \E_{i-1}[\theta_{A}]&\ge
        p^{-p}\textstyle\sum_{\ell=1}^n\zeta_\ell^{1-p}\rho_\ell^{2-p}/\bar\omega^0
        \quad\text{and}
        \\
		\label{eq:full-dual-rhol-omega}
		1&=
		\P[\norm{\nextx-\realoptx}\le\rho_x, \norm{Q_\ell(\nexty-\realopty)}_{\Pnl}\le\rho_\ell, (\ell=1,\ldots,n) \mid \SAlg_{i-1}].
		\end{align}
	\end{subequations}
	Then $\E[\norm{P_j(x^N-\realoptx)}^2] \to 0$ at the rate $O(1/N^2)$ for all $j$.
\end{theorem}

\begin{proof}
	We use \cref{thm:test-det} whose conditions we need to verify.
    We have already verified the nesting conditions \eqref{eq:chilo-definition} for the choices $\iset V(i+1) = \emptyset$, $V(i+1)=\{1,\ldots,n\}$, and $\iset S(i) = S(i)$ in \cref{alg:acc-full-dual-omega}.
    The coupling condition \eqref{eq:etamu-update-deter} we have reduced to \eqref{eq:etamu-full-dual}.
    To verify \eqref{eq:etamu-full-dual}, we initialise $\tauTest^0_j\defeq\eta^0(\iset\pi_j^0\iset\tau_j^0)^{-1}$ and $\sigmaTest_\ell^0 \defeq \eta^0/\dset\sigma_\ell^0$ for some $\eta^0>0$, and update
    \begin{equation}
	    \label{eq:acc-fulldual-tautest-tildegammag}
	    \tauTest^{i+1}_j \defeq (1+2\iset\tau_j^i\tilde\gamma_{G,j})\tauTest^i_j,
        \quad
        \nexxt\sigmaTest_\ell \defeq \this\sigmaTest_\ell,
        \quad\text{and}\quad
        \eta^{i+1} \defeq \eta^i/\bar\omega^i.
    \end{equation}
    Then from \eqref{eq:step-rules-acc-full-dual}, $\sigmaTest_\ell^{i+1}\dset\sigma_\ell^{i+1}=\sigmaTest_\ell^{i}\dset\sigma_\ell^{i}/\bar\omega^{i}$
    and $\tauTest^{i+1}_j\iset\tau^{i+1}_j=\tauTest^i_j\iset\tau^i_j/\bar\omega^i$.
    Therefore, \eqref{eq:etamu-full-dual} holds by induction.
    Clearly also \eqref{eq:conditionality-assumptions} holds due to the step length and testing parameters being updated deterministically.
    The conditions \eqref{eq:rules-det} follow from \eqref{eq:locality-acc-full-dual} and \eqref{eq:etamu-full-dual} given that $\iset\tau^i_j$ decreases so $\bar\omega^i\ge\bar\omega_0$ and $\theta_{\TauTest_i\Tau_i}=\eta^i\theta_{A}=\eta^{i+1}\bar\omega^i\theta_{A}$.

    We now verify \eqref{eq:sigmatest-result}.
    By \eqref{eq:step-rules-acc-full-dual} and \eqref{eq:acc-fulldual-tautest-tildegammag}, we get $\tauTest^{i+1}_j(\iset\tau_j^{i+1})^2\le\tauTest^i_j(\iset\tau_j^i)^2$. This and \eqref{eq:etamu-full-dual} yield
    \[
    	\bar\omega^i\dset\sigma^{i+1}_\ell\iset\tau_j^i
    	=\frac{\eta^i\iset\tau_j^i}{\sigmaTest^{i+1}_\ell}
    	=\frac{\tauTest^i_j(\iset\tau_j^i)^2}{\sigmaTest^{i+1}_\ell\iset\pi_j}
    	\le
    	\frac{\tauTest^0_j(\iset\tau_j^0)^2}{\sigmaTest^{i+1}_\ell\iset\pi_j}
    	=\frac{\eta^0\iset\tau_j^0}{\sigmaTest^0_\ell}
    	=\dset\sigma^0_\ell\iset\tau_j^0.
    \]
    Combining this estimate with \eqref{eq:initialisation-acc-full-dual-sigmatest} we verify \eqref{eq:sigmatest-fulldual}. Thus \cref{lemma:sigmatest-fulldual} establishes \eqref{eq:sigmatest-result}.

    We still need to verify \cref{thm:test-det}\,\cref{item:thm-test:primal} and \cref{item:thm-test:dual}.
    Regarding the dual test, $\sigmaTest^{i+2}_\ell=\sigmaTest^{i+1}_\ell\le(1+2\dset\sigma_\ell^{i+1}\bar\gamma_{F^*,\ell}^{i+1})\sigmaTest^{i+1}_\ell$ trivially as long as $\bar\gamma_{F^*,\ell}^{i+1}\ge0$, which follows from the assumptions on $\gamma_{F^*,\ell}$. Therefore \cref{thm:test-det}\,\ref{item:thm-test:dual} option \ref{item:thm-test:dual:a} holds.
	As far as \cref{thm:test-det}\,\cref{item:thm-test:primal} is concerned, we observe that \eqref{eq:lambda-lambda} reduces to $\this c_*=nL^2\eta^{i+1} \rho_x^2/(2\alpha_y)$.
    Consequently \eqref{eq:lij} becomes
    \begin{equation}
    \label{eq:acc-omega-full-dual-barl-bound}
        L^i_j \defeq
        L_3+L\iset\pi_j(\max_{j\in S(i)}(\omega^i_j+1)^2\textstyle\sum_{\ell=1}^{n}\rho_\ell+\frac{nL}{2\alpha_y}\rho_x^2)\eta^{i+1}/\eta^i
        \le \bar L
    \end{equation}
    thanks to $\omega^i_j\defeq \bar\omega^i/\iset\pi_j\le 1/\iset\pi_j$ and $\bar\omega^i\ge \bar\omega^0$.
    Also,  with $\tilde\gamma_{G,j}^i<(\iset\pi_j\gamma_{G,j}+\gamma_{K,j})$, \eqref{eq:initialisation-acc-full-dual-tau}, \eqref{eq:acc-omega-full-dual-barl-bound}, and $\iset\tau_j^i\le\iset\tau_j^0$ show \eqref{eq:tautest-update-tilde-eps}, hence, \eqref{eq:tautest-update-tilde}. Therefore,  \cref{thm:test-det}\ref{item:thm-test:primal} option \ref{item:thm-test:dual:b} holds for every $j=1,\ldots,m$.

	We can thus apply \cref{thm:test-det} to obtain \eqref{eq:convergence-estimate-blockwise}.
    Multiplying the $\tau$ update of \eqref{eq:step-rules-acc-full-dual} by $2\tilde\gamma_{G,j}$, plugging in $\bar\omega^i$, and taking the inverse,
    we have
    \[
    	(2\iset\tau^{i+1}_j\tilde\gamma_{G,j})^{-1} =
    	\frac{1+2\iset\tau^i_j\tilde\gamma_{G,j}}{2\iset\tau^i_j\tilde\gamma_{G,j}\sqrt{1+\min_{j=1\ldots m}(2\iset\tau^i_j\tilde\gamma_{G,j})}}
    	=\frac{1+(2\iset\tau^i_j\tilde\gamma_{G,j})^{-1}}{\sqrt{1+(\max_{j=1\ldots m}(2\iset\tau^i_j\tilde\gamma_{G,j})^{-1})^{-1}}}
    \]
    We now apply \cref{lemma:max-for-n2} with $z_j^i=(2\iset\tau^i_j\tilde\gamma_{G,j})^{-1}$ to get $\max_{j=1\ldots m}(2\iset\tau^N_j\tilde\gamma_{G,j})^{-1}\le \bar z_0 +N/2$ with $\bar z_0>0$.
    Then from \eqref{eq:acc-fulldual-tautest-tildegammag}, we have
    \[
    	\begin{split}
    	\tauTest^{N+1}_j&\ge
    	(1+\min_{j=1\ldots m}(2\iset\tau_j^i\tilde\gamma_{G,j}))\tauTest^N_j
    	\ge \Bigl(1+\frac{1}{\bar z_0+N/2}\Bigr)\tauTest^N_j
    	=\frac{2\bar z_0+N+2}{2\bar z_0+N}\tauTest^N_j
    	\\
    	&=
    	\frac{2\bar z_0+N+2}{2\bar z_0+N}\frac{2\bar z_0+N+1}{2\bar z_0 +N-1}\tauTest^{N-1}_j
    	=\ldots
    	=\frac{(2\bar z_0+N+2)(2\bar z_0+N+1)}{2\bar z_0(2\bar z_0+1)}\tauTest^0_j.
    	\end{split}
    \]
    Therefore, $\tauTest^{N}_j$ grows as $\Omega(N^2)$, and we obtain the claimed convergence rates from \eqref{eq:convergence-estimate-blockwise}.
\end{proof}

In \cref{alg:acc-full-dual-omega}, we chose $\omega^i_j$ to eliminate the $\kgrad{\thisx}$ term from the dual step. Selecting $\omega^i_j=-1$ keeps this term, but eliminates the necessity to have a finite $\rho_\ell$ as long as $p=2$ as \eqref{eq:lij} and \eqref{eq:zeta-rule-deter} will no longer depend on it. This yields \cref{alg:acc-full-dual-no-omega} and the following:

\begin{Algorithm}
    \caption{Full dual updates \#2}
	\label{alg:acc-full-dual-no-omega}
	Assume the problem structure \eqref{eq:main-problem-primalonly}, equivalently \eqref{eq:main-problem}.
	. For each iteration $i \in \N$, choose a sampling pattern for generating the random set of updated dual blocks $V(i+1) \in \Random(\SAlg_i; \powerset\{1,\ldots,n\})$ with corresponding blockwise probabilities $\iset\pi_j^i\defeq \P[j\in S(i)\mid\SAlg_{i-1}]>0$. Choose a rule for the iteration and block-dependent step length parameters $\iset\tau^i_j,\dset\sigma^{i+1}_\ell,\bar\omega^i>0$ based on one of \cref{thm:acc-full-dual}, \ref{thm:acc2-full-dual}, or \ref{thm:lin-full-dual}.
	Pick an initial iterate $(x^0,y^0)$ and on each iteration $i \in \N$ update all blocks $\nextx_j=P_j\nextx$, ($j=1,\ldots,m$), and  $\nexty_\ell=Q_\ell\nexty$, ($\ell=1,\ldots,n$), of $\nextx$ and $\nexty$ as:
	\begin{align*}
	\nextx_j& \defeq \begin{cases}
		(I+\iset\tau^i_j P_j\subdiff G_j P_j)^{-1}(\thisx_j-\iset\tau^i_jP_j\kgradconj{\thisx}\thisy),
		&
		j\in S(i),
		\\
		\thisx_j,
		&
		j\notin S(i),
	\end{cases}
	\\
	\nexty_\ell& \defeq
	(I+\dset\sigma^{i+1}_\ell Q_\ell\subdiff F^*_\ell Q_\ell)^{-1}\biggl(
        \thisy_\ell+\dset\sigma^{i+1}_\ell Q_\ell K(\thisx)
    	+\dset\sigma_\ell^{i+1}\sum_{j \in S(i)}\biggl(\frac{\bar\omega^i}{\iset\pi^i_j}+1\biggr)Q_\ell\kgrad{\thisx}(\nextx_j-\thisx_j)
    \biggr).
	\end{align*}
\end{Algorithm}

\begin{corollary}
	\label{col:acc-no-ry-full-dual}
	\Cref{thm:acc2-full-dual,thm:acc-full-dual} apply to \cref{alg:acc-full-dual-no-omega} if \cref{ass:k-nonlinear} holds with $p=2$, and instead of \eqref{eq:acc2-full-dual-barl-omega}, \eqref{eq:full-dual-barl-omega}, \eqref{eq:acc2-full-dual-rhol-omega}, and \eqref{eq:full-dual-rhol-omega}, we assume 
	\[
		\bar L \defeq L_3+nL^2\rho_x^2/(2\alpha_y)
		\quad\text{and}\quad
		\P[\norm{\nextx-\realoptx}\le\rho_x \mid \SAlg_{i-1}]=1.
	\]
\end{corollary}

\begin{proof}
	The proof remains exactly the same those of \cref{thm:acc2-full-dual,thm:acc-full-dual}.
	Inserting $\omega^i_j=-1$, \eqref{eq:acc2-omega-full-dual-barl-bound} and \eqref{eq:acc-omega-full-dual-barl-bound} as well as \eqref{eq:acc2-full-dual-rhol-theta} and \eqref{eq:full-dual-rhol-theta} lose their dependency on $\rho_\ell$. Hence $\rho_\ell$ can be taken infinitely large.
\end{proof}

\subsection{Linear convergence}

If all the primal and dual blocks exhibit second-order growth, i.e., $\bar\gamma_{F^*,\ell}>0$ and $\gamma_{G,j}+\gamma_{K,j}>0$, we obtain linear convergence:

\begin{theorem}
	\label{thm:lin-full-dual}
	Suppose \cref{ass:k-lipschitz,ass:k-nonlinear,ass:gf} hold with $L, L_3\ge0$; $p\in[1,2]$;
	$\gamma_{G,j}+\gamma_{K,j}>0$, ($j=1,\ldots,m$), and $\bar\gamma_{F^*,\ell}>0$, ($\ell=1,\ldots,n$), for some $\alpha_y,\zeta_\ell\ge0$ as defined in \eqref{eq:bargamma-fulldual}.
    Let the iterates $\{\thisu=(\thisx, \thisy)\}_{i \in \N}$ be generated by \cref{alg:acc-full-dual-omega} with iteration-independent probabilities $\iset\pi^i_j\equiv \iset\pi_j$ and step length parameters
	\begin{subequations}
		\label{eq:step-rules-lin-full-dual}
        \begin{align}
        \label{eq:step-rules-lin-full-dual:tausigma}
		\iset\tau^{i+1}_j& \defeq
		\frac{\iset\tau^{i}_j}{(1+2\iset\tau^{i}_j\tilde\gamma_{G,j})\bar\omega},
		\quad
		\dset\sigma^{i+1}_\ell \defeq
		\frac{\dset\sigma^{i}_\ell}{(1+2\dset\sigma^{i}_\ell\bar\gamma_{F^*,\ell})\bar\omega},
		\quad\text{and}\quad
        \\
        \label{eq:step-rules-lin-full-dual:omega}
		\bar\omega^i&\equiv\bar\omega\defeq\max\biggl\{\max_{j=1\ldots m}\frac{1}{1+2\iset\tau^{0}_j\tilde\gamma_{G,j}},
		\max_{\ell=1\ldots n}\frac{1}{1+2\dset\sigma^0_\ell\bar\gamma_{F^*,\ell}}\biggr\}
	\end{align}
	\end{subequations}
	with
	$0<\tilde\gamma_{G,j}<\iset\pi_j(\gamma_{G,j}+\gamma_{K,j})$;
	and initial $\iset\tau^0_j,\dset\sigma^0_\ell>0$ satisfying for some $0<\delta<\kappa<1$, $\rho_x,\rho_\ell\ge0$, ($\ell=1,\ldots,n$), and $w_{j,\ell}^i$ as in \eqref{eq:fulldual:w-r} the bounds
	\begin{subequations}
		\label{eq:initialisation-lin-full-dual}
		\begin{align}
        \label{eq:initialisation-lin-full-dual-r}
		1-\kappa
		&\ge
		\adaptNorm{
			\sum_{j \in \iset S(i)} \sqrt{\frac{w_{j,\ell}^i\dset\sigma_\ell^0\iset\tau^0_j}{\iset\pi_j}}
			Q_\ell \grad K(\thisx) P_j
        }^2
        \quad (i \in \N)
        \quad\text{and}
		\\
        \label{eq:initialisation-lin-full-dual-tau}
		\delta&\ge
		\iset\tau^0_j
		\left(\bar L+2(1-\iset\pi_j)(\gamma_{G,j}+\gamma_{K,j})
		\frac{\gamma_{G,j}+\gamma_{K,j}-\tilde\gamma_{G,j}}{\iset\pi_j(\gamma_{G,j}+\gamma_{K,j})-\tilde\gamma_{G,j}}
			\right)
		\quad
		(j\in S(i)),
        \quad\text{with}
		\\
		\label{eq:lin-full-dual-barl-omega}
		\bar L&\defeq
		L_3+\frac{L}{\bar\omega}\left(
		\max_{j=1\ldots m}\left(\frac{\bar\omega}{\iset\pi_j}+1\right)^2\textstyle\sum_{\ell=1}^{n}\rho_\ell
		+\frac{nL}{2\alpha_y}\rho_x^2\right).
 		\end{align}
	\end{subequations}
	Further assume for $A\defeq\textstyle\sum_{j\in S(i)}(\iset\pi_j)^{-1}P_j$ that
	\begin{subequations}
		\label{eq:locality-lin-full-dual}
		\begin{align}
		\label{eq:lin-full-dual-rhol-theta}
		\E_{i-1}[\theta_{A}]&\ge
		p^{-p}\textstyle\sum_{\ell=1}^n\zeta_\ell^{1-p}\rho_\ell^{2-p}/\bar\omega
        \quad\text{and}
        \\
        \label{eq:lin-full-dual-rhol-omega}
        1&=
        \P[\norm{\nextx-\realoptx}\le\rho_x,\norm{Q_\ell(\nexty-\realopty)}_{\Pnl}\le\rho_\ell, (\ell=1,\ldots,n) \mid \SAlg_{i-1}].
		\end{align}
	\end{subequations}
	Then $\E[\norm{P_j(x^N-\realoptx)}^2]$ and $\E[\norm{Q_\ell(y^N-\realopty)}^2]$ converge to zero at the linear rate $O((1/\bar \omega)^N)$ for all $j \in \{1,\ldots,m\}$ and $\ell \in \{1,\ldots,n\}$.
\end{theorem}

\begin{proof}
    We use \cref{thm:test-det} whose conditions we need to verify.
    We have already verified the nesting condition \eqref{eq:chilo-definition} for the choices $\iset V(i+1) = \emptyset$, $V(i+1)=\{1,\ldots,n\}$, and $\iset S(i) = S(i)$ in \cref{alg:acc-full-dual-omega}.
    The coupling condition \eqref{eq:etamu-update-deter} we have reduced to \eqref{eq:etamu-full-dual}.
    To verify \eqref{eq:etamu-full-dual}, we initialise $\tauTest^0_j\defeq\eta^0(\iset\pi_j^0\iset\tau_j^0)^{-1}$ and $\sigmaTest_\ell^0 \defeq \eta^0/\dset\sigma_\ell^0$ for some $\eta^0>0$, and update
    \begin{equation}
	    \label{eq:lin-fulldual-tautest-tildegammag}
	    \tauTest^{i+1}_j \defeq (1+2\iset\tau_j^i\tilde\gamma_{G,j})\tauTest^i_j,
	    \quad
	    \nexxt\sigmaTest_\ell \defeq (1+2\dset\sigma^i_\ell\bar\gamma_{F^*,\ell})\sigmaTest_\ell^i,
	    \quad\text{and}\quad
	    \eta^{i+1} \defeq \eta^i/\bar\omega.
    \end{equation}
    Then from \eqref{eq:step-rules-lin-full-dual}, $\sigmaTest_\ell^{i+1}\dset\sigma_\ell^{i+1}=\sigmaTest_\ell^{i}\dset\sigma_\ell^{i}/\bar\omega$
    and $\tauTest^{i+1}_j\iset\tau^{i+1}_j=\tauTest^i_j\iset\tau^i_j/\bar\omega$.
    Therefore, \eqref{eq:etamu-full-dual} holds by induction.
    Clearly also \eqref{eq:conditionality-assumptions} holds as the step length and testing parameters are updated deterministically.
    The conditions \eqref{eq:rules-det} follow from \eqref{eq:locality-lin-full-dual} given that $\theta_{\TauTest_i\Tau_i}=\eta^i\theta_{A}=\bar\omega\eta^{i+1}\theta_{A}$.

	We now prove \eqref{eq:sigmatest-result}. We start by proving by induction that
	\begin{equation}
		\label{eq:omega-linear-full-dual}
		\bar\omega=
		\max\biggl\{\max_{j=1\ldots m}\frac{1}{1+2\iset\tau^i_j\tilde\gamma_{G,j}},
		\max_{\ell=1\ldots n}\frac{1}{1+2\dset\sigma^i_\ell\bar\gamma_{F^*,\ell}}\biggr\},
    \end{equation}
    in other words
    \[
        \bar\omega^{-1}=1+\min\Bigl\{\min_{j=1\ldots m} 2\iset\tau^i_j\tilde\gamma_{G,j}, \min_{\ell=1\ldots n}2\dset\sigma^i_\ell\bar\gamma_{F^*,\ell}\Bigr\}.
    \]
	The inductive base for $i=0$ is clear from \eqref{eq:step-rules-lin-full-dual:omega}. Using \eqref{eq:step-rules-lin-full-dual:tausigma}, we obtain
	\begin{multline*}
		\min\biggl\{
			\min_{j=1\ldots m} 2\iset\tau^{i+1}_j\tilde\gamma_{G,j},
			\min_{\ell=1\ldots n}2\dset\sigma^{i+1}_\ell\bar\gamma_{F^*,\ell}
		\biggr\}
		=
		\frac{1}{\bar\omega}
		\min\biggl\{
			\min_{j=1\ldots m} \frac{1}{1+(2\iset\tau^i_j\tilde\gamma_{G,j})^{-1}},
			\min_{\ell=1\ldots n} \frac{1}{1+(2\dset\sigma^i_\ell\bar\gamma_{F^*,\ell})^{-1}}
		\biggr\}
		\\
		=
		\frac{1}{\bar\omega}
		\frac{1}{1+\min^{-1}\Bigl\{
				\min_{j=1\ldots m} 2\iset\tau^i_j\tilde\gamma_{G,j},
				\min_{\ell=1\ldots n}2\dset\sigma^i_\ell\bar\gamma_{F^*,\ell}
			\Bigr\}}
		=
		\min\biggl\{\min_{j=1\ldots m} 2\iset\tau^i_j\tilde\gamma_{G,j},
		\min_{\ell=1\ldots n}2\dset\sigma^i_\ell\bar\gamma_{F^*,\ell}\biggr\},
    \end{multline*}
    This establishes the inductive step, hence \eqref{eq:omega-linear-full-dual}.
    By \eqref{eq:omega-linear-full-dual} and \eqref{eq:step-rules-lin-full-dual:tausigma}, $\iset\tau^{i+1}_j$ and $\dset\sigma^{i+1}_\ell$ are non-increasing in $i$. Also using \eqref{eq:initialisation-lin-full-dual-r}, this verifies \eqref{eq:sigmatest-fulldual}. Thus \cref{lemma:sigmatest-fulldual} verifies \eqref{eq:sigmatest-result}.

    We need to verify \cref{thm:test-det}\,\cref{item:thm-test:primal} and \cref{item:thm-test:dual}. Option \ref{item:thm-test:dual:a} of the latter is trivially satisfied for every $\ell=1,\ldots,n$ based on \eqref{eq:lin-fulldual-tautest-tildegammag}.
    Regarding \cref{thm:test-det}\,\cref{item:thm-test:primal}, we first of all observe that \eqref{eq:lambda-lambda} reduces to $\this c_*=nL^2\eta^{i+1} \rho_x^2/(2\alpha_y)$.
    Consequently \eqref{eq:lij} becomes
    \begin{equation}
	    \label{eq:lin-omega-full-dual-barl-bound}
	    L^i_j \defeq
	    L_3+L\iset\pi_j(\max_{j\in S(i)}(\omega^i_j+1)^2\textstyle\sum_{\ell=1}^{n}\rho_\ell+\frac{nL}{2\alpha_y}\rho_x^2)\eta^{i+1}/\eta^i
	    \le \bar L
    \end{equation}
    for $\omega^i_j\defeq\bar\omega^i/\iset\pi_j$ as in \cref{alg:acc-full-dual-omega}.
    And with $\tilde\gamma_{G,j}<\iset\pi_j(\gamma_{G,j}+\gamma_{K,j})$, \eqref{eq:initialisation-lin-full-dual-tau}, \eqref{eq:lin-omega-full-dual-barl-bound}, and $\iset\tau^{i+1}_j\le \iset\tau^0_j$ show \eqref{eq:tautest-update-tilde-eps}.
    Therefore, \cref{thm:test-det}\ref{item:thm-test:primal} option \ref{item:thm-test:dual:b} holds for every $j=1,\ldots,m$.

    We can now apply \cref{thm:test-det} to obtain \eqref{eq:convergence-estimate-blockwise}. By \eqref{eq:lin-fulldual-tautest-tildegammag} and \eqref{eq:omega-linear-full-dual} we have
    \begin{align*}
    	\tauTest_j^{N+1}&=(1+2\iset\tau^N_j\tilde\gamma_{G,j})\tauTest_j^N
    	\ge
    	\tauTest_j^N/\bar\omega
    	\ge\ldots\ge
        \tauTest_j^0/\bar\omega^{N+1}
        \quad\text{and}\\
    	\sigmaTest_\ell^{N+1}&=(1+2\dset\sigma^N_\ell\bar\gamma_{F^*,\ell})\sigmaTest_\ell^N
    	\ge
    	\sigmaTest_\ell^N/\bar\omega
		\ge\ldots\ge
		\sigmaTest_\ell^0/\bar\omega^{N+1}.
    \end{align*}
    Applying these estimates in \eqref{eq:convergence-estimate-blockwise} establishes the claimed linear convergence rates.
\end{proof}

Similarly to \cref{alg:acc-full-dual-no-omega}, we could in the derivation of \cref{alg:acc-full-dual-omega} set $\omega^i_j=-1$ to remove any dependencies on $\rho_\ell$ from \eqref{eq:lin-full-dual-barl-omega} and \eqref{eq:lin-full-dual-rhol-theta}. This yields \cref{alg:acc-full-dual-no-omega} and:

\begin{corollary}
	\label{col:lin-no-ry-full-dual}
	\Cref{thm:lin-full-dual} applies to \cref{alg:acc-full-dual-no-omega} if \cref{ass:k-nonlinear} holds with $p=2$, and \eqref{eq:lin-full-dual-barl-omega} and \eqref{eq:lin-full-dual-rhol-omega} are replaced with
	\[
		\bar L \ge L_3+nL^2\rho_x^2/(2\alpha_y\bar\omega)
		\quad\text{and}\quad
		\P[\norm{\nextx-\realoptx}\le\rho_x \mid \SAlg_{i-1}]=1.
	\]
\end{corollary}

\begin{proof}
	The proof remains exactly the same as \cref{thm:lin-full-dual} given all $\omega^i_j=-1$ in \eqref{eq:lin-omega-full-dual-barl-bound} and \eqref{eq:lin-full-dual-rhol-theta} no longer depend on $\rho_\ell$, hence $\rho_\ell$ can be taken infinitely large.
\end{proof}

\begin{remark}[Stochastic block-coordinate forward--backward splitting]
	\label{remark:full-dual-fb-splitting}
    Let $F(z) \defeq z$ for $z \in \R$ and $K \in C^1(X)$.
    Then $F^*(y)=\delta_{\{1\}}(y)$. Taking $n=1$ and $Q_1=I$ results in $(I+\dset\sigma^{i+1}_1 Q_1\subdiff F^* Q_1)^{-1} \equiv 1$.
    Consequently $\thisy \equiv 1$ on all iterations, so that the updates of \cref{alg:acc-full-dual-omega,alg:acc-full-dual-no-omega} reduce to
    \begin{equation}
        \label{eq:block-fb}
        \nextx_j \defeq \begin{cases}
        (I+\iset\tau^i_j P_j\subdiff G_j P_j)^{-1}(\thisx_j-\iset\tau^i_jP_j\grad K(x)),
        &
        j\in S(i),
        \\
        \thisx_j,
        &
        j\notin S(i),
        \end{cases}
    \end{equation}
    In the step length conditions of \cref{thm:acc2-full-dual,thm:acc-full-dual,thm:lin-full-dual}, we can moreover take $\rho_1=0$ and let $\gamma_{F^*,1} \upto \infty$, consequently $a_y \upto \infty$. In particular, in all the theorems, $\bar L=L_3$, so that when $\iset\pi_j=1$, the upper bounds on the primal step lengths reduce to $\delta \ge \iset\tau_j^0 L_3$ for some $\delta \in (0, 1)$ similarly to the standard condition in forward--backward splitting type methods.
    Moreover, by \eqref{lemma:three-point-k-explained}, $\gamma_{K,1}$ is simply a (reduced) factor of strong monotonicity of $K$ at $\realoptx$ as defined in \cref{ass:gf}.
    Finally, since we can take $\dset \sigma_1^0>0$ arbitrarily small without affecting the updates \eqref{eq:block-fb}, the conditions in the theorems corresponding to \eqref{eq:sigmatest} become irrelevant.
\end{remark}

\section{Methods with full primal updates}
\label{sec:fullprimal}

We continue with developing more specific methods and their convergence results based on the updates of \eqref{eq:alg-ordered} and the conditions of \cref{thm:test-det}.
We now take $\iset S(i)=\emptyset$, $S(i)=\{1,\ldots,m\}$, and $\iset V(i+1)=V(i+1)$ for all iterations $i$.
Then the nesting condition \eqref{eq:chilo-definition} of \cref{thm:test-det} holds and the coupling condition \eqref{eq:etamu-update-deter} becomes
\begin{equation}
	\label{eq:etamu-full-primal}
	\tauTest^i_j\dset\tau^i_j=\eta^{i+1}=\iset\nu_\ell^{i+2}\sigmaTest_\ell^{i+2}\iset\sigma_\ell^{i+2}.
\end{equation}
Taking $\Omega_i=-I$, the updates of \eqref{eq:alg-ordered} simplify to those of \cref{alg:acc-full-primal} since for the last two terms in the primal update
\[
	\dset\tau^i_j\nexty_\ell+\frac{\sigmaTest^{i+1}_\ell\sigma^{i+1}_\ell}{\tauTest^i_j}(\nexty_\ell-\thisy_\ell)
	=
	\dset\tau^i_j\left(\nexty_\ell+\frac{\bar\omega^i}{\iset\nu_\ell^{i+1}}(\nexty_\ell-\thisy_\ell)\right)
	\quad\text{for}\quad
	\bar\omega^i\defeq\frac{\eta^i}{\eta^{i+1}}.
\]
Moreover, \eqref{eq:def-lambdacomponent} reduces to $\lambda^i_{j,\ell}=-\sigma^{i+1}_\ell\sigmaTest^{i+1}_\ell$. We thus verify \eqref{eq:sigmatest-result} via:

\begin{lemma}
    \label{lemma:sigmatest-fullprimal}
    Suppose $\iset S(i)=\emptyset$, $S(i)=\{1,\ldots,m\}$, and $\iset V(i+1)=V(i+1)$ for $i \in \N$;  the coupling condition \eqref{eq:etamu-full-primal} holds; $\bar\omega^i\le1$; as well as, for all $\ell=1,\ldots,n;\,j=1,\ldots,m$,
	\begin{equation}
		\label{eq:sigmatest-fullprimal}
		\iset\sigma^{i+1}_\ell\dset\tau_j^i\le\iset\sigma^1_\ell\dset\tau_j^0,
		\quad\text{and}\quad
		1-\kappa
		\ge
		\adaptNorm{
			\sum_{j=1}^m \sqrt{\frac{w_{j,\ell}^i\iset\sigma_\ell^1\dset\tau^0_j}{\iset\nu_\ell^{i+1}}}
			Q_\ell \grad K(\thisx) P_j
		}^2
    \end{equation}
    for some $0\le\kappa\le1$ and $w_{j,\ell,k}=1/w_{j,k,\ell}>0$ such that
	\begin{subequations}%
		\label{eq:fullprimal:w-r}%
		\begin{gather}
		\label{eq:fullprimal:w-r:w}
		w_{j,\ell}^i \defeq \chi_{\this\Neigh_j}(\ell) \sum_{k \in \this{\bar\Neigh_j}(\ell)} w_{j,\ell,k}
		\shortintertext{with}
		\label{eq:fullprimal:w-r:v}
		{\bar\Neigh_j}^i(\ell) = \{k \in \{1,\ldots,n\} \mid Q_\ell \kgrad{\thisx} P_j \kgradconj{\thisx} Q_k \ne 0,\, \ell \in\iset V(i+1) \}.
		\end{gather}%
	\end{subequations}
	Then the lower bound \eqref{eq:sigmatest-result} holds.
\end{lemma}

\begin{proof}
	By the first part of \eqref{eq:sigmatest-fullprimal}, \eqref{eq:etamu-full-primal}, and $\lambda^i_{j,\ell}=-\sigma^{i+1}_\ell\sigmaTest^{i+1}_\ell=-\iset\sigma^{i+1}_\ell\sigmaTest^{i+1}_\ell$, we have
	\[
        \iset\sigma^1_\ell\dset\tau_j^0
        \ge
        \iset\sigma^{i+1}_\ell\dset\tau_j^i
        =
        \frac{(\iset\sigma^{i+1}_\ell\nexxt\sigmaTest_\ell)^2\dset\tau_j^i}{\iset\sigma^{i+1}_\ell(\nexxt\sigmaTest_\ell)^2}
        =
        \frac{(\this\lambda_{j,\ell})^2\iset\nu_\ell^{i+1}}{\nexxt\sigmaTest_\ell\this\tauTest_j}
		\quad (j=1,\ldots,m).
	\]
	By the orthogonality of the projections $P_j$, we may insert this estimation into the second part of \eqref{eq:sigmatest-fullprimal}, obtaining \eqref{eq:sigmatest}; compare the proof of \cref{lemma:sigmatest}.
    The definition of ${\bar\Neigh_j}^i(\ell)$ in \eqref{eq:def-sim-connected-blocks} also reduces to that in \eqref{eq:fullprimal:w-r:v}, while the definition of $w_{j,\ell}^i$ in \eqref{eq:fullprimal:w-r:w} is exactly that in \eqref{eq:psi-bounding-bounds}.
    We finish by applying \cref{lemma:sigmatest} to verify \eqref{eq:sigmatest-result}.
\end{proof}

\begin{remark}
    The first part of \eqref{eq:sigmatest-fullprimal} is a relaxation of the property $\tau^i\sigma^{i+1} = \tau^0\sigma^1$ that would be satisfied by a dual-first variant of the basic PDPS; compare \cref{rem:sigmatest-fulldual-cond}.
\end{remark}

Finally, we also remind that \eqref{eq:bar-gamma-g} and \eqref{eq:bar-gamma-f} for this section simplify to
\begin{equation}
	\label{eq:bargamma-fullprimal}
			\bar\gamma_{GK,j}^i\defeq
		\gamma_{G,j}+\gamma_{K,j}-\alpha_x,
		\quad\text{and}\quad
		\bar\gamma_{F^*,\ell}^{i+1}\equiv\bar\gamma_{F^*,\ell}\defeq
		\begin{cases}
		\gamma_{F^*,\ell},& Q_\ell\Pnl=0,\\
		\gamma_{F^*,\ell}-(p-1)\zeta_\ell, & Q_\ell\Pnl\ne 0.
		\end{cases}
\end{equation}

\begin{Algorithm}
    \caption{Full primal updates}
    \label{alg:acc-full-primal}
	Assume the problem structure \eqref{eq:main-problem-primalonly}, equivalently \eqref{eq:main-problem}.
    For each iteration $i \in \N$, choose a sampling pattern for generating the random set of updated dual blocks $V(i+1) \in \Random(\SAlg_i; \powerset\{1,\ldots,n\})$ with corresponding blockwise probabilities $\iset\nu_\ell^{i+1}\defeq\P[\ell\in V(i+1)\mid\SAlg_{i-1}]>0$. Also choose a rule for the iteration and block-dependent step length parameters $\iset\sigma^{i+1}_\ell,\dset\tau^i_j, \bar\omega^i>0$ from one of \cref{thm:acc2-full-primal}, \ref{thm:acc-full-primal} or \ref{thm:lin-full-primal}.
	Pick an initial iterate $(x^0,y^0)$ and on each iteration $i \in \N$ update all blocks $\nextx_j=P_j\nextx$, ($j=1,\ldots,m$), and  $\nexty_\ell=Q_\ell\nexty$, ($\ell=1,\ldots,n$), of $\nextx$ and $\nexty$ as:
    \begin{align*}
        \nexty_\ell& \defeq \begin{cases}
            (I+\iset\sigma^{i+1}_\ell Q_\ell\subdiff F^*_\ell Q_\ell)^{-1}(\thisy_\ell+\iset\sigma^{i+1}_\ell Q_\ell K(\thisx)),
            &
            \ell\in V(i+1),
            \\
            \thisy_\ell,
            &
            \ell\notin V(i+1),
            \end{cases}
        \\
        \nextx_j& \defeq (I+\dset\tau^i_j P_j\subdiff G_j P_j)^{-1}P_j
        \biggl(\thisx_j-\dset\tau_j^i\kgradconj{\thisx}\sum_{\ell \in V(i+1)}
        \biggl(\nexty_\ell+\frac{\bar\omega^i}{\iset\nu_\ell^{i+1}}(\nexty_\ell-\thisy_\ell)\biggr)\biggr).
    \end{align*}
\end{Algorithm}

\subsection{Accelerated rates}

As in \cref{sec:fulldual}, we start with simple step length rules that yield $O(1/N)$ convergence rates for those blocks that exhibit second-order growth.

\begin{theorem}
	\label{thm:acc2-full-primal}
	Suppose \cref{ass:k-lipschitz,ass:k-nonlinear,ass:gf} hold with $L, L_3\ge0$; $p\in[1,2]$;
	$\gamma_{G,j}+\gamma_{K,j}>0$
	($j=1,\ldots,m$);  and $\gamma_{F^*,\ell}\ge(p-1)\zeta_\ell$ for some $\zeta_\ell\ge0$ when $Q_\ell\Pnl\ne 0$, ($\ell=1,\ldots,n$).
	Let the iterates $\{\thisu=(\thisx, \thisy)\}_{i \in \N}$ be generated by \cref{alg:acc-full-primal} with iteration-independent probabilities $\iset\nu_\ell^i\equiv \iset\nu_\ell$ and step lengths
	\begin{equation}
		\label{eq:step-rules-acc2-full-primal}
		\iset\sigma^{i+1}_\ell \defeq \frac{\iset\sigma^i_\ell}{1+2\iset\sigma^i_\ell\tilde\gamma_{F^*,\ell}},
		\quad\bar\omega^i\equiv 1,\quad
		\quad\text{and}\quad
		\dset\tau^{i+1}_j \defeq \frac{\dset\tau^i_j}{1+2\dset\tau^i_j\tilde\gamma_{G,j}}
	\end{equation}
	with
	$0\le\tilde\gamma_{G,j}<\gamma_{G,j}+\gamma_{K,j}$, ($j=1,\ldots,m$), and
	either $0\le\tilde\gamma_{F^*,\ell}<\iset\nu_\ell\bar\gamma_{F^*,\ell}$ or $\tilde\gamma_{F^*,\ell}=\bar\gamma_{F^*,\ell}=0$ for each $\ell=1,\ldots,n$, $\bar\gamma_{F^*,\ell}$ defined in \eqref{eq:bargamma-fullprimal};
	and initial $\dset\tau_j^0, \iset\sigma_\ell^1>0$ satisfying for some $\rho_\ell\ge0$, ($\ell=1,\ldots,n$), $0<\delta<\kappa<1$, and $w_{j,\ell}^i$ as in \eqref{eq:fullprimal:w-r} the bounds
	\begin{subequations}
		\label{eq:initialisation-acc2-full-primal}
		\begin{align}
		\label{eq:initialisation-acc2-full-primal-sigmatest}
		1-\kappa
		&\ge
		\adaptNorm{
			\sum_{j=1}^m \sqrt{\frac{w_{j,\ell}^i\iset\sigma_\ell^1\dset\tau^0_j}{\iset\nu_\ell}}
			Q_\ell \grad K(\thisx) P_j
        }^2
		\\
		\label{eq:initialisation-acc2-full-primal:taubound}
		\delta&\ge
		\dset\tau^0_j\left(L_3+\frac{mL^2}{2\min_{j=1\ldots m}(\gamma_{G,j}+\gamma_{K,j}-\tilde\gamma_{G,j})}\sum_{\ell=1}^{n}\rho_\ell^2\right),
        \quad\text{and}
		\\
		\label{eq:initialisation-acc2-full-primal:kappa-delta}
		\frac{\kappa-\delta}{1-\delta}&\ge
		2\chi_{V(i+1)}(\ell)(1-\iset\nu_\ell)\bar\gamma_{F^*,\ell}\iset\sigma_\ell^1
		\frac{\bar\gamma_{F^*,\ell}-\tilde\gamma_{F^*,\ell}}{\iset\nu_\ell\bar\gamma_{F^*,\ell}-\tilde\gamma_{F^*,\ell}}
		\quad  (i \in \N;\, j=1,\ldots,m).
		\end{align}
	\end{subequations}
	Assume that
	\begin{subequations}
		\label{eq:locality-acc2-full-primal}
		\begin{align}
		\label{eq:acc2-full-primal-rhol-theta}
		\theta_{I}&\ge
		p^{-p}\textstyle\sum_{\ell=1}^n(\iset\nu_\ell)^2\zeta_\ell^{1-p}\rho_\ell^{2-p}
        \quad\text{and}
        \\
        \label{eq:acc2-full-primal-rhol-omega}
        1&=
        \P[\norm{Q_\ell(\nexty-\realopty)}_{\Pnl}\le\rho_\ell, (\ell=1,\ldots,n) \mid \SAlg_{i-1}].
		\end{align}
	\end{subequations}
	Then $\E[\norm{P_j(x^N-\realoptx)}^2] \to 0$ at the rate $O(1/N)$ for all $j$ such that $\tilde\gamma_{G,j}>0$ and  $\E[\norm{Q_\ell(y^N-\realopty)}^2] \to 0$ at the rate $O(1/N)$ for all $\ell$ such that $\tilde\gamma_{F^*,\ell}>0$.
\end{theorem}

\begin{proof}
	We will use \cref{thm:test-det}, whose conditions we need to verify. With the choice of $\iset S(i)=\emptyset$, $S(i)=\{1,\ldots,m\}$, and $\iset V(i+1)=V(i+1)$ in \cref{alg:acc-full-primal}, we have already verified the nesting conditions \eqref{eq:chilo-definition} and reduced the coupling conditions \eqref{eq:etamu-update-deter} to \eqref{eq:etamu-full-primal}.
	To verify \eqref{eq:etamu-full-primal}, we set $\tauTest^0_j=\eta^1/\dset\tau_j^0$, $\sigmaTest_\ell^2=\eta^1/(\iset\sigma_\ell^2\iset\nu_\ell)$ for some $\eta^1>0$, and update
	\begin{equation}
		\label{eq:acc2-fullprimal-tautest-tildegammag}
		\tauTest^{i+1}_j=(1+2\dset\tau_j^i\tilde\gamma_{G,j})\tauTest^i_j,
		\quad
		\sigmaTest_\ell^{i+2}=(1+2\iset\sigma^{i+1}_\ell\tilde\gamma_{F^*,\ell})\sigmaTest_\ell^{i+1},
        \quad\text{and}\quad
        \nexxt\eta \defeq \this\eta.
	\end{equation}
	Then $\iset\nu_\ell\iset\sigma_\ell^{i+2}\sigmaTest_\ell^{i+2}=\nexxt\eta=\tauTest^{i}_j\dset\tau^{i}_j$ due to \eqref{eq:step-rules-acc2-full-primal} for all $\ell$ and $j$, and \eqref{eq:etamu-full-primal} follows.
    Clearly also \eqref{eq:conditionality-assumptions} holds because the step length and testing parameters are updated deterministically.
	The conditions \eqref{eq:rules-det} follow from \eqref{eq:locality-acc2-full-primal} given that in \cref{ass:k-nonlinear} we can take $\theta_{\TauTest_i\Tau_i}=\eta^{i+1}\theta_{I}=\eta^i\theta_{I}=\sigmaTest_\ell^{i+1}\dset\sigma^{i+1}_\ell\theta_{I}/\iset\nu_\ell$, and $\rho_x$ can be taken infinitely large.

    The step length parameters $\iset\sigma^{i+1}$ and $\dset\tau^i_j$ are non-increasing in $i$ by the defining \eqref{eq:step-rules-acc2-full-primal}.
	Also using \eqref{eq:initialisation-acc2-full-primal-sigmatest}, we thus verify \eqref{eq:sigmatest-fullprimal}. Hence \cref{lemma:sigmatest-fullprimal} establishes \eqref{eq:sigmatest-result}.

	We still need to verify \cref{thm:test-det}\,\cref{item:thm-test:primal} and \cref{item:thm-test:dual}.
	As far as the former is concerned,  $\tauTest^{i+1}_j\le(1+2\dset\tau_j^i\tilde\gamma_{G,j})\tauTest^i_j$ from \eqref{eq:acc2-fullprimal-tautest-tildegammag}. Moreover, after applying \eqref{eq:etamu-full-primal}, \eqref{eq:lambda-lambda} and \eqref{eq:lij} reduce to
	\[
		{\this c_*}=\frac{mL^2\eta^{i+1}}{2\alpha_x}\sum_{\ell=1}^{n}\rho_\ell^2
		\quad\text{and}\quad
		L^i_j=L_3+\frac{mL^2}{2\alpha_x}\sum_{\ell=1}^{n}\rho_\ell^2,
	\]
	which
    Thus, setting $\alpha_x=\min_{j=1\ldots m}(\gamma_{G,j}+\gamma_{K,j}-\tilde\gamma_{G,j})>0$, \cref{thm:test-det}\,\ref{item:thm-test:primal} option \ref{item:thm-test:primal:a} follows for every $j$ from \eqref{eq:initialisation-acc2-full-primal:taubound} and $\dset\tau^{i+1}_j$ being non-increasing.
	Regarding the dual test, we have $\sigmaTest^{i+2}_\ell\le(1+2\iset\sigma_\ell^{i+1}\tilde\gamma_{F^*,\ell}^{i+1})\sigmaTest_\ell^{i+1}$ which together with \eqref{eq:initialisation-acc2-full-primal:kappa-delta} leads to \eqref{eq:sigmatest-update-tilde-eps}. Therefore, \cref{thm:test-det}\,\cref{item:thm-test:dual} option \ref{item:thm-test:dual:b} holds for every $\ell$.

    We can now apply \cref{thm:test-det} to obtain \eqref{eq:convergence-estimate-blockwise}.
	From \eqref{eq:acc2-fullprimal-tautest-tildegammag} we have
	\begin{align*}
		\tauTest^{i+1}_j&=\tauTest^i_j+2\tilde\gamma_{G,j}\eta^{i+1}=\tauTest^i_j+2\tilde\gamma_{G,j}\eta^1
		=\ldots=\tauTest^0_j+2i\tilde\gamma_{G,j}\eta^1
        \quad\text{and}
		\\
		\sigmaTest_\ell^{i+2}&=\sigmaTest_\ell^{i+1}+2\tilde\gamma_{F^*,\ell}\eta^i/\iset\nu_\ell
		=\sigmaTest_\ell^{i+1}+2\tilde\gamma_{F^*,\ell}\eta^1/\iset\nu_\ell
		=\ldots=\sigmaTest_\ell^1+2(i+1)\tilde\gamma_{F^*,\ell}\eta^1/\iset\nu_\ell.
	\end{align*}
	Therefore, for any primal block $j$ with $\tilde\gamma_{G,j}>0$ and dual block $\ell$ with $\tilde\gamma_{F^*,\ell}>0$, $\tauTest^N_j$ and $\sigmaTest_\ell^{N+1}$ grow as $\Omega(N)$, respectively.
    This together with \eqref{eq:convergence-estimate-blockwise} gives the claim.
\end{proof}

We get improved $O(1/N^2)$ rates if all primal blocks exhibit second-order growth:

\begin{theorem}
	\label{thm:acc-full-primal}
	Suppose \cref{ass:k-lipschitz,ass:k-nonlinear,ass:gf} hold with $L, L_3\ge0$; $p\in[1,2]$;
	$\gamma_{G,j}+\gamma_{K,j}>0$, ($j=1,\ldots,m$);
	 and $\gamma_{F^*,\ell}\ge(p-1)\zeta_\ell$ for some $\zeta_\ell$ when $Q_\ell\Pnl\ne 0$, ($\ell=1,\ldots,n$).
	Let the iterates $\{\thisu=(\thisx, \thisy)\}_{i \in \N}$ be generated by \cref{alg:acc-full-primal} with iteration-independent probabilities $\iset\nu_\ell^i\equiv \iset\nu_\ell$ and step length parameters
	\begin{equation}
		\label{eq:step-rules-acc-full-primal}
		\iset\sigma^{i+2}_\ell=\frac{\iset\sigma^{i+1}_\ell}{\bar\omega^i},
		\quad
		\dset\tau^{i+1}_j=\frac{1}{1+2\dset\tau^i_j\tilde\gamma_{G,j}}\frac{\dset\tau^i_j}{\bar\omega^{i+1}},
		\quad\text{and}\quad
		\bar\omega^{i+1}\defeq\max_{j=1\ldots m}\frac{1}{\sqrt{1+2\dset\tau^i_j\tilde\gamma_{G,j}}};
	\end{equation}
	with $0<\tilde\gamma_{G,j}<\gamma_{G,j}+\gamma_{K,j}$, ($j=1,\ldots,m$), and the initial $\bar\omega^0=1$, $\dset\tau_j^0$ and $\iset\sigma_\ell^1$ satisfying for some $\rho_\ell\ge0$, ($\ell=1,\ldots,n$), $0<\delta\le\kappa<1$, and $w_{j,\ell}^i$ as in \eqref{eq:fullprimal:w-r} the bounds $(i \in \N; j=1,\ldots,m)$
	\begin{equation}
        \label{eq:initialisation-acc-full-primal}
		1-\kappa
        \ge
        \adaptNorm{
        	\sum_{j=1}^m \sqrt{\frac{w_{j,\ell}^i\iset\sigma_\ell^1\dset\tau^0_j}{\iset\nu_\ell}}
        	Q_\ell \grad K(\thisx) P_j
        }^2
        \quad
		\delta\ge\breve\tau_j^0\left(L_3+\frac{mL^2}{2\min_{j=1\ldots m}(\gamma_{G,j}+\gamma_{K,j}-\tilde\gamma_{G,j})}\sum_{\ell=1}^{n}\rho_\ell^2\right).
	\end{equation}
	Also assume
	\begin{subequations}
		\label{eq:locality-acc-full-primal}
		\begin{align}
        \theta_{I}&\ge
        p^{-p}\textstyle\sum_{\ell=1}^n(\iset\nu_\ell)^2\zeta_\ell^{1-p}\rho_\ell^{2-p}
        \quad\text{and}
		\\
        1&=
        \P[\norm{Q_\ell(\nexty-\realopty)}_{\Pnl}\le\rho_\ell, (\ell=1,\ldots,n) \mid \SAlg_{i-1}].
		\end{align}
	\end{subequations}
	Then  $\E[\norm{P_j(x^N-\realoptx)}^2] \to 0$ at the rate $O(1/N^2)$ for all $j$.
\end{theorem}

\begin{proof}
	We will use \cref{thm:test-det} whose conditions we need to verify. With the choice of $\iset S(i)=\emptyset$, $S(i)=\{1,\ldots,m\}$, and $\iset V(i+1)=V(i+1)$ in \cref{alg:acc-full-primal}, we have already verified the nesting conditions \eqref{eq:chilo-definition} and reduced the coupling conditions \eqref{eq:etamu-update-deter} to \eqref{eq:etamu-full-primal}.
	To verify \eqref{eq:etamu-full-primal}, we set $\tauTest^0_j=\eta^1/\dset\tau^0_j$ and $\sigmaTest_\ell^2 \defeq \eta^1/(\iset\nu_\ell\iset\sigma^2_\ell)$ for some $\eta^1>0$, and update
	\begin{equation}
		\label{eq:acc-fullprimal-tautest-tildegammag}
		\tauTest^{i+1}_j \defeq (1+2\dset\tau_j^i\tilde\gamma_{G,j})\tauTest^i_j,
        \quad
        \nexxt\sigmaTest_\ell \defeq \this\sigmaTest_\ell,
        \quad\text{and}\quad
        \eta^{i+1}=\eta^i/\bar\omega^i.
	\end{equation}
	Then from \eqref{eq:step-rules-acc-full-primal}, we inductively get $\iset\nu_\ell\sigmaTest_\ell^{i+2}\iset\sigma_\ell^{i+2}=\iset\nu_\ell\sigmaTest_\ell^{i+1}\iset\sigma_\ell^{i+1}/\bar\omega^i=\nexxt\eta$ for all $\ell$.
	From \eqref{eq:step-rules-acc-full-primal}, we also have inductively for all $j$,
	$\tauTest^{i+1}_j\dset\tau^{i+1}_j=\tauTest^i_j\dset\tau^i_j/\bar\omega^{i+1}=\eta^{i+2}$.
	Therefore \eqref{eq:etamu-full-primal} holds.
	Then, the conditions \eqref{eq:rules-det} follow from \eqref{eq:locality-acc-full-primal} given that $\bar\omega^i\le 1$ and in \cref{ass:k-nonlinear} we can take
	$\theta_{\TauTest_i\Tau_i}=\eta^{i+1}\theta_{I}=
	\eta^i\theta_{I}/\bar\omega^i=\sigmaTest_\ell^{i+1}\sigma^{i+1}_\ell\theta_{I}/(\iset\nu_\ell\bar\omega^i)$, and $\rho_x$ can be taken infinitely large.
    Clearly also \eqref{eq:conditionality-assumptions} holds because the step length and testing parameters are updated deterministically.

    We now verify \eqref{eq:sigmatest-result}.
	From \eqref{eq:step-rules-acc-full-primal} we obtain
	\[
		\begin{split}
		\iset\sigma^{i+2}_\ell\dset\tau^{i+1}_j
		&=
        \frac{\iset\sigma^{i+1}_\ell\dset\tau^i_j}{\bar\omega^i\bar\omega^{i+1}(1+2\tilde\gamma_{G,j}\dset\tau^i_j)}
		\le
		\iset\sigma^{i+1}_\ell\dset\tau^i_j
		\sqrt{\frac{1+2\tilde\gamma_{G,j}\dset\tau^{i-1}_j}{1+2\tilde\gamma_{G,j}\dset\tau^i_j}}
		\le\ldots\le
		\iset\sigma^2_\ell\dset\tau^1_j
		\sqrt{\frac{1+2\tilde\gamma_{G,j}\dset\tau^0_j}{1+2\tilde\gamma_{G,j}\dset\tau^i_j}}
		\\
		&=
		\iset\sigma^1_\ell\dset\tau^0_j
		\frac{1}{\bar\omega^1\sqrt{1+2\tilde\gamma_{G,j}\dset\tau^0_j}}\frac{1}{1+2\tilde\gamma_{G,j}\dset\tau^i_j}
		\le
		\iset\sigma^1_\ell\dset\tau^0_j.
		\end{split}
	\]
	This and \eqref{eq:initialisation-acc-full-primal} verify \eqref{eq:sigmatest-fullprimal}. Thus \cref{lemma:sigmatest-fullprimal} establishes \eqref{eq:sigmatest-result}.
    
    We still need to verify \cref{thm:test-det}\,\cref{item:thm-test:primal} and \cref{item:thm-test:dual}.
    Regarding the former,  $\tauTest^{i+1}_j\le(1+2\dset\tau_j^i\tilde\gamma_{G,j})\tauTest^i_j$ from \eqref{eq:acc-fullprimal-tautest-tildegammag}. Moreover, after applying \eqref{eq:etamu-full-primal}, equalities \eqref{eq:lambda-lambda} and \eqref{eq:lij} reduce to
    \[
	    {\this c_*}=\frac{mL^2\eta^{i+1}}{2\alpha_x}\sum_{\ell=1}^{n}\rho_\ell^2
	    \quad\text{and}\quad
	    L^i_j=L_3+\frac{mL^2}{2\alpha_x}\sum_{\ell=1}^{n}\rho_\ell^2.
    \]
    Thus, setting $\alpha_x=\min_{j=1\ldots m}(\gamma_{G,j}+\gamma_{K,j}-\tilde\gamma_{G,j})>0$, \cref{thm:test-det}\,\ref{item:thm-test:primal} option \ref{item:thm-test:primal:a} follows for every $j$ from the second inequality in \eqref{eq:initialisation-acc-full-primal} and $\dset\tau^{i+1}_j$ being decreasing.
	As for \cref{thm:test-det}\,\cref{item:thm-test:dual},
	$\sigmaTest_\ell^{i+1}=\sigmaTest^{i+2}_\ell\le(1+2\chi_{V(i+1)}(\ell)\sigma_\ell^{i+1}\bar\gamma_{F^*,\ell}^{i+1})\sigmaTest_\ell^{i+1}$
	trivially as we have assumed $\bar\gamma_{F^*,\ell}^{i+1}\ge0$.
	Thus \cref{thm:test-det}\,\ref{item:thm-test:dual} option \ref{item:thm-test:dual:a} holds for every $\ell$.

    We can now use \cref{thm:test-det} to verify \eqref{eq:convergence-estimate-blockwise}.	Multiplying the $\tau$ update of \eqref{eq:step-rules-acc-full-primal} by $2\tilde\gamma_{G,j}$, plugging in $\bar\omega^{i+1}$, and taking the inverse, we get
	\[
		(2\dset\tau^{i+1}_j\tilde\gamma_{G,j})^{-1} =
		\frac{1+2\dset\tau^i_j\tilde\gamma_{G,j}}{2\dset\tau^i_j\tilde\gamma_{G,j}\sqrt{1+\min_{j=1\ldots m}(2\dset\tau^i_j\tilde\gamma_{G,j})}}
		=\frac{1+(2\dset\tau^i_j\tilde\gamma_{G,j})^{-1}}{\sqrt{1+(\max_{j=1\ldots m}(2\dset\tau^i_j\tilde\gamma_{G,j})^{-1})^{-1}}}.
	\]
	We then apply \cref{lemma:max-for-n2} with $z_j^i=(2\dset\tau^i_j\tilde\gamma_{G,j})^{-1}$ to obtain $\max_{j=1\ldots m}(2\dset\tau^N_j\tilde\gamma_{G,j})^{-1}\le \bar z_0 +N/2$ with $\bar z_0>0$.
	Then from \eqref{eq:acc-fullprimal-tautest-tildegammag}, we have
	\[
		\begin{split}
		\tauTest^{N+1}_j&\ge
		(1+\min_{j=1\ldots m}(2\dset\tau_j^i\tilde\gamma_{G,j}))\tauTest^N_j
		\ge \Bigl(1+\frac{1}{\bar z_0+N/2}\Bigr)\tauTest^N_j
		=\frac{2\bar z_0+N+2}{2\bar z_0+N}\tauTest^N_j
		\\
		&=
		\frac{2\bar z_0+N+2}{2\bar z_0+N}\frac{2\bar z_0+N+1}{2\bar z_0 +N-1}\tauTest^{N-1}_j
		=\ldots
		=\frac{(2\bar z_0+N+2)(2\bar z_0+N+1)}{2\bar z_0(2\bar z_0+1)}\tauTest^0_j.
	\end{split}
	\]
	Therefore, $\tauTest^{N}_j$ grows as $\Omega(N^2)$, and we obtain the claimed convergence rates from \eqref{eq:convergence-estimate-blockwise}.
\end{proof}

\subsection{Linear convergence}

If all the primal and dual blocks exhibit second-order growth, i.e., $\bar\gamma_{F^*,\ell}>0$ and $\gamma_{G,j}+\gamma_{K,j}>0$, we obtain linear convergence:

\begin{theorem}
	\label{thm:lin-full-primal}
	Suppose \cref{ass:k-lipschitz,ass:k-nonlinear,ass:gf} hold with $L, L_3\ge0$; $p\in[1,2]$;
	$\gamma_{G,j}+\gamma_{K,j}>0$, ($j=1,\ldots,m$).
    Let the iterates $\{\thisu=(\thisx, \thisy)\}_{i \in \N}$ be generated by \cref{alg:acc-full-primal} with iteration-independent $\iset\nu_\ell^i\equiv \iset\nu_\ell$ and step lengths
    \begin{subequations}
    	\label{eq:step-rules-lin-full-primal}
        \begin{align}
        \label{eq:step-rules-lin-full-primal:tausigma}
    	\dset\tau^{i+1}_j& \defeq
    	\frac{\dset\tau^i_j}{(1+2\dset\tau^i_j\tilde\gamma_{G,j})\bar\omega},
    	\quad
    	\iset\sigma^{i+2}_\ell \defeq \frac{\iset\sigma^{i+1}_\ell}{(1+2\iset\sigma^{i+1}_\ell\tilde\gamma_{F^*,\ell})\bar\omega},
    	\quad\text{and}\quad
        \\
        \label{eq:step-rules-lin-full-primal:omega}
    	\bar\omega^i&\equiv\bar\omega\defeq\max\biggl\{\max_{j=1\ldots m}\frac{1}{1+2\dset\tau^{0}_j\tilde\gamma_{G,j}},
    	\max_{\ell=1\ldots n}\frac{1}{1+2\iset\sigma^1_\ell\tilde\gamma_{F^*,\ell}}\biggr\}
    	\end{align}
    \end{subequations}
	with
	$0<\tilde\gamma_{G,j}<\gamma_{G,j}+\gamma_{K,j}$, ($j=1,\ldots,m$), and
	$0<\tilde\gamma_{F^*,\ell}<\iset\nu_\ell\bar\gamma_{F^*,\ell}$, $(\ell=1,\ldots,n)$, $\bar\gamma_{F^*,\ell}$ defined in \eqref{eq:bargamma-fullprimal};
	and initial $\dset\tau^0_j,\iset\sigma^1_\ell>0$ satisfying for some $0<\delta<\kappa<1$, $\rho_\ell\ge 0$ ($\ell=1,\ldots,n$), with $w_{j,\ell}^i$ as in \eqref{eq:fullprimal:w-r} the bounds
	\begin{subequations}
		\label{eq:initialisation-lin-full-primal}
		\begin{align}
		\label{eq:initialisation-lin-full-primal:sigmabound}
		1-\kappa
		&\ge
		\adaptNorm{
			\sum_{j=1}^m \sqrt{\frac{w_{j,\ell}^i\iset\sigma_\ell^1\dset\tau^0_j}{\iset\nu_\ell}}
			Q_\ell \grad K(\thisx) P_j
        }^2,
		\\
		\delta & \ge \dset\tau^0_j\left(L_3+\frac{mL^2}{2\min_{j=1\ldots m}(\gamma_{G,j}+\gamma_{K,j}-\tilde\gamma_{G,j})}
			\sum_{\ell=1}^{n}\rho_\ell^2\right),
		\quad\text{and}
		\\
		\label{eq:initialisation-lin-full-primal:kappa-delta}
		\frac{\kappa-\delta}{1-\delta}&\ge
		2(1-\iset\nu_\ell)\bar\gamma_{F^*,\ell}\iset\sigma^1_\ell
		\frac{\bar\gamma_{F^*,\ell}-\tilde\gamma_{F^*,\ell}}{\iset\nu_\ell\bar\gamma_{F^*,\ell}-\tilde\gamma_{F^*,\ell}}
		\qquad
		(\ell\in V(i+1);\, j=1,\ldots,m;\, i \in \N).
		\end{align}
	\end{subequations}
    Further assume that
	\begin{subequations}
		\label{eq:locality-lin-full-primal}
		\begin{align}
        \theta_{I}&\ge
        p^{-p}\bar\omega\textstyle\sum_{\ell=1}^n(\iset\nu_\ell)^2\zeta_\ell^{1-p}\rho_\ell^{2-p}
        \quad\text{and}
        \\
		1&=
        \P[\norm{Q_\ell(\nexty-\realopty)}_{\Pnl}\le\rho_\ell, (\ell=1,\ldots,n) \mid \SAlg_{i-1}].
		\end{align}
	\end{subequations}
	Then $\E[\norm{P_j(x^N-\realoptx)}^2] \to 0$ and $\E[\norm{Q_\ell(y^N-\realopty)}^2] \to 0$ at the linear rate $O((1/\bar\omega)^N)$ for all $j \in \{1, \ldots, m\}$ and $\ell \in \{1, \ldots, n\}$.
\end{theorem}

\begin{proof}
    We will use \cref{thm:test-det}, whose conditions we need to verify. With the choice of $\iset S(i)=\emptyset$, $S(i)=\{1,\ldots,m\}$, and $\iset V(i+1)=V(i+1)$ in \cref{alg:acc-full-primal}, we have already verified the nesting conditions \eqref{eq:chilo-definition} and reduced the coupling conditions \eqref{eq:etamu-update-deter} to \eqref{eq:etamu-full-primal}.
    To verify \eqref{eq:etamu-full-primal}, we set $\tauTest^0_j=\eta^1/\dset\tau^0_j$ and $\sigmaTest_\ell^2 \defeq \eta^1/(\iset\nu_\ell\iset\sigma^2_\ell)$ for some $\eta^1>0$, and update
    \begin{equation}
	    \label{eq:lin-fullprimal-tautest-tildegammag}
	    \tauTest^{i+1}_j \defeq (1+2\dset\tau^i_j\tilde\gamma_{G,j})\tauTest^i_j,
	    \quad
	    \nexxt\sigmaTest_\ell \defeq (1+2\iset\sigma^i_\ell\tilde\gamma_{F^*,\ell})\sigmaTest_\ell^i,
	    \quad\text{and}\quad
	    \eta^{i+1}=\eta^i/\bar\omega.
    \end{equation}
    Then from \eqref{eq:step-rules-lin-full-primal}, we inductively get $\iset\nu_\ell\sigmaTest_\ell^{i+2}\iset\sigma_\ell^{i+2}=\iset\nu_\ell\sigmaTest_\ell^{i+1}\iset\sigma_\ell^{i+1}/\bar\omega=\nexxt\eta$ for all $\ell$ and $\tauTest^{i+1}_j\dset\tau^{i+1}_j=\tauTest^i_j\dset\tau^i_j/\bar\omega=\eta^{i+2}$ for all $j$, therefore, \eqref{eq:etamu-full-primal} holds.
    Then, the conditions \eqref{eq:rules-det} follow from \eqref{eq:locality-lin-full-primal} given that in \cref{ass:k-nonlinear} we can take
    $\theta_{\TauTest_i\Tau_i}=\eta^{i+1}\theta_{I}=
    \eta^i\theta_{I}/\bar\omega=\sigmaTest_\ell^{i+1}\sigma^{i+1}_\ell\theta_{I}/(\iset\nu_\ell\bar\omega)$, and $\rho_x$ can be taken infinitely large.
    Clearly also \eqref{eq:conditionality-assumptions} holds because the step length and testing parameters are updated deterministically.

    We now verify \eqref{eq:sigmatest-result}. We start by proving by induction that
    \begin{equation}
	    \label{eq:omega-linear-full-primal}
	    \bar\omega=
	    \max\biggl\{\max_{j=1\ldots m}\frac{1}{1+2\dset\tau^i_j\tilde\gamma_{G,j}},
	    \max_{\ell=1\ldots n}\frac{1}{1+2\iset\sigma^{i+1}_\ell\tilde\gamma_{F^*,\ell}}\biggr\},
    \end{equation}
    in other words
    \[
        \bar\omega^{-1}=1+\min\left\{\min_{j=1\ldots m} 2\dset\tau^i_j\tilde\gamma_{G,j}, \min_{\ell=1\ldots n}2\iset\sigma^{i+1}_\ell\tilde\gamma_{F^*,\ell}\right\}.
    \]
    The inductive base for $i=0$ holds by \eqref{eq:step-rules-lin-full-primal:omega}.
    Using \eqref{eq:step-rules-lin-full-primal:tausigma},
    \begin{multline*}
    	\min\biggl\{
    		\min_{j=1\ldots m} 2\dset\tau^{i+1}_j\tilde\gamma_{G,j},
    		\min_{\ell=1\ldots n} 2\iset\sigma^{i+2}_\ell\tilde\gamma_{F^*,\ell}
    	\biggr\}
    	=
    	\frac{1}{\bar\omega}
    	\min\biggl\{
    		\min_{j=1\ldots m} \frac{1}{1+(2\dset\tau^i_j\tilde\gamma_{G,j})^{-1}},
    		\min_{\ell=1\ldots n} \frac{1}{1+(2\iset\sigma^{i+1}_\ell\tilde\gamma_{F^*,\ell})^{-1}}
    	\biggr\}
    	\\
    	=
    	\frac{1}{\bar\omega}
    	\frac{1}{1+\min^{-1}\Bigl\{\min_{j=1\ldots m} 2\dset\tau^i_j\tilde\gamma_{G,j},
    		\min_{\ell=1\ldots n}2\iset\sigma^{i+1}_\ell\tilde\gamma_{F^*,\ell}\Bigr\}}
    	=
    	\min\left\{\min_{j=1\ldots m} 2\dset\tau^i_j\tilde\gamma_{G,j},
    	\min_{\ell=1\ldots n}2\iset\sigma^{i+1}_\ell\tilde\gamma_{F^*,\ell}\right\}.
    \end{multline*}
    This establishes the inductive step, hence \eqref{eq:omega-linear-full-primal}, which in turn shows that $\dset\tau^i_j$ and $\iset\sigma^{i+1}_\ell$ as updated according to \eqref{eq:step-rules-lin-full-primal:tausigma} are non-increasing in $i$. Also using \eqref{eq:initialisation-lin-full-primal}, this proves \eqref{eq:sigmatest-fullprimal}. Thus \cref{lemma:sigmatest-fullprimal} verifies \eqref{eq:sigmatest-result}.

	We need to verify \cref{thm:test-det}\,\cref{item:thm-test:primal} and \cref{item:thm-test:dual}.
	As for the former, \eqref{eq:lambda-lambda} and \eqref{eq:lij} reduce to
    \[
	    {\this c_*}=\frac{mL^2}{2\alpha_x}\sum_{\ell=1}^{n}\rho_\ell^2\eta^{i+1}
	    \quad\text{and}\quad
	    L^i_j=L_3+\frac{mL^2}{2\alpha_x}\sum_{\ell=1}^{n}\rho_\ell^2,
    \]
    so \eqref{eq:initialisation-lin-full-primal}, together with non-increasing $\dset\tau^i_j$ and the update rule for $\tauTest^{i+1}_j$ in \eqref{eq:lin-fullprimal-tautest-tildegammag}, verify \cref{thm:test-det}\,\ref{item:thm-test:primal} option \ref{item:thm-test:primal:a} for every $j$ and $\alpha_x=\min_{j=1\ldots m}(\gamma_{G,j}+\gamma_{K,j}-\tilde\gamma_{G,j})$.
    Regarding the latter, since we take $\tilde\gamma_{F^*,\ell}<\iset\nu_\ell\bar\gamma_{F^*,\ell}$, we obtain \eqref{eq:sigmatest-update-tilde-eps} using the last inequality of \eqref{eq:initialisation-lin-full-primal} and that $\iset\sigma^{i+1}_\ell$ is non-increasing by definition in \eqref{eq:step-rules-lin-full-primal}. Hence \cref{thm:test-det}\,\ref{item:thm-test:dual} option \ref{item:thm-test:primal:b} holds for every $\ell$.
    
    Therefore, we can apply \cref{thm:test-det} to obtain \eqref{eq:convergence-estimate-blockwise}. By \eqref{eq:lin-fullprimal-tautest-tildegammag} and \eqref{eq:omega-linear-full-primal},
    \begin{align*}
	    \tauTest_j^{N+1}&=(1+2\dset\tau^N_j\tilde\gamma_{G,j})\tauTest_j^N
	    \ge
	    \tauTest_j^N/\bar\omega
	    \ge\ldots\ge
        \tauTest_j^0/\bar\omega^{N+1}
        \quad\text{and}
        \\
	    \sigmaTest_\ell^{N+1}&=(1+2\iset\sigma^N_\ell\tilde\gamma_{F^*,\ell})\sigmaTest_\ell^N
	    \ge
	    \sigmaTest_\ell^N/\bar\omega
	    \ge\ldots\ge
	    \sigmaTest_\ell^1/\bar\omega^{N}.
    \end{align*}
    Applying these estimates in \eqref{eq:convergence-estimate-blockwise} establishes the  claimed linear convergence rates.
\end{proof}

\begin{remark}[Stochastic sum-sampling forward--backward splitting]
	\label{remark:full-primal-fb-splitting}
	Consider the problem \eqref{eq:fb-problem} with $F^*(y)=\delta_{\{\mathbb{1}\}}$ for $\mathbb{1} \defeq (1,\ldots,1) \in \R^n$ and $\kgradconj{x} y=\sum_{\ell=1}^n \grad J_\ell(x) y_{(\ell)}$ with $y=(y_{(1)},\ldots, y_{(n)})$.
	Taking $Q_\ell y \defeq (0,\ldots,0,y_{(\ell)},0,\ldots,0)$, it follows that $(I+\dset\sigma^{i+1}_\ell Q_\ell\subdiff F^*_\ell Q_\ell)^{-1} \equiv (0,\ldots,0,1,0,\ldots,0)$.
	Consequently $\thisy \equiv \mathbb{1}$ on all iterations, so that with just a single primal block with corresponding step length $\dset\tau^i=\dset\tau^i_1$, \cref{alg:acc-full-primal} reduces to
	\[
		\nextx \defeq
		(I+\dset\tau^i \subdiff G)^{-1}\biggl(\thisx-\dset\tau^i\sum_{\ell \in V(i+1)}\grad J_\ell(\thisx)\biggr).
	\]
	With random $V(i+1)$, this is a forward--backward splitting method that stochastically samples $\sum_\ell J_\ell$ in \eqref{eq:fb-problem}.
	We can take any $\gamma_{F^*,\ell} \in (0, \infty)$, which in \cref{thm:acc2-full-primal,thm:acc-full-primal,thm:lin-full-primal} also allows us to take $\zeta_\ell$ arbitrarily large and $\iset\sigma_\ell^i>0$ arbitrarily small. Consequently, the systems of step length bounds \eqref{eq:initialisation-acc2-full-primal} and  \eqref{eq:initialisation-lin-full-primal} reduce to their second part (with first and third part unnecessary), and  \eqref{eq:initialisation-acc-full-primal} reduces to its second part. In other words, we only need to choose $\dset\tau^0$ sufficiently small.
\end{remark}

\section{Numerical experience}
\label{sec:numerical}

We will now study the performance of our proposed methods on two application problems: diffusion tensor imaging (DTI), which is a form of magnetic resonance imaging (MRI), and electrical impedance tomography (EIT).

\subsection{Diffusion tensor imaging}
\label{sec:dti}

Diffusion tensor imaging is covered by the Stejskal--Tanner equation: given a tensor field $x: \Omega \to \Sym^2(\R^3)$, associating each point on the domain $\Omega \subset \R^3$ with a of symmetric $2$-tensor (presentable as a symmetric $3 \times 3$ matrix), and a non-diffusion-weighted image $s_0: \Omega \to \R$, the diffusion-weighted image $s_k: \Omega \to \R$ corresponding to a diffusion-sensitising gradient $b_k \in \R^3$ is given by
\begin{equation}
    \label{eq:sj}
    s_k(\xi)=s_0(\xi)e^{-\iprod{x(\xi)b_k}{b_k}}
    \quad (\xi \in \Omega).
\end{equation}
At each spatial point $\xi$, the tensor $x(\xi)$ models the covariance of a Gaussian probability distribution for the spatial directions of the diffusion of water at that point. Models more advanced than DTI, such as HARDI, consider composite probability distributions at each $\xi$. For our purposes a simplified DTI model will be sufficient.
One can measure $s_k$ and $s_0$ by suitable MRI pulse sequences, inversion of a Fourier transform, and taking the absolute value of a complex number; for details we refer to \cite{basser2002diffusion,kingsley2006introduction}, among others. We recommend \cite{nishimura1996principles} as an introduction to MRI.

\begin{figure}%
	\centering%
	\begin{subfigure}[t]{0.333\textwidth}%
		\ifprint%
			\includegraphics[width=\textwidth,trim=30 0 0 0]{helix.pdf}%
		\else%
			\includegraphics[width=\textwidth,trim=30 0 0 0]{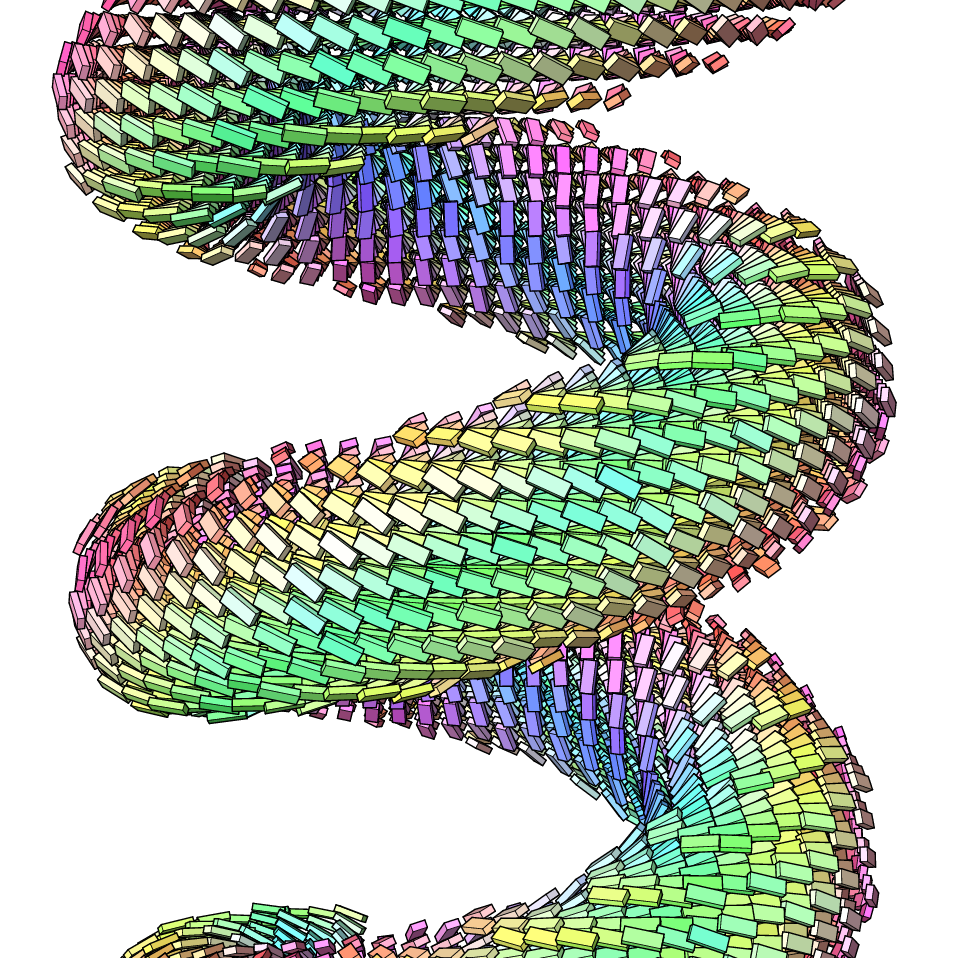}%
		\fi%
		\caption{Original helix}%
		\label{fig:helix-orig}%
	\end{subfigure}%
	\begin{subfigure}[t]{0.333\textwidth}%
		\ifprint%
			\includegraphics[width=\textwidth,trim=30 0 0 0]{dti_lsq.pdf}%
		\else%
			\includegraphics[width=\textwidth,trim=30 0 0 0]{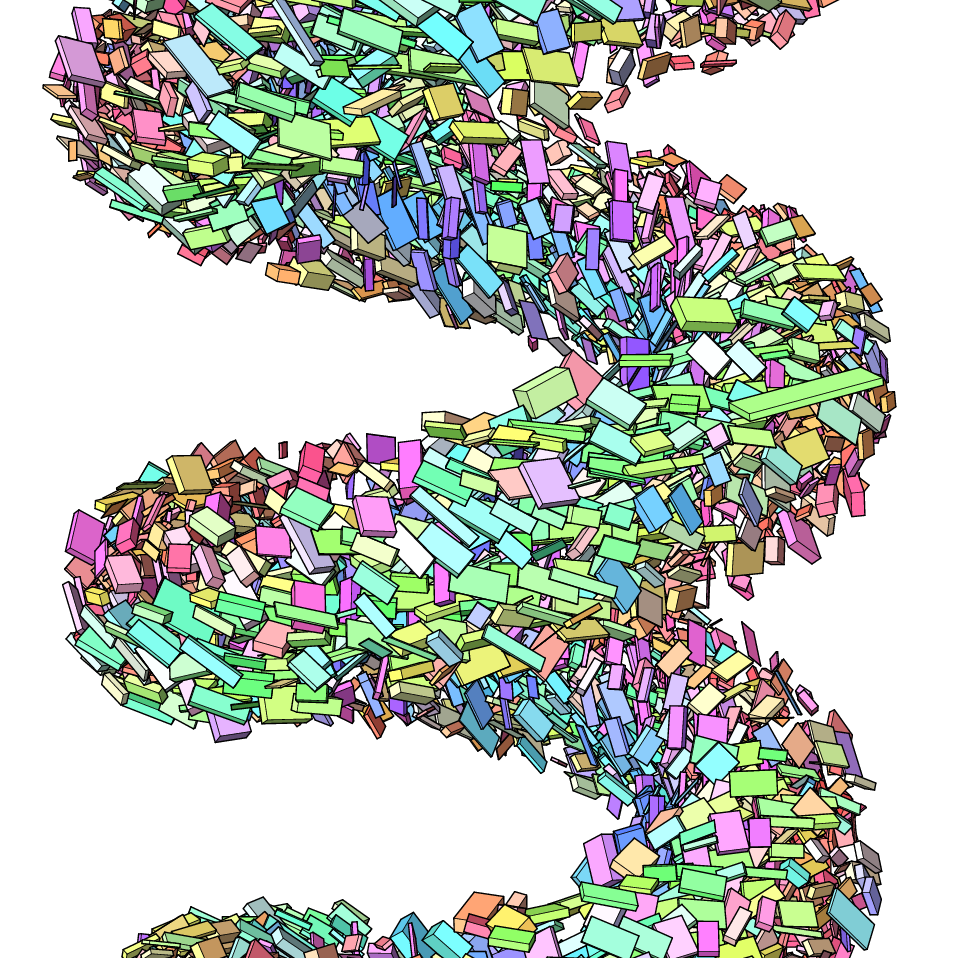}%
		\fi%
		\caption{Least squares reconstruction}%
		\label{fig:helix-lsq}%
	\end{subfigure}%
	\begin{subfigure}[t]{0.333\textwidth}%
		\ifprint%
			\includegraphics[width=\textwidth,trim=30 0 0 0]{dti_two1_wy.pdf}%
		\else%
			\includegraphics[width=\textwidth,trim=30 0 0 0]{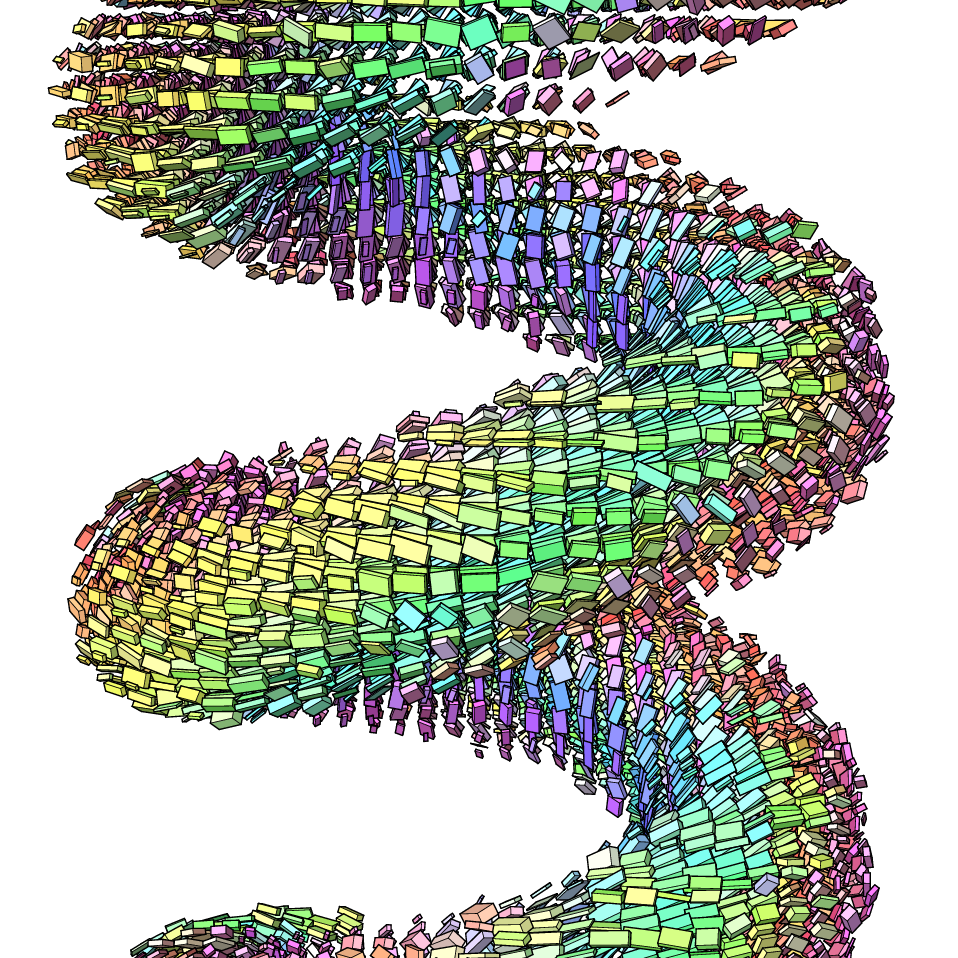}%
		\fi%
		\caption{Regularised reconstruction}%
		\label{fig:helix-recon}%
	\end{subfigure}%
	\caption{Visualisation of original helix data (\subref{fig:helix-orig}) and the reconstruction from noisy diffusion-weighted measurements. The reference least squares reconstruction in (\subref{fig:helix-lsq}) is based on linearising \eqref{eq:sj} with respect to $x$ by taking the logarithm. The regularised reconstruction (\subref{fig:helix-recon}) is the numerical solution of \eqref{eq:dti-model} for $\alpha=0.005$ with the variant \ref{item:dti:twoblock} of our method after 10000 iterations.
	The visualisation, generated with Teem \cite{teem}, displays the tensor at each voxel of the 3D volume as a cuboid oriented along the eigenvectors of the tensor, size of each side proportional to the corresponding eigenvalue. The cuboids are also colour-coded based on the principal eigenvector. Tensors with too small eigenvalues are suppressed; in essence this suppresses the background outside the helix, letting the latter to be inspected unobstructedly.}
	\label{fig:helix}
 \end{figure}
 
We want to determine $x$ from noisy measurements of $s_0$ and $s_k$, ($k=1,\ldots,N$). Clearly, \eqref{eq:sj} can be converted into an invertible system of linear equations with respect to $x$ if $N \ge 6$ and the tensors $b_k \otimes b_k$ are linearly independent. With noise involved, to get a good-quality image, we want to obtain a regularised solution. We therefore consider a problem of the form \eqref{eq:generic-problem} where $G$ is a data term modelling \eqref{eq:sj} along with any noise, and $F \circ K$ is the regulariser. Ideally, our data term would model the Rician noise distribution, which is the distribution of the absolute value of a complex number when the latter has Gaussian noise distribution. However, the numerical treatment of the Rician distribution is quite involved -- we refer to \cite{martin2013tgvdti,getreuer2011rician} for some variational approaches -- and instead of modelling it directly, a more fruitful approach may be to work with complex data directly, even incorporating the Fourier transform into our model. For the purposes of the present work, since we only use synthetic data, we will therefore assume that the noise in $s_k$ is Gaussian.
We note that \eqref{eq:dti-model} in infinite dimensions requires the use of the Banach space of functions of bounded deformation, so, since our algorithms require Hilbert spaces, only discretised versions of the model can be considered. Consequently, taking the discretised domain $\Omega_d \defeq \{1, \ldots, n_1\} \times \{1, \ldots, n_2\} \times \{1,\ldots n_3\}$ and incorporating total deformation regularisation with parameter $\alpha>0$, we seek to solve
\begin{equation}
    \label{eq:dti-model}
    \min_{x: \Omega_d \to \Sym^2(\R^3)}~ \frac{1}{2}\norm{T(x)}^2 + \alpha\norm{\symD_d x}_{F,1},
    \quad
    [T(x)]_k \defeq s_k(\xi)-s_0(\xi)e^{-\iprod{x(\xi)b_k}{b_k}}
    \quad
    (k=1,\ldots,N).
\end{equation}
Here $[\symD_d x](\xi) \in \Sym^3(\R^3)$ is forward--differences discretisation of the symmetrised gradient, a symmetric third-order tensor. The $F,1$-norm is based on taking pointwise the Frobenius norm of $[\symD_d x](\xi)$ and integration of the space ($1$-norm).
This model is sightly simplified from our previous work in \cite{tuomov-dtireg,escoproc,ipmsproc}, where second-order total generalised variation regularisation was considered and we included a positivity semi-definiteness constraint on $x(\xi)$.

To write \eqref{eq:dti-model} in the form \eqref{eq:main-problem}, we take with $y=(\mu,\lambda)$ the functions
\[
    G(x) \defeq 0,
    \ \,%
    K(x) \defeq (\symD_d x, T(x)),
    \ \,%
    F^*(y) \defeq F_{\mu}^*(\mu)+F_{\lambda}^*(\lambda),
    \ \,%
    F_{\mu}^*(\mu) \defeq \delta_{\alpha \B}(\mu),
    \ \,%
    F_{\lambda}^*(\lambda) \defeq \frac{1}{2}\norm{\lambda}^2.
\]
Here $\B$ is the product of the voxelwise unit balls of $\Sym^3(\R^3)$ over $\Omega_d$. To better satisfy the conditions of our convergence theorems, we replace $F_{\mu}^*$ by $F_{\mu,\gamma}^*(\mu) \defeq \delta_{\alpha \B}(\mu) + \gamma\inv\alpha\norm{\mu}^2$ with $\gamma=10^{-9}$.
This is the same as applying Moreau--Yosida regularisation to $\norm{\freevar}_{F,1}$ in \eqref{eq:dti-model}.

We generated our test data, a simple helix depicted in \cref{fig:helix}, with the Teem toolkit \cite{teem}. The dimensions are $n_1 \times n_2 \times n_3 = 38 \times 39 \times 40$. In the background, outside the helix, the tensors are fully isotropic with the eigenvalues of 10\% of the maximal eigenvalue of the tensors within the helix. The exact generation details can be deciphered from our codes \cite{nlpdhgm-block-code} written in Julia \cite{bezanson2017julia}.
After generating the helix data, we took $s_0(\xi)=\norm{x(\xi)}_F$.
Then we generated $s_k$, ($k=1,\ldots, 6$), from the Stejskal--Tanner equation \eqref{eq:sj} with the diffusion-sensitising gradients $b_1=(1, 0, 0)$, $b_2=(0, 1, 0)$, $b_3=(0, 0, 1)$, $b_4=(\sqrt 2, \sqrt 2, 0)$, $b_5=(\sqrt 2, 0, \sqrt 2)$, and $b_6=(0, \sqrt 2, \sqrt 2)$. To these diffusion-weighted images we added synthetic Gaussian noise of standard deviation 30\% of the mean magnitude of $s_0$.
As the regularisation parameter in the model \eqref{eq:dti-model} we took $\alpha=0.005$.

\begin{figure}
    \centering
    \begin{subfigure}[t]{0.495\textwidth}
        \begin{tikzpicture}
    \begin{axis}[%
        width=\linewidth,
        xmode=log,
        xmin=1,xmax=10000,
        scaled x ticks=false,
        tick label style={/pgf/number format/fixed, /pgf/number format/set thousands separator={\,}},
        xminorticks=true,
        minor x tick num=1,
        ymode=log,
        yminorticks=true,
        ytick={40, 50, 70, 100, 140, 200},
        log ticks with fixed point,
        minor y tick num=1,
        axis x line*=bottom,
        axis y line*=left,
        legend style={legend pos=south west,inner sep=0pt,outer sep=10pt,legend cell align=left,align=left,draw=none,fill=none,font=\scriptsize}
        ]

        \addplot [color=Set2-E, line width=1pt]
            table[x=iter,y=function_value]{dti_non1_tau0.1.txt};
        \addlegendentry{$\tau=0.1/R$};

        \addplot [color=Set2-F, line width=1pt]
            table[x=iter,y=function_value]{dti_non1_tau0.5.txt};
        \addlegendentry{$\tau=0.5/R$};

        \addplot [color=Set2-A, line width=1pt]
            table[x=iter,y=function_value]{dti_non1_tau1.0.txt};
        \addlegendentry{$\tau=1/R$};

        \addplot [color=Set2-G, line width=1pt]
            table[x=iter,y=function_value]{dti_non1_tau5.0.txt};
        \addlegendentry{$\tau=5/R$};
    \end{axis}

\end{tikzpicture}
        \caption{Multiple step length parametrisations of the non-block-adapted reference algorithm \ref{item:dti:reference} to justify the choice $\tau=1/R$.}
        \label{fig:dtitau}
    \end{subfigure}
    \begin{subfigure}[t]{0.495\textwidth}
        \def\DTIVALmult{1}\def\DTIVALx{iter}
        \def\DTIVALiter{iter}
\def\DTIVALmultten{10}
\ifx\DTIVALx\DTIVALiter
    \def\DTIVALxlimit{xmin=1,xmax=1e4}
\else
    \def\DTIVALxlimit{xmin=0.1}
\fi
\ifx\DTIVALmult\DTIVALmultten
    \def\DTIVALysetup{ytick={4000, 5000, 7000, 10000, 15000,20000}}
\else
    \def\DTIVALysetup{ytick={40,50, 70, 100, 140, 200}}
\fi
\begin{tikzpicture}
    \begin{axis}[%
        width=\linewidth,
        xmode=log,
        \DTIVALxlimit,
        scaled x ticks=false,
        tick label style={/pgf/number format/fixed, /pgf/number format/set thousands separator={\,}},
        xminorticks=true,
        minor x tick num=1,
        ymode=log,
        yminorticks=true,
        \DTIVALysetup,
        log ticks with fixed point,
        minor y tick num=1,
        axis x line*=bottom,
        axis y line*=left,
        legend style={legend pos=north east,inner sep=0pt,outer sep=5pt,legend cell align=left,align=left,draw=none,fill=none,font=\scriptsize}
        ]

        \addplot [color=Set2-A, line width=1pt]
            table[x=\DTIVALx,y=function_value]{dti_non\DTIVALmult.txt};
        \addlegendentry{\ref{item:dti:reference}};


        \addplot [color=Set2-B, line width=2pt, loosely dashed]
            table[x=\DTIVALx,y=function_value]{dti_two\DTIVALmult_wy.txt};
        \addlegendentry{\ref{item:dti:twoblock}};


        \addplot [color=Set2-C, line width=1pt]
            table[x=\DTIVALx,y=function_value]{dti_vxd\DTIVALmult_wy2.txt};
        \addlegendentry{\ref{item:dti:voxelwise-dual}};

        \addplot [color=Set2-C, line width=1pt, dotted, forget plot]
            table[x=\DTIVALx,y=function_value]{dti_vxd\DTIVALmult_accbasic0.5_wy2.txt};

        \addplot [color=Set2-D, line width=1pt]
            table[x=\DTIVALx,y=function_value]{dti_vxp\DTIVALmult_vartau.txt};
        \addlegendentry{\ref{item:dti:voxelwise-primaldual}};

        \addplot [color=Set2-D, line width=1pt, dotted, forget plot]
            table[x=\DTIVALx,y=function_value]{dti_vxp\DTIVALmult_accbasic0.5_vartau.txt};
    \end{axis}

\end{tikzpicture}
        \caption{Comparison of the algorithm variants \ref{item:dti:reference}--\ref{item:dti:voxelwise-primaldual}. The dotted lines show the effect of accelerating the dual blocks in \ref{item:dti:voxelwise-dual} and \ref{item:dti:voxelwise-primaldual} following \cref{thm:acc2-full-dual}.}
        \label{fig:dtiperf}
    \end{subfigure}
    \caption{Reference algorithm step length justification (\subref{fig:dtitau}) and algorithm performance (\subref{fig:dtiperf}) on the DTI problem. Function values are on the vertical axis, and iteration counts are on the horizontal axis. Based on (\subref{fig:dtitau}), we take $\tau=1/R$ in (\subref{fig:dtiperf}): $\tau=5/R$ appears to have convergence issues and $\tau=0.5/R$ yields slower convergence.}
    \label{fig:dtiresults}
\end{figure}
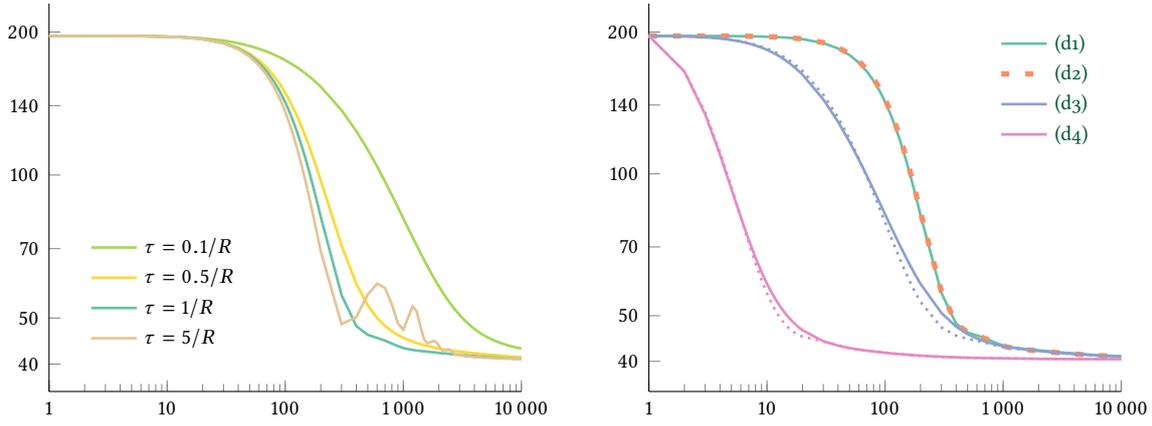

We only consider deterministic updates. We develop step length rules for \cref{alg:acc-full-dual-omega} based on \cref{thm:acc2-full-dual}, however, although $F_{\lambda}^*$ is strongly convex, and the Moreau--Yosida regularisation makes also $F_{\mu,\gamma}^*$ strongly convex, we generally do \emph{not} employ acceleration and instead keep the step length parameters fixed throughout the iterations. Therefore the theorem does not generally provide any convergence claims.

For convenience, we will identify the linear primal indices $j$ and dual indices $\ell$ (used for arbitrary blocks) with symbolic indices corresponding to the different variables $x$, $\mu$, $\lambda$ and their sub-blocks (used for specific blocks).
The primal variable will be just a single block ``$x$'', or be divided into voxelwise blocks ``$x_\xi$'' for $\xi \in \Omega_d$.
The dual variable will consist of just a single block ``$y$'', the two blocks corresponding to the variables ``$\mu$'' and ``$\lambda$'', or ``$\mu$'' and the sub-blocks ``$\lambda_{k,\xi}$'' over $k=1,\ldots,N$ and $\xi \in \Omega_d$.

Of the conditions of \cref{thm:acc2-full-dual}, we will not seek to satisfy the boundedness \eqref{eq:locality-acc2-full-dual}; following \cref{rem:boundedness-assumptions} this seems likely to hold if we initialise close enough to a solution and take the primal step length parameters $\iset \tau^0_j$ small enough.
However, we do not know, how small and how close would be theoretically required.
Likewise, \eqref{eq:initialisation-acc2-full-dual-tau}, which with deterministic updates simplifies to $\delta \ge \iset \tau_j^0\bar L$, is satisfied by taking $\iset \tau_j^0$ small enough. To do this exactly, we would need to calculate the constant $L$ that satisfies the Lipschitz requirement of \cref{ass:k-lipschitz}.
\Cref{ass:gf} readily holds (with Moreau--Yosida regularisation, as discussed above) with $\gamma_{G,x}=0$ and any $0 \le \gamma_{F^*,\mu} \le \gamma\inv\alpha$ and $0 \le \gamma_{F^*,\lambda} \le 1$. We take the latter as well as $\alpha_y$ and $\zeta_\ell$ such that \eqref{eq:bargamma-fulldual} yields $\bar\gamma_{F^*,\ell} \equiv 0$ for all $\ell$.
\Cref{ass:k-nonlinear} we do not hope to verify in the confines of the present manuscript. With \eqref{eq:initialisation-acc2-full-dual-tau} out of the way, for the calculation of the step lengths, it would only be needed for the constants $\gamma_{K,j}$. We simply make the reasonable assumption that we start close enough to a local minimiser satisfying the “second-order necessary condition” $\gamma_{G,j}+\gamma_{K,j}\ge 0$, i.e., $\gamma_{K,j} \ge 0$.
Then we may simply assume $\gamma_{K,j}=0$ and are justified in taking $\tilde\gamma_{G,j}=0$.

It remains to satisfy the relationship \eqref{eq:initialisation-acc2-full-dual-sigmatest} between the primal and dual step lengths.
Taking the weights $w_{j,\ell,k}=w_{j,\ell,k}^i$ and the set of connections $\this{\bar\Neigh_j}(\ell)={\bar\Neigh_j}(\ell)$ given in \eqref{eq:fulldual:w-r:v} independent of the iteration and inserting $w_{j,k}$ from \eqref{eq:fulldual:w-r:w} into \eqref{eq:initialisation-acc2-full-dual-sigmatest}, the latter holds if
\begin{equation}
	\label{eq:dti:sigmarule0}
	1-\kappa
		\ge
		\adaptNorm{
			\sum_{j=1}^m \sqrt{\textstyle \dset\sigma_\ell^0\iset\tau^0_j \chi_{\this\Neigh_j}(\ell) \sum_{\ell' \in \this{\bar\Neigh_j}(\ell)} w_{j,\ell,\ell'} }
			Q_\ell \grad K(\thisx) P_j
		  }^2.
\end{equation}

In particular, with just a single primal block $x$, we then satisfy \eqref{eq:dti:sigmarule0} by taking
\begin{equation}
	\label{eq:dti:sigmarule}
	\dset\sigma^0_{\ell}
	=\frac{1-\kappa}{\iset\tau_x^0 \sum_{\ell' \in {\bar\Neigh_j}(\ell)} w_{x,\ell,\ell'} R_{\ell}^2}
	\quad\text{where we need the estimate}\quad
	R_{\ell}
	\ge \norm{Q_\ell \grad K(\thisx)}.
\end{equation}
Similarly to \cite{chambolle2004algorithm} we estimate
$\norm{\symD_d} \le R_\symD \defeq \sqrt{12}$.
Assuming that each $\this x(\xi)$ for $\xi \in \Omega_d$ is positive semi-definite, we also estimate with $r_{k,\xi} \defeq \abs{s_0(\xi)}\norm{b_k}_2^2$ that
\[
    \norm{\grad T(x^i)} \le R_T \defeq \sqrt{\sum_{k=1}^N \sum_{\xi \in \Omega_d} r_{k,\xi}^2}
    \quad\text{and}\quad
    \norm{\grad K(x^i)} \le R \defeq \sqrt{R_\symD^2 + R_T^2}.
\]
We obtain $R_\ell$ for \eqref{eq:dti:sigmarule} from the same constituents $r_{k,\xi}$ and $R_\symD$, depending on the exact block structure.

It then remains to choose the primal step lengths and the weights $w_{j,k,\ell}$.
We consider the following four block structures and choices of weights:

\begin{enumerate}[label=(d\arabic*)]
    \item\label{item:dti:reference}
    As our reference case, corresponding to earlier non-block-adapted works \cite{tuomov-nlpdhgm,tuomov-nlpdhgm-redo}, a single primal block $x$ ($m=1$) and a single dual block $y$ ($n=1$). Based on the rough optimisation of the step length parameters illustrated in \cref{fig:dtitau}, for a range of $\tau=\iset\tau^0_x$ with $\dset\sigma_y^0 = \sigma \defeq (1-\kappa)/(\tau R^2)$ with $\kappa=0.05$,
    we take $\tau \defeq 1/R$.

    \item\label{item:dti:twoblock}
    A single primal block $x$ ($m=1$) and the two dual blocks $\mu$ and $\lambda$ ($n=2$).
    We take $\tau=\iset\tau^0_1$ as in \ref{item:dti:reference} and with $w_{x,\lambda,\mu} \defeq R_\symD/(R-R_\symD)$ calculate from \eqref{eq:dti:sigmarule} the dual step length parameters as $\dset\sigma_\mu^0 = (1-\kappa)/(\tau(1+\inv w_{x,\lambda,\mu})R_\symD^2)$ and $\dset\sigma_\lambda^0 = (1-\kappa)/(\tau(1+w_{x,\lambda,\mu})R_T^2)$. Thus $\dset\sigma_\mu^0 R_\symD$ equals $\sigma R$ of \ref{item:dti:reference}.

    \item\label{item:dti:voxelwise-dual}
	A single primal block $x$ ($m=1$) and in addition to the dual block $\mu$, we split $\lambda$ into voxelwise and $b_k$-wise blocks $\lambda_{k,\xi}$ ($n=1+N n_1n_2n_3$) indexed by $k=1,\ldots,N$ and $\xi \in \Omega_d$. We still take $\tau=\iset\tau^0_1$ as in \ref{item:dti:reference} and with $w_{x,\lambda_{(k,\xi)},\mu} \defeq \sum_{k',\xi'} r_{k',\xi'} R_\symD/((R-R_\symD)r_{k,\xi})$ and $w_{x,\lambda_{(k,\xi)},\lambda_{(k',\xi')}} \equiv 1$ calculate from \eqref{eq:dti:sigmarule} the dual step length parameters as $\dset\sigma_\mu^0 \defeq (1-\kappa)/(\tau(1 + \sum_{k,\xi} \inv w_{x,\lambda_{(k,\xi)},\mu} )R_\symD^2)$ and $\dset\sigma_{\lambda_{k,\xi}}^0 \defeq (1-\kappa)/(\tau(N + w_{x,\lambda_{(k,\xi)},\mu})r_{k,\xi}^2)$.
    This also keeps $\dset\sigma_\mu^0 R_\symD$ equal to $\sigma R$ of \ref{item:dti:reference}.

    \item\label{item:dti:voxelwise-primaldual}
    Voxelwise primal blocks $x_\xi$ for $\xi \in \Omega$ ($n=n_1n_2n_3$) in addition to dual blocks as in \ref{item:dti:voxelwise-dual}.
	We take the blockwise primal step length parameters $\iset\tau_\xi^0 = \tau_\xi \defeq R\tau/(1 + N \max_{k=1,\ldots,N} r_{k,\xi})$ for $\xi \in \Omega_d$, where $\tau$ is as in \ref{item:dti:reference}. Then we take $w_{x_\xi,\lambda_{(k,\xi)},\mu} \defeq r_{k,\xi}$ and $w_{x_\xi,\lambda_{(k,\xi)},\lambda_{(k',\xi')}}=1$. Observe that according to the definition of the connection set $\bar\Neigh_j(\ell)$ in \eqref{eq:fulldual:w-r:v} that the dual block $(k,\xi)$ is not connected by $K$ to $(k',\xi')$ for $\xi' \ne \xi$.
	Therefore, we satisfy \eqref{eq:dti:sigmarule0} by taking $\dset\sigma_\mu^0 = (1-\kappa)/(\max_{\xi \in \Omega_d} \tau_\xi(1+\sum_{k=1}^N r_{k, \xi}) R_\symD^2)$ and $\dset\sigma_{\lambda_{k,\xi}}^0 = (1-\kappa)/(\tau_{\xi}(N + \inv r_{k,\xi})r_{k,\xi}^2)$.
	The maximum comes from estimating the norm in  \eqref{eq:dti:sigmarule0}.
\end{enumerate}

We report in \cref{fig:dtiperf} for the first 10000 iterations the function value achieved by each algorithm variant. For \ref{item:dti:voxelwise-dual} and \ref{item:dti:voxelwise-primaldual} we also display the effect of the $O(1/N)$ acceleration of \cref{thm:acc2-full-dual}; on \ref{item:dti:reference} and \ref{item:dti:twoblock} this has no notable effect.

On a mid-2014 MacBook Pro with a 2.8GHz Intel Core i5 processor and 16GB RAM running Julia 1.1.0, each iteration of \ref{item:dti:reference}--\ref{item:dti:voxelwise-dual} takes roughly 0.048 seconds. For \ref{item:dti:voxelwise-primaldual} this is roughly 0.062 seconds due to a more complicated primal update.\footnote{In the Julia code \cite{nlpdhgm-block-code}, we update $x^{i+1}(\xi) \defeq x^i(\xi) - \tau_\xi \dir x^i(\xi)$ and $\lambda^{i+1}(k, \xi) \defeq (\lambda^i(k, \xi) + \sigma_{k,\xi} \dir \lambda^i(k, \xi))/(1+\sigma_{k,\xi})$ for some temporary $\dir x^i$ and $\dir \lambda^i$ and all $\xi \in \Omega_d$ and $k=1,\ldots,N$.
The latter does not appear to cause a notable performance penalty compared to a spatially constant $\sigma$ while the former does. However, each $x^{i+1}(\xi)$ is a tensor consisting of multiple floating point numbers while $\lambda^{i+1}(k, \xi)$ is a single floating point number. Our guess is that, due to uneven memory indexing when $\tau$ is spatially varying, the tensor update cannot make as good use of processor SIMD instructions.}
However, in terms of computational times, \ref{item:dti:voxelwise-primaldual} is clearly much faster than the other variants:
0.77s against 14.7--19.2s for \ref{item:dti:reference} and 13.6--18.1s for \ref{item:dti:twoblock} and \ref{item:dti:voxelwise-dual} to reach function value 50. The time ranges account for us sampling the function values only every 100 iterations after the first 100.
The visual character of the approximate solution provided by \ref{item:dti:voxelwise-primaldual} is on closer inspection slightly smoothed out compared to the other variants. This may be due to non-optimal $\alpha$ in the model \eqref{eq:dti-model} or due to a different local solution.


\subsection{Electrical impedance tomography}
\label{sec:eit}

In this problem, we want to solve
\begin{equation}
    \label{eq:eit:objective}
    \min_{x \in V} \sum_{k=1}^N \frac{1}{2}\norm{A_k(x)}^2 + \alpha\norm{\grad x}_{2,1}
\end{equation}
on a finite-dimensional subspace $V \subset L^2(\Omega)$ with $\Omega \subset \R^2$ and each $A_k: V \to \R^{N}$ a non-linear operator corresponding to the fit of the solution of a partial differential equation controlled by $x$ to measured data. We specifically use the complete electrode model of EIT \cite{vauhkonen1999threedimensional}. Our implementation of the model will be described in detail in \cite{jauhiainen2019gaussnewton}.
The rough idea is that $N$ \term{electrodes} are placed on the boundary of the domain $\Omega$ inside which we want to reconstruct an unknown conductivity $x$; see \cref{fig:eitimages}, which presents a synthetic 2D slice model of an object in a cylindrical water tank. As our data, we only have $N$ boundary measurements corresponding to exciting in turn each of the electrodes $k=1,\ldots,N$ with a positive electric potential. In each of these excitations, the remaining electrodes are grounded, and the electric current generated by these excitations is measured at each electrode, yielding $N$ measurements. The operators $A_k$ correspond to each such excitation setup.
In the example of \cref{fig:eitimages}, the number of electrodes $N=16$.

We can again write this problem in the form \eqref{eq:main-problem} with
\[
    G(x) \defeq 0,\quad
    K(x) \defeq (\grad x, A_1(x), \ldots, A_N(x)),
    \quad\text{and}\quad
    F^*(y)=\delta_{\alpha \B}(\mu) + \sum_{k=1}^N \norm{\lambda_k}_2^2,
\]
where $y=(\mu, \lambda_1, \ldots, \lambda_N)$ and $\B$ is the product of the pointwise Euclidean unit balls of $\R^2$ over $\Omega$.

\begin{figure}
    \centering
    \begin{subfigure}[t]{0.33\textwidth}%
        \includegraphics[width=0.95\textwidth]{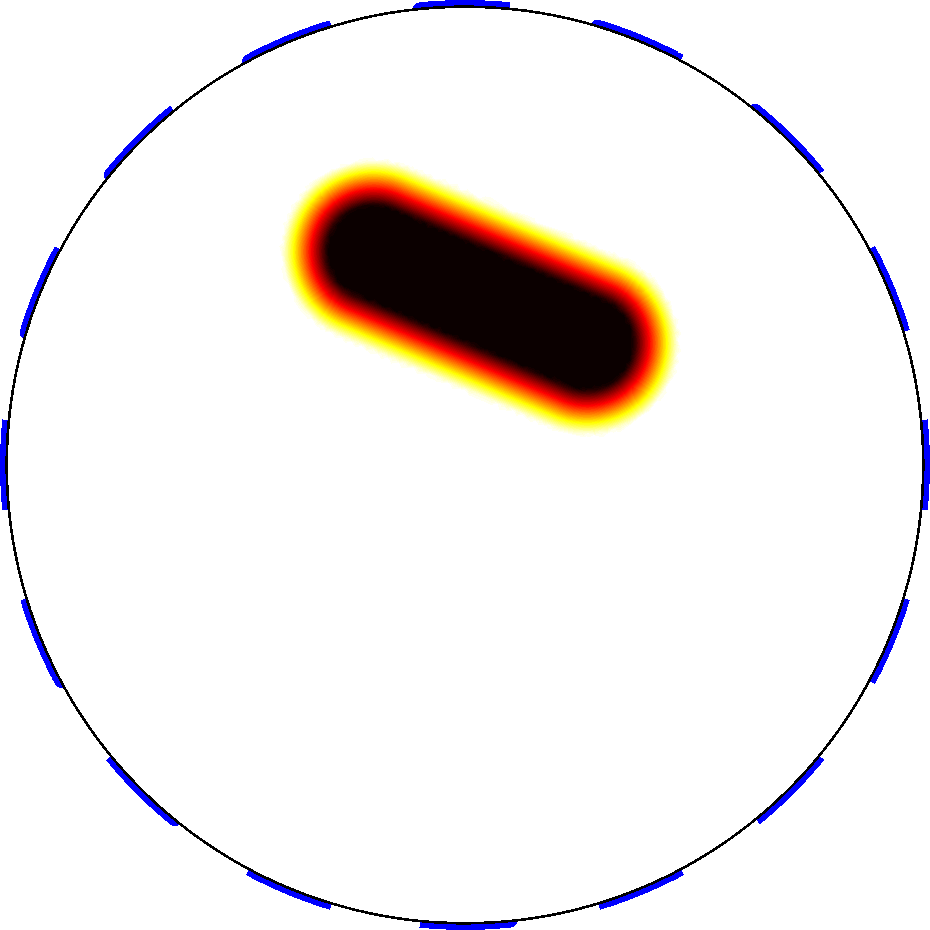}%
        \caption{Synthetic conductivity}%
        \label{fig:eithsynth}
    \end{subfigure}%
    \begin{subfigure}[t]{0.33\textwidth}%
        \includegraphics[width=0.95\textwidth]{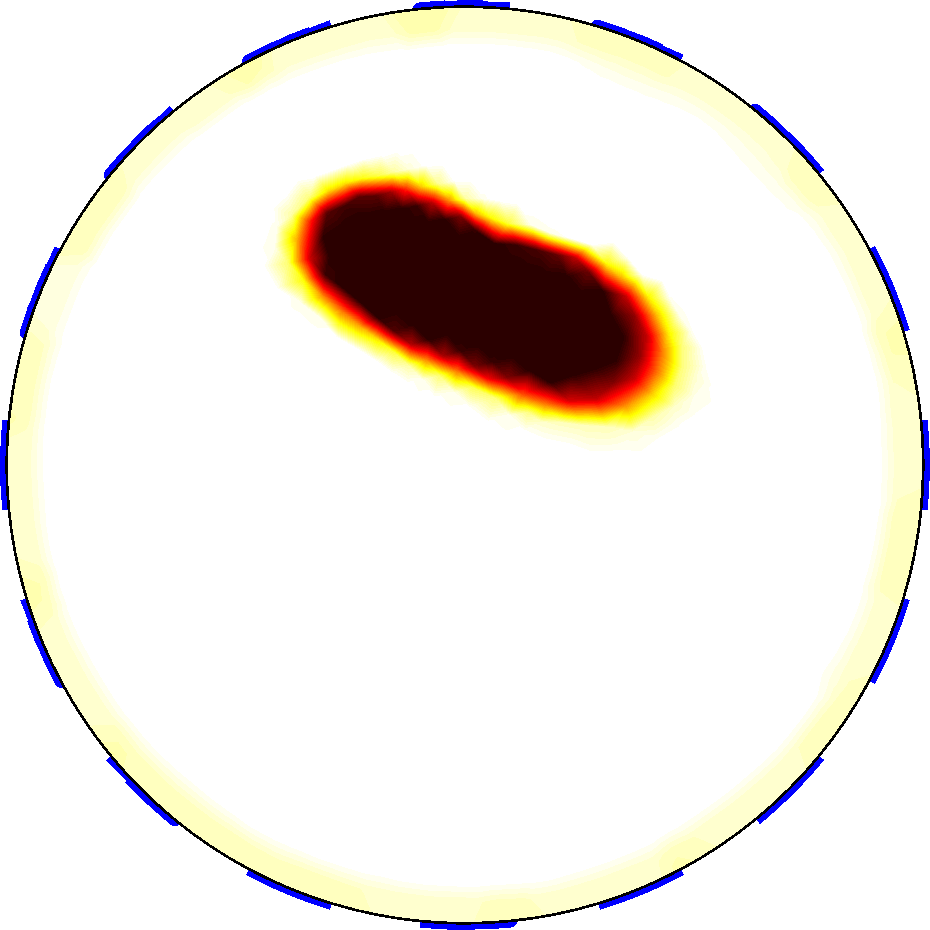}%
        \caption{Reconstructed conductivity}%
        \label{fig:eitreco}
    \end{subfigure}%
    \begin{subfigure}[t]{0.33\textwidth}%
		\ifprint%
        	\includegraphics[width=0.95\textwidth]{img/nlpdhgm_block/eit_InversionMesh.pdf}%
		\else%
	        \includegraphics[width=0.95\textwidth]{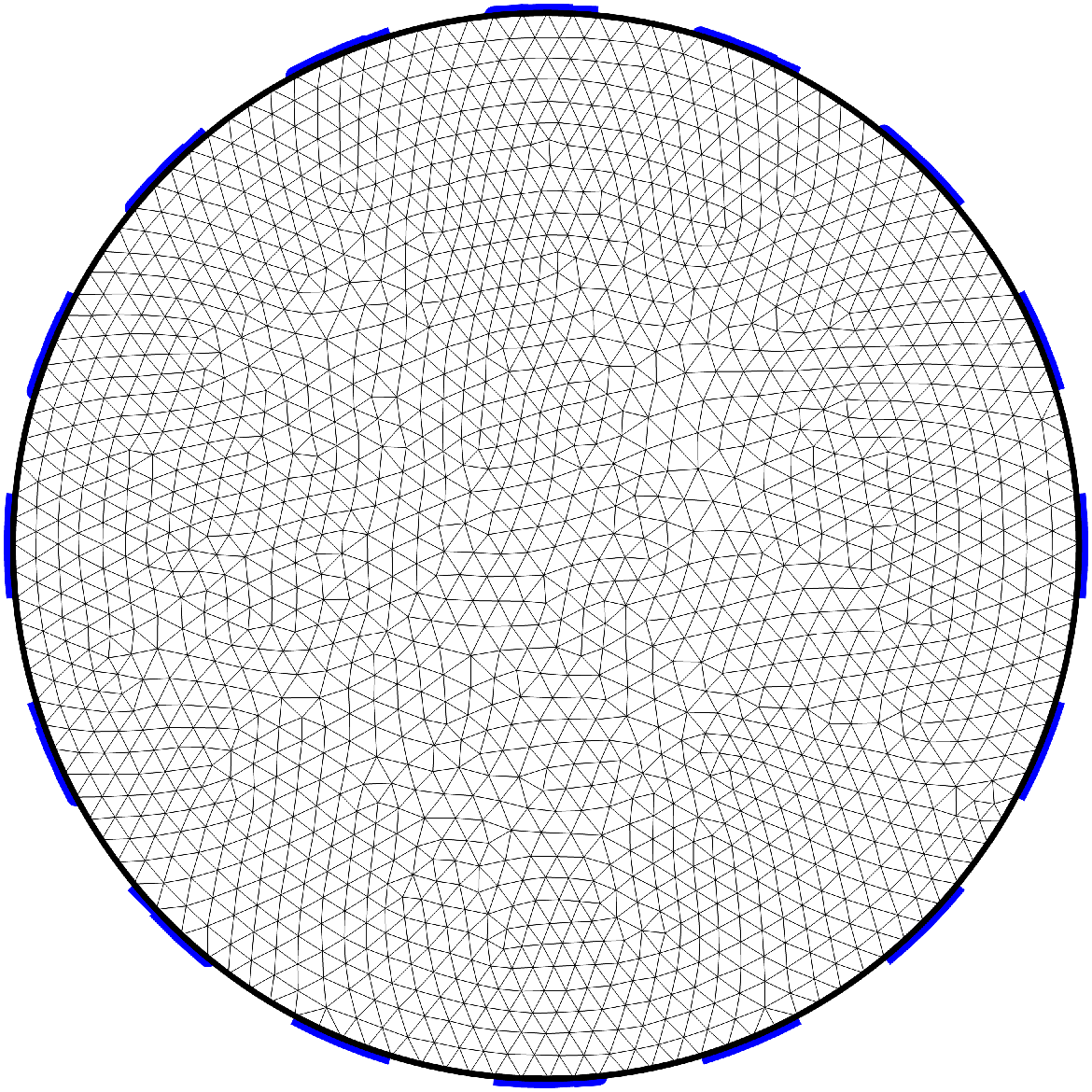}%
		\fi%
        \caption{Finite element mesh}%
        \label{fig:eitmesh}%
    \end{subfigure}%
    \caption{Synthetic true conductivity and reconstructed conductivity for the EIT example. The reconstruction is the one obtained with the block structure and dual step length setup of \ref{item:eit:pdewise-dual-meas} with $\tau=500/R$ after 15000 iterations. The blue patches on the boundary of the domain indicate the electrodes. We display in (\subref{fig:eitmesh}) the finite element mesh used to represent the conductivity.}
    \label{fig:eitimages}
\end{figure}

As a first case of the dual blocks, we take $y_0$ corresponding to the total variation term, and the full measurement vectors $y_k$ corresponding to each excitation $k=1,\ldots,N$.
We estimate $\norm{\grad} \le R_\grad$ for $R_\grad$ being the largest singular value of $\grad$ on $V$.
We do not have exact estimates on the norm of $\grad A_k(\thisx)$.
Therefore, we take a dynamic norm estimate $r_k=r_k(i)$ over the last 100 iterations,
\[
    \norm{\grad A_k(x^i)} \le r_k \defeq 1.05 \max_{\max\{i-99, 0\} \le \iota \le i} \norm{\grad A_k(x^\iota)}
    \quad (k=1,\ldots,N).
\]
We may then estimate $\norm{\grad K(x^i)} \le R \defeq \sqrt{R_\grad^2 + r_1^2 + \cdots + r_N^2}$.
As a second case, we further split each $y_k$ into sub-blocks $y_{k,j} \in \R$ corresponding to each individual electrode $j=1,\ldots,N$ being measured.
We then take norm estimates $r_{k,j}=r_{k,j}(i)$ over the last 100 iterations,
\[
    \abs{[\grad A_k(x^i)]_j} \le r_{k,j} \defeq 1.05 \max_{\max\{i-99, 0\} \le \iota \le i} \abs{[\grad A_k(x^\iota)]_j}
    \quad (k,j=1,\ldots,N).
\]

We work in the setting of \cref{sec:fullprimal}.
Note that unlike \cref{alg:acc-full-dual-omega} in the DTI experiments of \cref{sec:dti}, \cref{alg:acc-full-primal} allows partial calculation of $K$ in both the primal and dual updates, which should in principle be beneficial in stochastic methods.
We develop step length rules for \cref{alg:acc-full-primal} based on \cref{thm:acc2-full-primal}. Similarly to \eqref{eq:dti:sigmarule}, with $w_{j,\ell,k}=w_{j,\ell,k}^i$ and $\this{\bar\Neigh_j}(\ell)={\bar\Neigh_j}(\ell)$ independent of the iteration, for non-stochastic methods with a single primal block $x$, \eqref{eq:initialisation-acc2-full-primal-sigmatest} in particular holds by taking
\begin{equation}
	\label{eq:eit:sigmarule}
	\iset\sigma^1_\ell
	=\frac{1-\kappa}{\dset\tau_x^0 \sum_{\ell' \in {\bar\Neigh_j}(\ell)} w_{x,\ell,\ell'} R_{\ell}^2}
	\quad\text{where we estimate}\quad
	R_{\ell} \ge \norm{Q_\ell \grad K(\thisx)}.
\end{equation}
Again, for convenience, we identify the linear primal indices $j$ and dual indices $\ell$ and $\ell'$ with symbolic indices $x$, $\mu$, and $\lambda_k$.
It then remains to choose $\dset\tau^0_x$ and the weights $w_{x,\ell,\ell'}$. For this we consider four different block and weight setups:

\begin{enumerate}[label=(e\arabic*)]
    \item\label{item:eit:reference}
    Again, as our reference case, corresponding to earlier non-block-adapted works \cite{tuomov-nlpdhgm,tuomov-nlpdhgm-redo}, a single primal block $x$ ($m=1$) and a single dual block $y$ ($n=1$). Based on rough optimisation of the step length parameters, illustrated in \cref{fig:eitbasetau} for a range of $\tau=\dset \tau_x^0$ with $\iset \sigma_y^1 =(1-\kappa)/(\tau R^2)$ with $\kappa=0.05$, we take $\tau \defeq 5/R$ for $R$ computed using just the initial iterate $x^0$ as explained above.

    \item\label{item:eit:pdewise-dual}
	A single primal block $x$ ($m=1$) and the dual blocks $\mu, \lambda_1, \ldots, \lambda_N$.
	We take $\tau=\dset \tau_x^0$ as in \ref{item:eit:reference} and with $w_{x,\lambda_p,\mu} \defeq \sum_{k} r_{k} R_\grad/((R-R_\grad)r_{p})$ and $w_{x,\lambda_p,\lambda_k} \defeq 1$ for $p,k=1,\ldots,N$, solve from \eqref{eq:eit:sigmarule} that $\iset\sigma_\mu^1 \defeq (1-\kappa)/(\tau(1 + \sum_{k} \inv w_{x,\lambda_k,\mu} )R_\grad^2)$ and $\iset\sigma_{\lambda_p}^1 \defeq (1-\kappa)/(\tau(N + w_{x,\lambda_p,\mu})r_{p}^2)$ for $p=1,\ldots,N$.
	This case and the step length rules are analogous to \ref{item:dti:voxelwise-dual} for DTI.

	\item\label{item:eit:pdewise-dual-meas}
    As \ref{item:eit:pdewise-dual} but split each $\lambda_p$ into further measurement-wise dual blocks $y_{p,j}$ ($p,j=1,\ldots,N$), replacing in the expressions of \ref{item:eit:pdewise-dual} the indices $p$ and $k$ by $(p, j)$ and $(k, j')$ with $j, j' \in \{1,\ldots,N\}$. Thus $r_k$ becomes $r_{k,j'}$, etc.

	\item\label{item:eit:pdewise-dual-meas-simple}
	Measurement-wise dual blocks as in \cref{item:eit:pdewise-dual-meas} but $w_{x,\lambda_{(p, j)},\mu} \defeq \inv r_{p,j}$.
\end{enumerate}

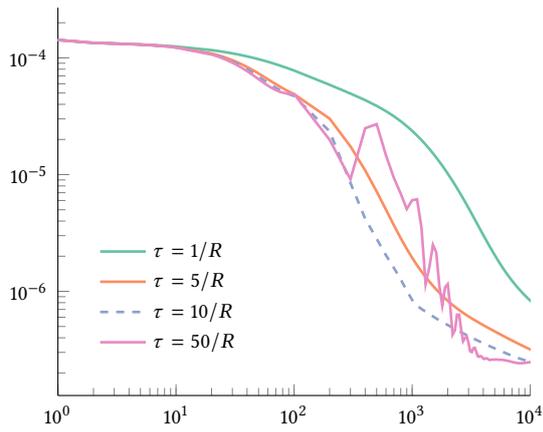
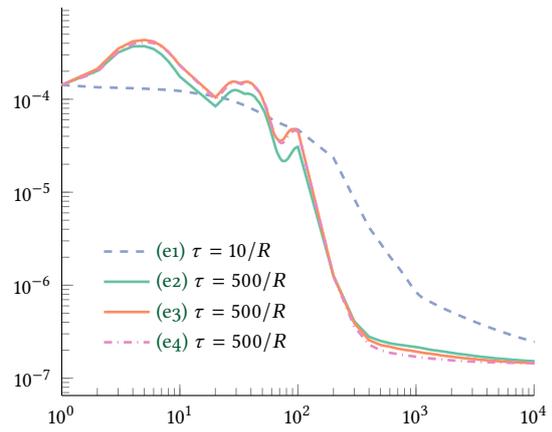
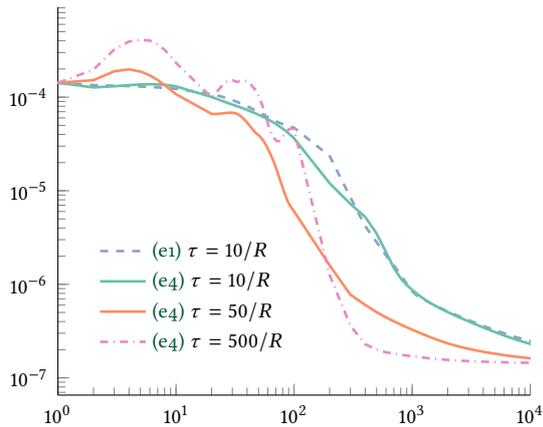
\begin{figure}[t!]
    \centering
    \begin{subfigure}[t]{0.495\textwidth}
        \begin{tikzpicture}
    \begin{axis}[%
        width=\linewidth,
        xmode=log,
        xmin=1,xmax=10000,
        scaled x ticks=false,
        tick label style={/pgf/number format/fixed, /pgf/number format/set thousands separator={\,}},
        xminorticks=true,
        minor x tick num=1,
        ymode=log,
        yminorticks=true,
        minor y tick num=1,
        axis x line*=bottom,
        axis y line*=left,
        legend style={legend pos=south west,inner sep=0pt,outer sep=10pt,legend cell align=left,align=left,draw=none,fill=none,font=\scriptsize}
        ]

        \addplot [color=Set2-A, line width=1pt]
            table[x=iter,y=type1]{Capsule_NLPDPS_wktest_tau1_filt.txt};
        \addlegendentry{$\tau=1/R$};

        \addplot [color=Set2-B, line width=1pt]
            table[x=iter,y=type1]{Capsule_NLPDPS_wktest_tau5_filt.txt};
        \addlegendentry{$\tau=5/R$};

        \addplot [color=Set2-C, dashed, line width=1pt]
            table[x=iter,y=type1]{Capsule_NLPDPS_wktest_tau10_filt.txt};
        \addlegendentry{$\tau=10/R$};

        \addplot [color=Set2-D, line width=1pt]
            table[x=iter,y=type1]{Capsule_NLPDPS_wktest_tau50_filt.txt};
        \addlegendentry{$\tau=50/R$};

    \end{axis}
\end{tikzpicture}
        \caption{Reference algorithm \ref{item:eit:reference}, multiple step lengths}
        \label{fig:eitbasetau}
    \end{subfigure}
    \begin{subfigure}[t]{0.495\textwidth}
        \begin{tikzpicture}
    \begin{axis}[%
        width=\linewidth,
        xmode=log,
        xmin=1,xmax=10000,
        scaled x ticks=false,
        tick label style={/pgf/number format/fixed, /pgf/number format/set thousands separator={\,}},
        xminorticks=true,
        minor x tick num=1,
        ymode=log,
        yminorticks=true,
        minor y tick num=1,
        axis x line*=bottom,
        axis y line*=left,
        legend style={legend pos=south west,inner sep=0pt,outer sep=10pt,legend cell align=left,align=left,draw=none,fill=none,font=\scriptsize}
        ]

        \addplot [color=Set2-C, dashed, line width=1pt]
            table[x=iter,y=type1]{Capsule_NLPDPS_wktest_tau10_filt.txt};
        \addlegendentry{\ref{item:eit:reference} $\tau=10/R$};

        \addplot [color=Set2-A, line width=1pt]
            table[x=iter,y=type2mult1]{Capsule_NLPDPS_wktest_tau250_filt.txt};
        \addlegendentry{\ref{item:eit:pdewise-dual} $\tau=500/R$};  

        \addplot [color=Set2-B, line width=1pt]
            table[x=iter,y=type4mult1]{Capsule_NLPDPS_wktest_tau500_filt.txt};
        \addlegendentry{\ref{item:eit:pdewise-dual-meas} $\tau=500/R$};  

        \addplot [color=Set2-D, dashdotted, line width=1pt]
            table[x=iter,y=type5mult1]{Capsule_NLPDPS_wktest_tau500_filt.txt};
        \addlegendentry{\ref{item:eit:pdewise-dual-meas-simple} $\tau=500/R$};  

    \end{axis}

\end{tikzpicture}
        \caption{Comparison of algorithm variants \ref{item:eit:reference}--\ref{item:eit:pdewise-dual-meas-simple}}
        \label{fig:eitwtype}
	\end{subfigure}
	
	\medskip

    \begin{subfigure}[b]{0.495\textwidth}
        \begin{tikzpicture}
    \begin{axis}[%
        width=\linewidth,
        xmode=log,
        xmin=1,xmax=10000,
        scaled x ticks=false,
        tick label style={/pgf/number format/fixed, /pgf/number format/set thousands separator={\,}},
        xminorticks=true,
        minor x tick num=1,
        ymode=log,
        yminorticks=true,
        minor y tick num=1,
        axis x line*=bottom,
        axis y line*=left,
        legend style={legend pos=south west,inner sep=0pt,outer sep=10pt,legend cell align=left,align=left,draw=none,fill=none,font=\scriptsize}
        ]

        \addplot [color=Set2-C, dashed, line width=1pt]
            table[x=iter,y=type1]{Capsule_NLPDPS_wktest_tau10_filt.txt};
        \addlegendentry{\ref{item:eit:reference} $\tau=10/R$}

        \addplot [color=Set2-A, line width=1pt]
            table[x=iter,y=type5mult1]{Capsule_NLPDPS_wktest_tau10_filt.txt};
        \addlegendentry{\ref{item:eit:pdewise-dual-meas-simple} $\tau=10/R$};  

        \addplot [color=Set2-B, line width=1pt]
            table[x=iter,y=type5mult1]{Capsule_NLPDPS_wktest_tau50_filt.txt};
        \addlegendentry{\ref{item:eit:pdewise-dual-meas-simple} $\tau=50/R$};  


        \addplot [color=Set2-D, dashdotted, line width=1pt]
            table[x=iter,y=type5mult1]{Capsule_NLPDPS_wktest_tau500_filt.txt};
        \addlegendentry{\ref{item:eit:pdewise-dual-meas-simple} $\tau=500/R$};  

    \end{axis}

\end{tikzpicture}
        \caption{Blocked algorithm \ref{item:eit:pdewise-dual-meas-simple}, multiple step lengths}
        \label{fig:eittau}
    \end{subfigure}
    \begin{minipage}[b]{0.495\textwidth}
    \caption{EIT reconstruction performance: iteration counts are on the $x$ axis and primal objective function values \eqref{eq:eit:objective} are on the $y$ axis. We start with step length justification for the non-blocked reference algorithm \ref{item:eit:reference} in (\subref{fig:eitbasetau}). Based on this we use step length $\tau=10/R$ for the reference algorithm as higher step lengths become unstable. Comparison of the different blocked algorithm variants is given in (\subref{fig:eitwtype}) for $\tau=500/R$: with lower parameters the differences are less noticeable, and with higher parameters insignificant improvement is obtained.
    Based on this, in (\subref{fig:eittau}) we represent the performance of \ref{item:eit:pdewise-dual-meas-simple} for multiple step lengths.}
    \label{fig:eitresults}
    \end{minipage}
   
\end{figure}

The performance of the algorithm variants \ref{item:eit:reference}--\ref{item:eit:pdewise-dual-meas-simple} is depicted in \cref{fig:eitresults}, and a sample reconstruction in \cref{fig:eitreco}. Observe how the block-adapted algorithms allow in practise larger $\tau$ than the reference algorithm without block-adaptation. This has significant performance benefits:
To reach and stay below objective function value in the order $10^{-7}$, \ref{item:eit:pdewise-dual-meas-simple} with $\tau=500/R$ requires 208 iterations while   \ref{item:eit:reference} with $\tau=10/R$ requires 906 iterations. (With $\tau=500/R$ the latter requires 3544 iterations, no longer converging well with high $\tau$.)
We also tested stochastic variants of the algorithms for the EIT problem, updating on each iteration only a random subset of the dual blocks. This did not, however, offer any performance benefits over the block-adapted variants, neither in terms of epoch count (iteration count scaled by the fraction of updated blocks) nor actual computational time. 

\section{Conclusion}
In this paper, we studied block-proximal primal-dual splitting methods for non-convex non-smooth optimisation.
From an abstract starting point---also able to model doubly-stochastic methods---we derived explicit algorithms and step-length bounds for two particular cases: methods with full dual updates and methods with full primal updates.
For both of the cases, we derived rules ensuring local $O(1/N)$, $O(1/N^2)$ and linear rates under varying conditions and choices of the step lengths parameters.

We demonstrated the performance of the methods on practical inverse problems. Based on our experience with both the DTI and EIT examples, the block-adaptation provides significant performance benefits.
Random updates, by contrast, did not offer benefits in our sample problems.
We suspect they might be more beneficial on very large scale problems that do not share work between the blocks, yet the blocks have overlapping information, or where communication delays within a computing cluster become significant.
This may be one of the possible directions for further research on the presented methods and their application.

\section*{\texorpdfstring{\normalsize}{}A data statement for the EPSRC}

The codes and data for the DTI experiments are available at \cite{nlpdhgm-block-code}.
The codes for EIT, based on historical work of several people, cannot be made available at this point.

\appendix

\section{Satisfaction of the three-point condition}
\label{sec:kcond}

The following lemma provides simplified conditions under which \cref{ass:k-nonlinear} holds, e.g., whenever $x \mapsto \iprod{K(x)}{\realopty}$ is block-separable and strongly-convex.

\begin{lemma}
	\label{lemma:three-point-k-explained}
	Suppose \cref{ass:k-lipschitz} holds and the following is true for the given neighbourhood $\neighx_K$ of $\realoptx$, $\Gamma_K=\textstyle\sum_{j=1}^{m}\gamma_{K,j} P_j \in \linear(X; X)$, $\gamma_{K,j}\in \R$, some $\gamma_x>0$:
	\begin{subequations}
		\begin{align}
		\label{eq:three-point-lemma-gammax}
		\iprod{[\kgrad{x'}-\kgrad{\realoptx}]^*\realopty}{x'-\realoptx}&\ge \norm{x'-\realoptx}_{\Gamma_K}^2 + \gamma_x\norm{x'-\realoptx}^2,
		\\
		\label{eq:three-point-lemma-pj}
		\iprod{[P_j\kgrad{x'}-P_j\kgrad{\realoptx}]^*\realopty}{x'_j-\realoptx_j} & \ge \gamma_{K,j}\norm{x'_j-\realoptx_j}^2 \quad (j=1,\ldots,m).
		\end{align}
	\end{subequations}
    Let $\beta_1,\beta_2>0$, $A=\sum_{j=1}^m a_j P_j$, and $\MIN a\defeq \min_j a_j$.
	Then \cref{ass:k-nonlinear} holds for $p=1$ when
    \begin{align*}
        L\theta_A & \le  \MIN a(\gamma_x-\beta_1)-\beta_2\max_j(a_j-\MIN a)
        \quad\text{and}
        \\
        L_3 & \ge L^2\norm{\Pnl \realopty}(\beta_1^{-1}+(\beta_2\MIN a)^{-1}\textstyle\sum_{j=1}^m (a_j-\MIN a))/2+2L\theta_A.
    \end{align*}
\end{lemma}

\begin{proof}
	We need to study \eqref{eq:ass-k-nonlinear}. We have
	\[
		\begin{split}
		R^K
		&\defeq\iprod{[\kgrad{x}-\kgrad{\realoptx}]^*\realopty}{x'-\realoptx}_A-\norm{x'-\realoptx}^2_{A\Gamma_K}
		\\
		&=\MIN a(\iprod{[\kgrad{x}-\kgrad{\realoptx}]^*\realopty}{x'-\realoptx}-\norm{x'-\realoptx}^2_{\Gamma_K})
		\\
		&\quad+
		\textstyle\sum_{j=1}^m (a_j-\MIN a)(\iprod{[\kgrad{x}-\kgrad{\realoptx}]^*\realopty}{x'_j-\realoptx_j}-\gamma_{K,j}\norm{x'_j-\realoptx_j}^2).
		\end{split}
	\]
	We now apply \eqref{eq:three-point-lemma-gammax}, Young's inequality with the factor $\beta_1>0$, and \cref{ass:k-lipschitz} to bound
	\[
		\begin{split}
		\iprod{[\kgrad{x}&-\kgrad{\realoptx}]^*\realopty}{x'-\realoptx}-\norm{x'-\realoptx}^2_{\Gamma_K}
        \\
        &
		=
		\iprod{[\kgrad{x'}-\kgrad{\realoptx}]^*\realopty}{x'-\realoptx}-\norm{x'-\realoptx}^2_{\Gamma_K}
		+\iprod{[\kgrad{x}-\kgrad{x'}]^*\realopty}{x'-\realoptx}
		\\
		&
		\ge(\gamma_x-\beta_1)\norm{x'-\realoptx}^2-L^2\norm{\Pnl \realopty}^2(4\beta_1)^{-1} \norm{x'-x}^2.
		\end{split}
	\]
	Similarly, for any $\beta_2>0$, we have
	\[
		\begin{split}
		\iprod{[\kgrad{x}&-\kgrad{\realoptx}]^*\realopty}{x'_j-\realoptx_j}
        \\
        &
		=
		\iprod{[P_j\kgrad{x'}-P_j\kgrad{\realoptx}]^*\realopty}{x'_j-\realoptx_j}
		+\iprod{[\kgrad{x}-\kgrad{x'}]^*\realopty}{x'_j-\realoptx_j}
		\\
		&\ge
		\gamma_{K,j}\norm{x'_j-\realoptx_j}^2-L^2\norm{\Pnl\realopty}^2(4\beta_2)^{-1}\norm{x'-x}^2-\beta_2\norm{x'_j-\realoptx_j}^2.
		\end{split}
	\]
	Combining the two estimates, we arrive at
	\[
		\begin{split}
		R^K&\ge
		\MIN a(\gamma_x-\beta_1)\norm{x'-\realoptx}^2-\MIN aL^2\norm{\Pnl \realopty}^2(4\beta_1)^{-1} \norm{x'-x}^2
		\\
		&\quad
		-\textstyle\sum_{j=1}^m (a_j-\MIN a)
		(\beta_2\norm{x'_j-\realoptx_j}+ L^2\norm{\Pnl\realopty}^2\norm{x'-x}^2)
		\\
		&=
		\textstyle\sum_{j=1}^m(\MIN a(\gamma_x-\beta_1)-(a_j-\MIN a)\beta_2)\norm{x'_j-\realoptx_j}^2
		\\
		&\quad
		-\MIN aL^2\norm{\Pnl \realopty}(\beta_1^{-1}+(\beta_2\MIN a)^{-1}\textstyle\sum_{j=1}^m (a_j-\MIN a)) \norm{x'-x}^2/4.
		\end{split}
	\]
	At the same time, using \cref{ass:k-lipschitz}, we get for the right-hand side of \eqref{eq:ass-k-nonlinear} the bound
	\[
		\norm{K(\realoptx)-K(x)-\kgrad{x}(\realoptx-x)}
		\le \frac{L}{2}\norm{x-\realoptx}^2
		\le L \norm{x'-\realoptx}^2+L \norm{x'-x}^2.
	\]
	So \cref{ass:k-nonlinear} holds if we take $p=1$, $L\theta_A\le \min_j \MIN a(\gamma_x-\beta_1)-(a_j-\MIN a)\beta_2$, and
	$L_3 \ge L^2\norm{\Pnl \realopty}(\beta_1^{-1}+(\beta_2\MIN a)^{-1}\textstyle\sum_{j=1}^m (a_j-\MIN a))/2+2L\theta_A$.
\end{proof}

\section{Technical lemma}
\label{sec:max}

\begin{lemma}
	\label{lemma:max-for-n2}
    We have $\bar z_N\le\bar z_0 +N/2$ whenever $z^i_j>0$, ($i=1,\ldots,N$; $j=1,\ldots,m$) satisfy
	\begin{equation}
	\label{eq:max-for-n2}
		z_j^{i+1}=\frac{1+z_j^i}{\sqrt{1+\bar z_i^{-1}}}
		\quad\text{with}\quad
		\bar z_i\defeq\max_{j=1,\ldots,m}z^i_j.
	\end{equation}
\end{lemma}
\begin{proof}
	Taking $\max_{j=1\ldots m}$ on both sides of the first part of \eqref{eq:max-for-n2}, we obtain
	\[
		\bar z_{i+1}=(1+\bar z_i)\sqrt{\frac{\bar z_i}{\bar z_i+1}}=\sqrt{\bar z_i^2+\bar z_i}.
	\]
	We thus obtain the claim by telescoping
	\[
		\bar z_{i+1}-\bar z_i=
		\sqrt{\bar z_i^2+\bar z_i}-\bar z_i
		=\frac{\bar z_i}{\sqrt{\bar z_i^2+\bar z_i}+\bar z_i}
		=\frac{1}{\sqrt{1+\bar z_i^{-1}}+1}
		\le \frac{1}{2}.
		\qedhere
	\]
\end{proof}

 \providecommand{\eprint}[1]{\href{http://arxiv.org/abs/#1}{arXiv:#1}}
  \providecommand{\eprint}[1]{\href{http://arxiv.org/abs/#1}{arXiv:#1}}
  \providecommand{\noopsort}[1]{}


\end{document}